\documentclass[11pt]{article}

\usepackage[latin1]{inputenc} 
\usepackage{amsmath}   
\usepackage{amsfonts}  
\usepackage{amssymb}  
\usepackage{latexsym}  
\usepackage{graphicx}
\usepackage{epstopdf}
\usepackage[a4paper]{geometry}
\usepackage{color}

\usepackage{subfigure}
\usepackage{url}

\usepackage{hyperref}

\usepackage{graphicx}

\def\ds{\displaystyle}

\def\hM{\hat{M}}
\def\hL{\hat{L}}
\def\hA{\hat{A}}
\def\hG{\hat{M}}
\def\hH{\hat{H}}

\def\hgamma{\hat{\gamma}}

\def\nb{n_{\bullet}}
\def\nbw{n_{\bullet\circ}}
\def\nwb{n_{\circ\bullet}}
\def\nbb{n_{\bullet\bullet}}

\def\cA{\mathcal{A}}

\def\cL{\mathcal{L}}
\def\chL{\hat{\mathcal{L}}}
\def\cM{\mathcal{M}}
\def\cE{\mathcal{E}}
\def\cH{\mathcal{H}}

\def\cR{\mathcal{R}}
\def\cRg{\cR^{g}}
\def\cS{\mathcal{S}}
\def\cSgd{\cS^{g}_d}
\def\cRgd{\cR^{g}_d}
\def\cU{\mathcal{U}}
\def\cUgd{\cU^{g}_d}
\def\cT{\mathcal{T}}
\def\cTgd{\cT^{g}_d}
\def\cF{\mathcal{F}}
\def\chF{\hat{\mathcal{F}}}

\def\cO{\mathcal{O}}

\def\cV{\mathcal{V}}
\def\cOgd{\cO_d^{g}}
\def\cMgd{\cT_d^{g}}

\def\Minf{M^{\infty}}

\def\tG{\widetilde{G}}
\def\tH{\widetilde{H}}

\def\cHci{\cH^{\circ}}
\def\cHbs{\cH^{{\diamond}}}
\def\cHbu{\cH^{\bullet}}

\def\zZ{\mathbb{Z}}
\def\Pi{P^{(i)}}

\def\oC{\bar{C}}
\def\oX{\overline{X}}
\def\oZ{\overline{Z}}

\def\Nbb{N_{\bullet\bullet}}
\def\Nbw{N_{\bullet\circ}}
\def\Nwb{N_{\circ\bullet}}
\def\Nww{N_{\circ\circ}}

\def\NI{N_{\mathrm{I}}}
\def\NII{N_{\mathrm{II}}}
\def\NIi{N_{\mathrm{I}}^{(a)}}
\def\NIii{N_{\mathrm{I}}^{(b)}}
\def\NIiii{N_{\mathrm{I}}^{(c)}}

\def\Kbbi{K_{\bullet\bullet}^{(i)}}
\def\Kbwi{K_{\bullet\circ}^{(i)}}
\def\Kwwi{K_{\circ\circ}^{(i)}}

\def\eps{\epsilon}

\def\NN{\mathbb{N}}
\def\ZZ{\mathbb{Z}}

\def\nb{n_{\bullet}}
\def\nbw{n_{\bullet\circ}}
\def\nwb{n_{\circ\bullet}}
\def\nbb{n_{\bullet\bullet}}

\def\Vb{\mathcal{V}_{\bullet}}
\def\Vw{\mathcal{V}_{\circ}}
\def\nww{n_{\circ\circ}}
\def\nw{n_{\circ}}

\def\aM{M^{\star}}

\def\bb{\frac{b}{b-1}}

\def\ab{\frac{\alpha}{\beta}}
\def\N{\mathbb{N}}

\def\hG{\hat{G}}
\def\hH{\hat{H}}

\def\ha{\hat{\alpha}}

\def\hZ{\hat{Z}}

\def\hX{\hat{X}}
\def\oZ{\bar{Z}}
\def\oX{\bar{X}}

\newtheorem{theorem}{Theorem}
\newtheorem{proposition}[theorem]{Proposition}

\newtheorem{lemma}[theorem]{Lemma}

\newtheorem{claim}{Claim}

\newtheorem{corollary}[theorem]{Corollary}

\newcommand{\mb}[1]{\mathbb{#1}}

\newcommand{\ccw}{counterclockwise }

\newcommand{\cw}{clockwise }

\newcounter{sclaim}
\newcounter{ssclaim}

\newenvironment{proof}{\noindent \setcounter{ssclaim}{0}\emph{Proof.}\ }{\hfill
    $\Box$\vspace{1em}}

  \newenvironment{proofclaim}{\noindent \emph{Proof of the claim.}\ }{\hfill
    $\Diamond$\vspace{1em}}

{\refstepcounter{sclaim}\vspace{1ex}\noindent{\it  (\arabic{sclaim})  {#1}{}}\it}{\vspace{1ex}}

	{\noindent {}{#1}{}}{ This proves (\arabic{sclaim}).\vspace{1ex}}

\title{A bijection for essentially 3-connected toroidal maps
          \thanks{This work was
      supported by the ANR grant GATO
      ANR-16-CE40-0009-01.}}
\date{\vspace{-5ex}}

\graphicspath{{Figures/}}

\title{Orientations and bijections for toroidal 
maps with prescribed face-degrees and essential girth\thanks{This work was
      supported by the grant EGOS ANR-12-JS02-002-01 and GATO
      ANR-16-CE40-0009-01.}}

\author{\'Eric Fusy\thanks{LIX UMR 7161, \'Ecole Polytechnique, 1 rue Honor\'e d'Estienne d'Orves 91120 Palaiseau, France. \texttt{fusy@lix.polytechnique.fr}} , Benjamin Lévêque\thanks{G-SCOP UMR 5272, Universit\'e Grenoble Alpes, 46 avenue F\'elix Viallet 
38031 Grenoble Cedex 1, France. \texttt{benjamin.leveque@cnrs.fr}}}

\begin{document}
\maketitle


\begin{abstract}
  We present unified bijections for maps on the torus with control on
  the face-degrees and essential girth (girth of the periodic planar
  representation).  A first step is to show that for $d\geq 3$ every
  toroidal $d$-angulation of essential girth $d$ can be endowed with a
  certain `canonical' orientation (formulated as a weight-assignment
  on the half-edges). Using an adaptation of a construction by
  Bernardi and Chapuy, we can then derive a bijection between
  face-rooted toroidal $d$-angulations of essential girth $d$ (with
  the condition that, apart from the root-face contour, no other closed walk of 
length $d$ encloses the root-face) and a family of decorated 
  unicellular maps. The orientations and bijections can then be
  generalized, for any $d\geq 1$, to toroidal face-rooted maps of essential girth $d$
 with a root-face of degree $d$ (and with the same root-face contour condition as 
for $d$-angulations), and they
  take a simpler form in the bipartite case, as a parity specialization. 
  On the enumerative side we obtain explicit algebraic expressions
  for the generating functions of rooted essentially simple
  triangulations and bipartite quadrangulations on the torus.  Our
  bijective constructions can be considered as toroidal counterparts
  of those obtained by Bernardi and the first author in the planar
  case, and they also build on ideas introduced by Despré,
  Gon\c{c}alves and the second author for essentially simple
  triangulations, of imposing a balancedness condition on the
  orientations in genus $1$.
\end{abstract}

\section{Introduction}
\label{sec:introduction}
The enumerative study of (rooted) maps 
 has been a very active research topic since Tutte's seminal results on the 
enumeration of planar maps~\cite{T62a,Tu63}, later extended to higher genus
by Bender and Canfield~\cite{BeCa86}. Tutte's approach is based on so-called loop-equations
for the associated generating functions with a catalytic variable for the root-face 
degree. Powerful methods have been developed to compute the solution of such equations (originally solved by guessing/checking),  
both in the planar case~\cite{GoJa83,BJ05a} and in higher genus~\cite{Eyn}. 

The striking simplicity of counting formulas discovered by Tutte (e.g., the number
of rooted planar simple triangulations with $n+3$ vertices is equal to 
$\frac{2}{n(n+1)}\binom{4n+1}{n-1}$) asked 
for bijective explanations. The first such constructions, bijections from 
maps to certain decorated trees, were introduced by Cori and Vauquelin~\cite{CoriVa}
and Arqu\`es~\cite{Ar86} and later further developed by Schaeffer~\cite{S-these}, who also introduced
with Marcus  
the first bijection (for bipartite quadrangulations) that extends to higher genus~\cite[Chap.6]{S-these}. 
The bijection has been adapted in~\cite{CMS09} to a form better suited for computing  the generating functions, and has been recently extended
 to non-orientable surfaces~\cite{chapuy2017bijection,bettinelli2015bijection}. 

In the planar case many natural  families of maps considered in the literature are 
given by restrictions on the face-degrees and on the girth (length of a shortest
cycle). For instance loopless triangulations are (planar) maps with all face-degrees equal
to $3$ and girth at least $2$. The bijections developed over the years for such families (in particular, simple quadrangulations~\cite[Sect.2.3.3]{S-these}, loopless triangulations~\cite[Sect.2.3.4]{S-these}, 
simple triangulations~\cite{PS03b}, irreducible quadrangulations~\cite{FuPoScL} and triangulations~\cite{Fu07b}) 
shared the feature that each map of the considered family can be endowed with a `canonical' orientation
that is usually specified by outdegree prescriptions
 (so-called $\alpha$-orientations~\cite{Fe03}), which is then exploited to associate
to the map a decorated tree structure. 
For instance simple triangulations with a distinguished outer face can be endowed
with an orientation where all outer vertices have outdegree $1$ and all inner vertices
have outdegree $3$, such orientations being closely related to Schnyder woods~\cite{S90}.  
In recent works~\cite{BF12,AlPo13} the methodology has been given a unified formalism, where
each such bijective construction can be obtained as a specialization of a `meta'-bijection
between certain oriented maps and certain decorated trees, 
which itself is an adaptation of a bijection developed in~\cite{Be05} (and extended
in~\cite{BC11} to higher genus) to count tree-rooted planar maps. A success of this strategy
has been to solve for the first time~\cite{BF12b} 
the problem of counting planar maps with control 
on the face-degrees and on the girth (this has been subsequently recovered in~\cite{BG15} and  extended to the so-called irreducible setting), and to  
 adapt the bijections to hypermaps~\cite{BeFu13} and maps with boundaries~\cite{bernardi2015bijections}.

 Up to now this general strategy based on canonical orientations has
 been mostly applied in the planar case, while the only bijections
 known to extend to any genus $g\geq 0$ deal with maps (or bipartite
 maps) with control on the face-degrees but not on the girth: 
 bijections to labeled mobiles~\cite{BoDiGu04,CMS09,Ch09} or to
 blossoming trees and unicellular maps~\cite{Sc97,Lep18}.  It has
 however recently appeared~\cite{DGL15} that in the case of genus $1$,
 a bijection based on canonical orientations can be designed for
 essentially simple triangulations\footnote{A map $M$ on the torus is
   said to have `essentially' property $P$ if the periodic planar
   representation $M^\infty$ of $M$ has property $P$; thus $M$
   is essentially simple means that $M^{\infty}$ is simple. Similarly
   the \emph{essential girth} of $M$ is defined as the girth of
   $M^{\infty}$.}.  The canonical orientations used in this
 construction are $3$-orientations (all vertices have outdegree $3$)
 with an additional `balancedness' property (every non-contractible
 cycle has the same number of outgoing edges to the left side as to
 the right side), see Figure~\ref{fig:examples_d_ori}(a) for examples.
 The existence of such orientations builds on an earlier work on
 toroidal Schnyder woods~\cite{GL14} (see also~\cite{LevHDR}), and the
   bijection thus obtained can be considered as a toroidal counterpart of the
 one in~\cite{PS03b}.  This strategy has also been recently applied to
 essentially 4-connected triangulations~\cite{BoLe18}, where the
 obtained bijection (based on certain `balanced' transversal
 structures) is now a toroidal counterpart of the one in~\cite{Fu07b}.

\begin{figure}
\begin{center}
\includegraphics[width=13cm]{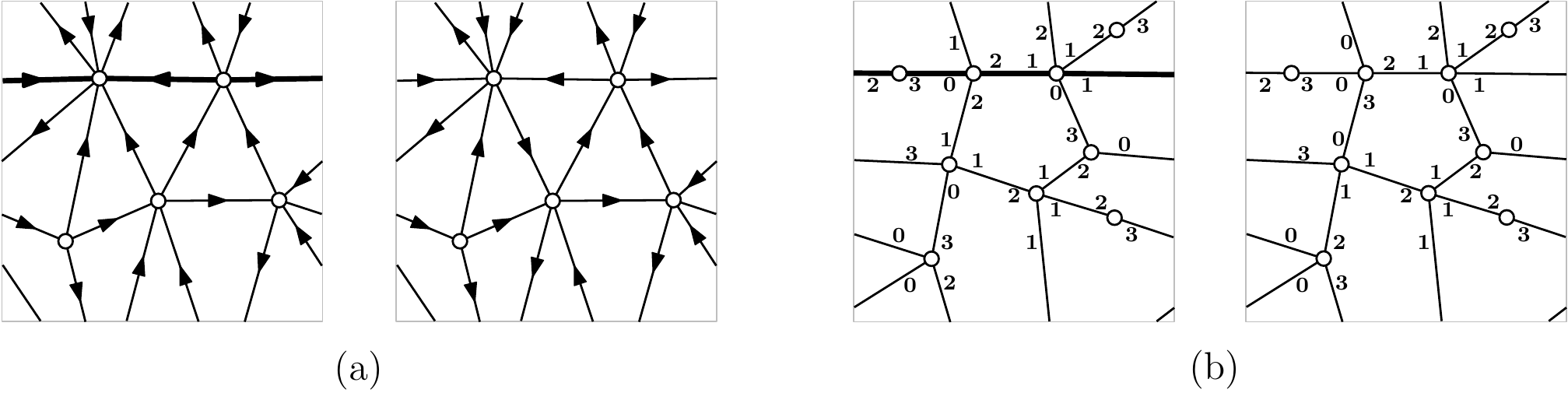}
\end{center}
\caption{(a) Examples of $3$-orientations for a toroidal essentially simple triangulation 
(the first example is not balanced as the bold cycle has outdegree $3$ on the upper side
and outdegree $1$ on the lower side, the second example is balanced). (b) Examples of $\frac{5}{3}$-orientations for a toroidal pentagulation of essential girth $5$ 
(the first example is not balanced as the bold cycle has total weight $4$ on the upper side 
and total weight $2$ on the lower side, the second example is balanced)}
\label{fig:examples_d_ori}
\end{figure} 


\subsubsection*{Main results and outline of the article.} 
In this article, we extend the strategy of~\cite{DGL15} to toroidal maps of prescribed essential  girth and face-degrees, 
thereby obtaining bijections with certain decorated unicellular maps. 
 Our bijections can be seen as toroidal counterparts of those given in~\cite{BF12} 
for planar toroidal 
$d$-angulations of essential girth $d\geq 3$, and in~\cite{BF12b} for 
planar maps with prescribed girth and face-degrees. 

Our first results deal with toroidal $d$-angulations of essential girth $d$, for $d\geq 3$. 
In the planar case it is known~\cite{BF12} that $d$-angulations of girth $d$, with a marked face
considered as the outer face, can be endowed
with certain `weighted biorientations' (given by assigning a weight in $\NN$ to every half-edge) called 
$\frac{d}{d-2}$-orientations, such that for every inner edge (resp. inner vertex) 
the sum of the weights of the incident half-edges is $d-2$ (resp. $d$). Moreover, each $d$-angulation of girth $d$
admits a `canonical' such orientation, called the minimal one. The meta-bijection given in~\cite{BF12}
can then be applied to the minimal  $\frac{d}{d-2}$-orientations, giving a correspondence with well-characterized
decorated trees.  

We will prove that a parallel strategy can be applied in genus
$1$. Precisely, we show in Section~\ref{sec:balanced} that every
toroidal $d$-angulation of essential girth $d\geq 3$ admits a
so-called \emph{balanced $\frac{d}{d-2}$-orientation}, where again
every half-edge is assigned a weight-value in $\NN$ such that the
total weight of each edge (resp. vertex) is $d-2$ (resp. $d$) and
`balanced' means that for every non-contractible cycle $C$, the total
weight of half-edges incident to each side of $C$ is the same, see
Figure~\ref{fig:examples_d_ori} for examples ($d=3$ on the left side, 
$d=5$ on the right side).  Similarly as in the
planar case, when the $d$-angulation has a distinguished face, the map
admits a `canonical' such orientation, called the minimal one.  An
extension of the `meta-bijection' to higher genus (described in
Section~\ref{sec:bijPhi+} and obtained by adapting the construction
of~\cite{BC11}) can then be applied to these orientations, yielding a
bijection, stated in Section~\ref{sec:bijdangul}, between face-rooted
toroidal $d$-angulations of essential girth $d$ (with the extra
condition that apart from the root-face contour, there is no other closed walk of length $d$ that 
encloses the root-face) and a family of well-characterized decorated
unicellular maps of genus $1$.

Similarly as in the planar case~\cite{BF12b}, the strategy can then be extended to face-rooted
toroidal maps of essential girth $d\geq 1$, with root-face degree $d$
(with the same root-face contour conditions as for $d$-angulations). The canonical orientations 
in that case have similar weight conditions, now allowing for half-edges of negative weights,  and
the obtained bijections, stated in Section~\ref{sec:bij_extended}, keep track of the distribution of the face-degrees,
and have a simpler form in the bipartite case (which can be seen as a parity specialization of the general
bijection, as in the planar case~\cite{BF12,BF12b}).

Regarding counting results, we show in Section~\ref{sec:counting} that in certain cases (essentially
simple triangulations and essentially simple bipartite quadrangulations), the generating function of the corresponding mobiles can be computed by a similar approach as in~\cite{CMS09}, and the expressions simplify nicely. 
Unfortunately, for general $d$, even if the corresponding unicellular decorated trees are well-characterized, we have not succeeded in deriving an explicit simple expression of the generating function of rooted toroidal $d$-angulations 
of essential girth $d$, as was done in the planar case~\cite{BF12,BF12b,BG15}.  


\subsubsection*{Higher genus extensions?}  
It is unclear to us if our results could be extended to higher genus. The nice property of the torus
is that the Euler characteristic is zero, which is compatible with orientations having homogeneous 
outdegrees (e.g. for triangulations on the torus there are exactly $3$ times more edges than vertices,
and the orientations exploited to derive a bijection are those with outdegree $3$ at each vertex). 

In higher genus it has been shown in~\cite{albar2016orienting} 
that every simple triangulation has an orientation where every 
vertex-outdegree is a nonzero multiple of $3$, hence all vertices have outdegree $3$ except for $O(g)$ 
special vertices whose outdegree is a multiple of $3$ larger than $3$ 
 (e.g. in genus $2$ all vertices have outdegree $3$ except
for either two vertices of outdegree $6$ or one vertex of outdegree $9$), and the presence of these special  
vertices makes it more difficult to come up with a natural canonical orientation amenable to a bijection.

\section{Preliminaries}
\label{sec:preliminaries}

\subsection{Maps and essential girth for toroidal maps} 
 A \emph{map} $M$ of genus $g$ is an embedding of a connected graph
 (possibly with loops and multiple edges) on
 the orientable surface $\Sigma$ of genus $g$, such that all
 components of $\Sigma\backslash M$ are homeomorphic to open disks; we
 will mostly consider maps of genus $1$, which we call toroidal maps.
 A map is called \emph{rooted} if it has a marked corner, and is
 called \emph{face-rooted} if it has a marked face.  The \emph{dual}
 $M^*$ of $M$ is the map obtained by inserting a vertex in each face
 of $M$, every edge $e\in M$ yielding a dual edge $e^*$ in $M^*$ that
 connects the vertices dual to the faces on each side of $e$.
 A \emph{walk} in $M$ is a (possibly infinite) sequence of edges traversed in a
   given direction, such that the head of an edge in the sequence
   coincides with the tail of the next edge in the sequence (possibly
   two successive edges in the sequence are the same edge traversed in
   opposite directions).  A
 \emph{path} in $M$ is a walk with no repeated vertices.  A
 \emph{closed walk} in $M$ is a finite walk  such
 that the head of the first edge in the sequence coincides with the
 tail of the last edge. We identify two closed walks if they differ by a cyclic shift of the sequence of edges. Hence
a closed walk can be seen as a cyclic sequence of edges such that the head of each edge coincides with the tail of the 
next edge in the sequence. 
A closed walk is called \emph{non-repetitive} if it does not pass twice by a same edge taken in the same direction. A \emph{cycle} is a closed walk with no repeated vertices.

 The \emph{girth} of a map $M$ is the length of a shortest cycle in
 $M$.  The \emph{essential girth} of a toroidal map $M$ is the girth
 of the universal cover $M^{\infty}$ (periodic planar representation). 
As we will see, the essential girth is at least the
 girth.  
A \emph{contractible closed walk} of $M$ (resp. of $M^{\infty}$) is defined
as a non-repetitive closed walk $W$ having a contractible region on its right, which is 
called the \emph{interior} of $W$.   

\begin{lemma}\label{lem:charac_ess}
Let $M$ be a toroidal map. Then the essential girth of $M$ coincides with the length of a shortest
contractible closed walk in $M$.
\end{lemma}
\begin{proof}
Let $d$ be the essential girth of $M$ and let $d'$ be the length of a shortest contractible closed walk in $M$. 
We first make a few observations. Any contractible closed walk $W$ of $\Minf$ yields a contractible closed walk
$w$ in $M$, called the \emph{projection} of $W$. Any contractible closed walk of $\Minf$ that projects to $w$ is called
a \emph{replication} of $W$ (in the periodic planar representation, a replication of $W$ 
is a translate of $W$ by an integer linear combination 
of two vectors spanning an elementary cell). A closed walk of $\Minf$ is called \emph{admissible} if its interior
does not overlap with the interior of any of its other replications. Clearly, for $W$ a 
contractible closed walk of $\Minf$, the projection of $W$ is a contractible closed walk of $M$ iff $W$ is admissible. 
This ensures that there is a contractible closed walk of length $d'$ in $\Minf$, from which a cycle can be extracted. 
Hence $d'\geq d$. It remains to show that $d\geq d'$. For this, 
we just have to find an admissible cycle of length $d$ in $\Minf$. Let $C$ be a cycle of length $d$ in $\Minf$, 
with the property that the interior of $C$ does not contain the interior of another cycle of length $d$. We are going
to show that $C$ is admissible.  
 Let $C'\neq C$ be a replication of $C$, and let $R,R'$ be the respective interiors of $C$ and $C'$. 
Assume by contradiction that $R\cap R'\neq\emptyset$. Note that we have
\[
|\partial R|+|\partial R'|\geq |\partial(R\cap R')|+|\partial(R\cup R')|,
\]
(indeed we have $\partial(R\cup R')\subseteq \partial R\cup \partial R'$, $\partial(R\cap R')\subseteq \partial R\cup \partial R'$, and $\partial(R\cup R')\cap\partial(R\cap R')\subseteq \partial R\cap \partial R'$, so that
for every edge $e$ of $M$, the contribution of $e$ to $|\partial R|+|\partial R'|$ is at least 
its contribution to $|\partial(R\cap R')|+|\partial(R\cup R')|$). Hence
\[
2d\geq |\partial(R\cap R')|+|\partial(R\cup R')|.
\]
Since $d$ is the minimal cycle-length in $\Minf$ we must have $|\partial(R\cap R')|=|\partial(R\cup R')|=d$.
Hence the contour of $R\cap R'$ is a cycle of length $d$, contradicting the initial hypothesis on $C$. 
\end{proof}

The characterization given in Lemma~\ref{lem:charac_ess} 
easily ensures that the girth of $M$ is at most its essential girth (indeed,
a cycle can be extracted from a shortest contractible closed walk).    
If $M$ has essential girth $d$, a \emph{$d$-angle} of $M$ is a contractible
closed walk of length $d$. 
It is called \emph{maximal} 
if its interior is not contained in the interior of another $d$-angle. 
A toroidal map $M$ is called \emph{essentially simple} if it has essential girth at least $3$
(it means that $M^{\infty}$ is simple, i.e., has no loop nor multiple edges).

For $d\geq 3$, a map is
called a $d$-angulation if all its faces have degree $d$. For $d=3,4,5$,
such maps are respectively called triangulations, quadrangulations,
pentagulations.   
Note that a toroidal $d$-angulation has  essential girth less than or equal
to $d$ (and it can be strictly less), since every face-contour is a $d$-angle.
A toroidal $d$-angulation of essential girth $d$ is called  a 
\emph{$d$-toroidal map}\footnote{The extra condition on the root-face contour mentioned in the abstract and introduction amounts to considering $d$-toroidal maps where the root-face contour is a maximal $d$-angle.}. 
Note that $3$-toroidal maps are exactly essentially simple toroidal
triangulations. By Euler's formula, one can check that, in a toroidal map 
with all face-degrees even, a contractible closed walk must have even length. 
In particular, $4$-toroidal maps are the same as essentially simple quadrangulations.

\subsection{Constrained orientations and weighted biorientations of maps}

For $M$ a map with vertex-set $V$ and edge-set $E$,   
and $\alpha:V\to\mathbb{N}$, an \emph{$\alpha$-orientation}~\cite{Fe03} of $M$ is an orientation of $M$ such that 
every vertex has outdegree $\alpha(v)$.
A \emph{biorientation} of $M$ is the assignment of a direction 
to every half-edge 
(half-edges can be either outgoing or ingoing at 
their incident vertex). 
The \emph{outdegree} of a vertex $v$ is the total number
of outgoing half-edges incident to $v$. 
An \emph{$\N$-biorientation} 
 of $M$ is a biorientation of $M$ where every 
half-edge is given a value in $\mathbb{N}$, which is 
in $\ZZ_{>0}$ if the half-edge is outgoing and equal to zero if the 
half-edge is ingoing. The \emph{weight of a vertex} 
is the total weight of its incident half-edges. The \emph{weight of an edge}  
is the total weight of its two half-edges. Note that 
an orientation can be identified with an $\N$-biorientation 
where every edge has weight $1$.
For $\alpha:V\to\mathbb{N}$ and $\beta:E\to\mathbb{N}$, an \emph{$\ab$-orientation} of $M$ is an $\N$-biorientation of $M$ such that 
every vertex $v$ 
has weight $\alpha(v)$ and every edge $e$ has weight $\beta(e)$. 
In all this paper, we assume that 
 $\beta$ takes only strictly positive values. By doing so we can define
the \emph{$\beta$-expansion} of $M$ as the map $H$ obtained
from $M$ after replacing every edge $e=\{u,v\}$ of $M$ by a group
of $\beta(e)$ parallel edges connecting $u$ and $v$. 
Note that every $\alpha$-orientation of $H$ yields
an $\ab$-orientation of $M$, see Figure~\ref{fig:alpha_ori_rules}.(a). Conversely every 
$\ab$-orientation $X$ of $M$ yields an $\alpha$-orientation 
of $H$, called the $\beta$-expansion of $X$, with the convention that the edge-directions 
in the group of parallel edges 
are chosen in the unique way consistent with the weights
and such that there is no clockwise cycle within
the group,
as shown in Figure~\ref{fig:alpha_ori_rules}(b).

\begin{figure}[!h]
\begin{center}
\includegraphics[width=12cm]{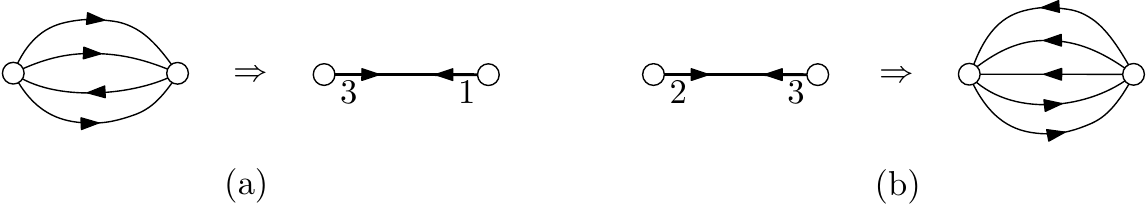}
\end{center}
\caption{(a)
  Rule to obtain an $\ab$-orientation of $M$
from an $\alpha$-orientation of $H$ (with $H$ the $\beta$-expansion of $M$). (b) Rule to obtain an $\alpha$-orientation of $H$
from an $\ab$-orientation of $M$.}
\label{fig:alpha_ori_rules}
\end{figure}

Assume $M$ is a face-rooted  map, with $f$ its marked face.  An orientation
of $M$ is called \emph{non-minimal} if there exists a non-empty set $S$ of faces
such that $f\notin S$ and every edge on the boundary of $S$ has a
face in $S$ on its right (and a face not in $S$ on its left).  It is
called \emph{minimal} otherwise. An $\ab$-orientation of $M$ is called
\emph{minimal} if its $\beta$-expansion $H$ is minimal (where the root-face
of $H$ is the one corresponding to $f$). 
Equivalently, 
 an $\ab$-orientation of $M$ is non-minimal if there exists a non-empty set $S$ of faces
such that $f\notin S$ and every edge on the boundary of $S$ either is simply directed with a face in $S$ on its right
or is bidirected. 

Consider an orientation of $M$ and
a non-contractible cycle $C^*$ of $M^*$
given with a traversal direction (i.e., a cyclic ordering $(h_0,\ldots,h_{2k-1})$ 
of the half-edges on the cycle such that any two successive half-edges $h_{2i},h_{2i+1}$ are opposite on the same edge,
and any two successive half-edges $h_{2i+1},h_{(2i+2)\ \mathrm{mod}\ 2k}$ are at the same vertex).     
Let $\delta_R(C^*)$ (resp. $\delta_L(C^*)$) 
be the number of edges of $M$ crossing $C^*$ from 
left to right (resp. from right to left). 
Then the \emph{$\delta$-score} 
of $C^*$ is defined as $\delta(C^*)=\delta_R(C^*)-\delta_L(C^*)$.
Two $\alpha$-orientations 
$X,X'$ are called \emph{$\delta$-equivalent} if
every non-contractible cycle of $M^*$ has the same $\delta$-score 
in $X$ and in $X'$. 
The following statement is easily deduced from the results and observations in~\cite{Pro93}
(in particular the fact that the set of contours of non-root faces plus two non-homotopic non-contractible cycles 
form a basis of the cycle-space):

\begin{theorem}[\cite{Pro93}]\label{theo:propp}
Let $M$ be a face-rooted 
map on the orientable surface of genus $g$  endowed with an $\alpha$-orientation $X$.  
Then $M$ has a unique $\alpha$-orientation
 $X_0$ that is minimal\footnote{It is actually proved in~\cite{Pro93} that the set of $\alpha$-orientations that are $\delta$-equivalent
to $X$ is a distributive lattice, of which $X_0$
is the minimum element.} and $\delta$-equivalent to $X$.

Moreover, suppose that $M$ is a toroidal map and $X,X'$
are two $\alpha$-orientations of $M$. If there exist two non-contractible
non-homotopic\footnote{Two closed 
curves on a surface are called \emph{homotopic} 
if one can be continuously deformed into the other.} cycles of $M$ that have the same $\delta$-score in $X$ and in
$X'$, then $X,X'$ are $\delta$-equivalent.
\end{theorem}

We now define the analogue of the function $\gamma$ 
introduced in~\cite{DGL15,GKL15} for Schnyder woods (see also~\cite{LevHDR} for a detailed presentation).

 If $M$ is endowed with an orientation, and $C$ is a non-contractible
 cycle of $M$ given with a traversal direction, we denote by
 $\gamma_R(C)$ (resp. $\gamma_{L}(C)$) the total number of edges going
 out of a vertex on $C$ on the right (resp. left) side of $C$, and
 define the \emph{$\gamma$-score} of $C$ as
 $\gamma(C)=\gamma_R(C)-\gamma_L(C)$. Two $\alpha$-orientations $X,X'$ 
of $M$ are called \emph{$\gamma$-equivalent} if every non-contractible cycle of $M$
has the same $\gamma$-score in $X$ as in $X'$. 
The following theorem is  an analog (and a consequence) of Theorem~\ref{theo:propp}; 
we only state it in genus $1$, to keep the proof simpler and as it is the focus of the article. 

\begin{corollary}\label{theo:gamma}
Let $M$ be a face-rooted toroidal 
map endowed with an $\alpha$-orientation $X$.  
Then $M$ has a unique $\alpha$-orientation
 $X_0$ that is minimal and $\gamma$-equivalent to $X$.

 Moreover, for two $\alpha$-orientations $X,X'$ of $M$ to be $\gamma$-equivalent, it 
is enough that two non-contractible
non-homotopic cycles of $M$ have the same $\gamma$-score in $X$ and in
$X'$.
\end{corollary}
\begin{proof}
The \emph{completion-map} of $M$ is the map $\hM$ obtained by 
superimposing $M$ and $M^*$. The vertices of $\hM$ are of 3 types: primal vertices
(those of $M$), dual vertices (those of $M^*$) and edge-vertices (those, of degree $4$, at the 
intersection of an edge $e\in M$ with its dual edge $e^*\in M^*$). Let $\ha$ 
be the function from the vertex-set of $\hM$ to $\mathbb{N}$ such that, 
if $v$ is a primal vertex of $\hM$ then $\ha(v)=\alpha(v)$, if $v$
is a dual vertex of $\hM$ then $\ha(v)=\mathrm{deg}(v)$, and if $v$
is an edge-vertex of $\hM$ then $\ha(v)=1$. Note that any $\alpha$-orientation
$Z$ of $M$ yields an $\ha$-orientation $\hZ$ of $\hM$: each edge of $\hM$
corresponding to a half-edge of an edge $e\in M$ is assigned the direction of $e$ in $Z$, 
and each edge of $\hM$ corresponding to a half-edge of an edge $e^*\in M^*$ 
is directed toward the incident edge-vertex. Clearly the mapping sending $Z$ to $\hZ$ 
is a bijection from the $\alpha$-orientations of $M$ to the $\ha$-orientations of $\hM$,
with the property that $Z$ is minimal if and only if $\hZ$ is minimal. 

Let $C$ be a non-contractible cycle of $M$ given with a traversal direction. 
Let $(c_1,\ldots,c_k)$ be the cyclic sequence of corners of $M$ that are encountered
when walking ``just to the right'' of $C$.  
Since every corner
of $M$ corresponds to a face of $\hM^*$, the cyclic sequence $(c_1,\ldots,c_k)$ 
identifies to a non-contractible cycle of $\hM^*$, 
which we denote by $C^*$, see Figure~\ref{fig:to_right_noncontract_cycle}
 (note that $C^*$ is clearly homotopic to $C$). 
It is then easy to see that for every $\alpha$-orientation 
$Z$ of $M$, we have
\[
\gamma_R^{Z}(C)=\delta_R^{\hZ}(C^*).
\]
Hence, for two $\alpha$-orientations $X,X'$ of $M$, and for $C$ a
non-contractible cycle of $M$ given with a traversal direction, we
have $\gamma^X(C)=\gamma^{X'}(C)$ iff $\gamma_R^X(C)=\gamma_R^{X'}(C)$
iff $\delta_R^{\hX}(C^*)=\delta_R^{\hX'}(C^*)$ iff
$\delta^{\hX}(C^*)=\delta^{\hX'}(C^*)$.  Hence $X,X'$ are
$\gamma$-equivalent if and only if $\hX,\hX'$ are $\delta$-equivalent,
where we use the second statement in Theorem~\ref{theo:propp} to have
the `only if' direction\footnote{While we do not need it here, we also
  mention that it is easy to prove by similar arguments that
  $\hX,\hX'$ are $\delta$-equivalent iff $X,X'$ are
  $\delta$-equivalent. Hence the $\gamma$-equivalence classes on
  $\alpha$-orientations are the same as the $\delta$-equivalence
  classes on $\alpha$-orientations (which are distributive
  lattices).}.

\begin{figure}[!h]
\begin{center}
\includegraphics[width=7cm]{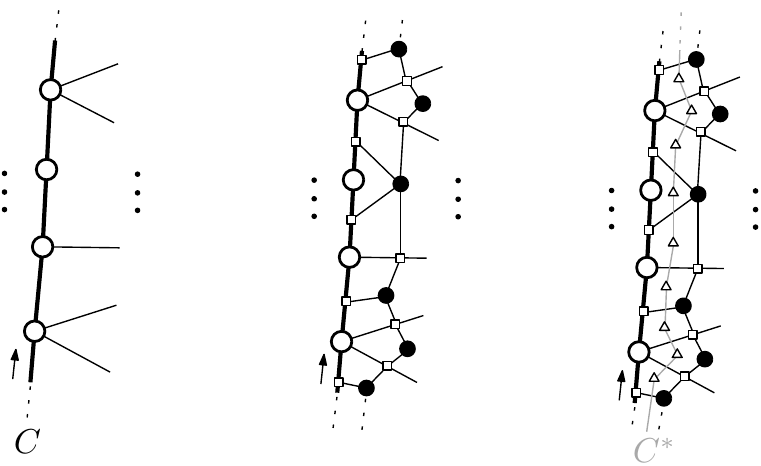}
\end{center}
\caption{Left: a non-contractible cycle $C$ of $M$. 
Middle: the situation in $\hM$ in the right neighborhood of $C$. 
Right: the corresponding non-contractible cycle $C^*$ of $\hM^*$
(which passes by vertices, represented as triangles, that are in the faces corresponding to the corners incident to $C$ on its right side).}
\label{fig:to_right_noncontract_cycle}
\end{figure}

It is then easy to prove the theorem. For $X$ an $\alpha$-orientation of $M$, 
Theorem~\ref{theo:propp} ensures that there exists an $\ha$-orientation  $\hX_0$
of $\hM$ that is minimal and $\delta$-equivalent to $\hX$. By what precedes, 
$X_0$ is $\gamma$-equivalent to $X$ (and is minimal), hence we have the existence part. 
Moreover, if there was another $\alpha$-orientation $X_1$ minimal and $\gamma$-equivalent
to $X$, then $\hX_1$ would be minimal, $\delta$-equivalent to $\hX$, and different from $\hX_0$, 
yielding a contradiction. This gives the uniqueness part.

We now prove the second statement of the theorem.
Let $X,X'$ be two  $\alpha$-orientations of $M$ that have the same
$\gamma$-score for two  non-contractible non-homotopic cycles
$C_1,C_2$. 
By what precedes, $C_1^*$ and $C_2^*$ have the same $\delta$-score in
$\hX$ and in $\hX'$. Hence, by Theorem~\ref{theo:propp}, $\hX$ and $\hX'$ 
are $\delta$-equivalent, so that $X$ and $X'$ are $\gamma$-equivalent. 
\end{proof}

 More generally if $M$ is
 endowed with an $\N$-biorientation
 and $C$ is a non-contractible
 cycle of $M$ given with a traversal direction, we denote by $\gamma_R(C)$
 (resp. $\gamma_L(C)$) the total weight of half-edges incident to a
 vertex on $C$ on the right (resp. left) side of $C$, and define the
 \emph{$\gamma$-score} of $C$ as $\gamma(C)=\gamma_R(C)-\gamma_L(C)$.

Two  $\ab$-orientations, 
$X,X'$ are called
\emph{$\gamma$-equivalent} if every non-contractible cycle
of $M$ has the same $\gamma$-score in $X$ and in $X'$. 
The following theorem is a generalization (and a consequence) 
of Corollary~\ref{theo:gamma} that will be useful for our purpose.

\begin{corollary}\label{theo:general_gamma}
Let $M$ be a face-rooted toroidal 
map endowed with an $\ab$-orientation $X$.  
Then $M$ has a unique $\ab$-orientation
 $X_0$ that is minimal and $\gamma$-equivalent to $X$.

 Moreover, for two $\ab$-orientations $X,X'$ of $M$ to be $\gamma$-equivalent, it 
is enough that two non-contractible
non-homotopic cycles of $M$ have the same $\gamma$-score in $X$ and in
$X'$.
\end{corollary}

\begin{proof}
Let $H$ be the $\beta$-expansion of $M$. For $Z$ an $\ab$-orientation of $M$, 
let $\oZ$ be the $\beta$-expansion of $Z$, i.e., the $\alpha$-orientation of $H$ 
obtained from $Z$ by applying the rule of Figure~\ref{fig:alpha_ori_rules}(b). 
For $C$ a non-contractible cycle of $M$ given with a traversal direction, 
let $\oC$ be the non-contractible cycle of $H$ that goes along $C$ 
in the ``rightmost'' way, i.e. for each edge $e$ of $C$, the cycle $\oC$ passes
by the rightmost edge in the group of $\beta(e)$ edges arising from~$e$. 
Clearly
\[
\gamma_R^Z(C)=\gamma_R^{\oZ}(\oC).
\]
Hence, for two $\ab$-orientations $X,X'$ of $M$ and for $C$ a
non-contractible cycle of $M$ given with a traversal direction, 
we have $\gamma^X(C)=\gamma^{X'}(C)$ iff $\gamma_R^X(C)=\gamma_R^{X'}(C)$
iff $\gamma_R^{\oX}(\oC)=\gamma_R^{\oX'}(\oC)$ iff $\gamma^{\oX}(\oC)=\gamma^{\oX'}(\oC)$.
 Hence  $X,X'$ 
are $\gamma$-equivalent if and only if $\oX,\oX'$ are $\gamma$-equivalent
(we use the second statement in Corollary~\ref{theo:gamma} to have the `only if' direction). 

For $X$ an $\ab$-orientation of $M$, Corollary~\ref{theo:gamma} ensures that 
$H$ has an $\alpha$-orientation $\oX_0$ that is minimal and $\gamma$-equivalent
to $\oX$. Let $X_0$ be the $\ab$-orientation obtained from $\oX_0$ 
by applying the rule of Figure~\ref{fig:alpha_ori_rules}(a). Since $\oX_0$ is minimal, there is 
no clockwise cycle inside any group of $\beta(e)$ edges associated to an edge $e\in M$. 
Hence $\oX_0$ is the $\beta$-expansion of $X_0$, so that (by definition) 
$X_0$ is minimal, and moreover it is $\gamma$-equivalent
 to $X$. This proves the existence part. If there was another $\ab$-orientation $X_1$
minimal and $\gamma$-equivalent to $X$, then $\oX_1$ would be minimal, $\gamma$-equivalent
to $\oX$, and different from $\oX_0$, contradicting Corollary~\ref{theo:gamma}. 
This proves the uniqueness part. 

Let us now prove the second statement of the theorem.
Let  $X,X'$ be two  $\alpha$-orientations of $M$ that have the same
$\gamma$-score for two  non-contractible non-homotopic cycles $C_1,C_2$.
By what precedes, $\oC_1$ and $\oC_2$ have the same $\gamma$-score in
$\oX$ and in $\oX'$. Hence, by Corollary~\ref{theo:gamma}, $\oX$ and $\oX'$ 
are $\gamma$-equivalent, so that $X$ and $X'$ are $\gamma$-equivalent. 
\end{proof}

\section{Balanced $\frac{d}{d-2}$-orientations on the torus}\label{sec:balanced}
Let $M$ be a toroidal map. 
We say that an $\N$-biorientation of $M$ is \emph{balanced} if the
 $\gamma$-score of any non-contractible cycle of $M$ is $0$. 
Note that Corollary~\ref{theo:general_gamma} implies that if  $M$ is face-rooted and 
admits a balanced $\ab$-orientation, then $M$ admits a unique balanced
$\ab$-orientation that is minimal. 
For a toroidal $d$-angulation $M$, we define a \emph{$\frac{d}{d-2}$-orientation} of $M$
as an $\NN$-biorientation of $M$ such that every vertex has weight $d$ and every edge
has weight $d-2$ (our bijections for toroidal $d$-angulations of essential girth $d$
 will crucially rely on minimal balanced \emph{$\frac{d}{d-2}$-orientation}). The purpose of this 
section is to show that a toroidal $d$-angulation admits a $\frac{d}{d-2}$-orientation iff it has essential girth $d$,
and that in that case it admits a balanced $\frac{d}{d-2}$-orientation.

\subsection{Necessary condition on the essential girth}
\label{sec:exist}
The following lemma gives a necessary condition for a  toroidal
$d$-angulation to admit a $\frac{d}{d-2}$-orientation.

\begin{lemma}
\label{th:existencefractional}
  If a toroidal $d$-angulation admits a $\frac{d}{d-2}$-orientation
  then it has essential girth~$d$ (i.e. it is a $d$-toroidal map).
\end{lemma}
To prove it, note that the essential girth is clearly at most $d$ since faces have degree~$d$. The fact that the essential girth is at least $d$ is actually a direct consequence of the following statement (which will also be useful in proofs later):
\begin{claim}\label{claim:epsilon}
Let $M$ be a toroidal $d$-angulation endowed with a $\frac{d}{d-2}$-orientation, and let $W$ be a contractible closed walk of length $k$, with $R$ the (contractible) enclosed region. Let $\epsilon$ be the sum of the weights of half-edges in $R$ that are incident to a vertex on $W$. Then $\epsilon=k-d$.  
\end{claim}
\begin{proof}
Let $n',m',f'$ be respectively the numbers of vertices, edges and faces of $M$ that are (strictly) inside $R$. 
Since all faces of $M$ have degree $d$ we have (i) $df'=2m'+k$. Since the weight of every vertex (resp. edge) is $d$ (resp. $d-2$), we have (ii) $dn'+\epsilon=(d-2)m'$. Finally, since $R$ is contractible, the Euler relation ensures that (iii) $n'-m'+f'=1$. Taking (i)+(ii) gives $d(n'-m'+f')=k-\epsilon$, which together with (iii) gives $d= k-\epsilon$. 
\end{proof}

We will see in the Section~\ref{sec:existence} that, conversely, 
any $d$-toroidal map admits a $\frac{d}{d-2}$-orientation, 
and even more, it admits a balanced one.

\subsection{Sufficient condition for balancedness}
\label{sec:pdc}

The next lemma shows that $\gamma$ behaves well with respect to~homotopy
in  $\frac{d}{d-2}$-orientations:

\begin{lemma}
  \label{lem:gammahomology} 
  Let $M$ be a $d$-toroidal map endowed with 
a $\frac{d}{d-2}$-orientation, let $C$ be a non-contractible
  cycle of $M$ given with a traversal direction, and let $\{B_1,B_2\}$ be a
  basis for the homotopy of $M$, such that $B_1,B_2$ are
  non-contractible cycles whose intersection is a single vertex or a
  common path. Let $k_1,k_2\in \mathbb Z^2$, such that $C$ is homotopic
  to $k_1 B_1 + k_2B_{2}$. Then
  $\gamma(C)=k_1\,\gamma (B_1) + k_2\,\gamma (B_2)$.
\end{lemma}

\begin{proof} 
  Let $v$ and $u$ be the two extremities of the path $B_1\cap B_2$ 
(possibly $v=u$, if $B_1\cap B_2$ is reduced to a single vertex). 
  Consider a drawing of ${M}^\infty$ obtained by replicating a flat
  representation of ${M}$ to tile the plane.  Let $v_0$ be a copy of
  $v$ in ${M}^\infty$.  Consider the walk $W$ starting from $v_0$ and
  following $k_1$ times the edges corresponding to $B_1$ and then
  $k_2$ times the edges corresponding to $B_2$ (we are going backward
  if $k_i$ is negative). This walk ends at a copy $v_1$ of $v$.  Since
  $C$ is non-contractible we have $k_1$ or $k_2$ not equal to $0$ and
  thus $v_1$ is distinct from $v_0$.  Let $W^\infty$ be the infinite
  walk obtained by replicating $W$ (forward and backward) from $v_0$.
  Note that their might be some repetition of vertices in $W^\infty$
  if the intersection of $B_1,B_2$ is a path. But in that case, by the
  choice of $B_1,B_2$, the walk $W^\infty$ is almost a path,
  except maybe at all the transitions from ``$k_1 {B_1}$'' to
  ``$k_2B_{2}$'', or (exclusive or) at all the transitions from ``$k_2 {B_2}$'' to
  ``$k_1B_{1}$'', where it can go back and forth a path
  corresponding to the intersection of $B_1$ and $B_2$. The existence
  or not of such ``back and forth'' parts depends on the signs of
  $k_1, k_2$ and the way $B_1,B_2$ are going through their common
  path. Figure~\ref{fig:replicating} gives an example of this
  construction with $(k_1,k_2)=(1,1)$ and $(k_1,k_2)=(1,-1)$ when
  $B_1,B_2$ intersect on a path and are oriented the same way along
  this path as in Figure~\ref{fig:replicating0}.

\begin{figure}[!h]
\center
\includegraphics[scale=0.3]{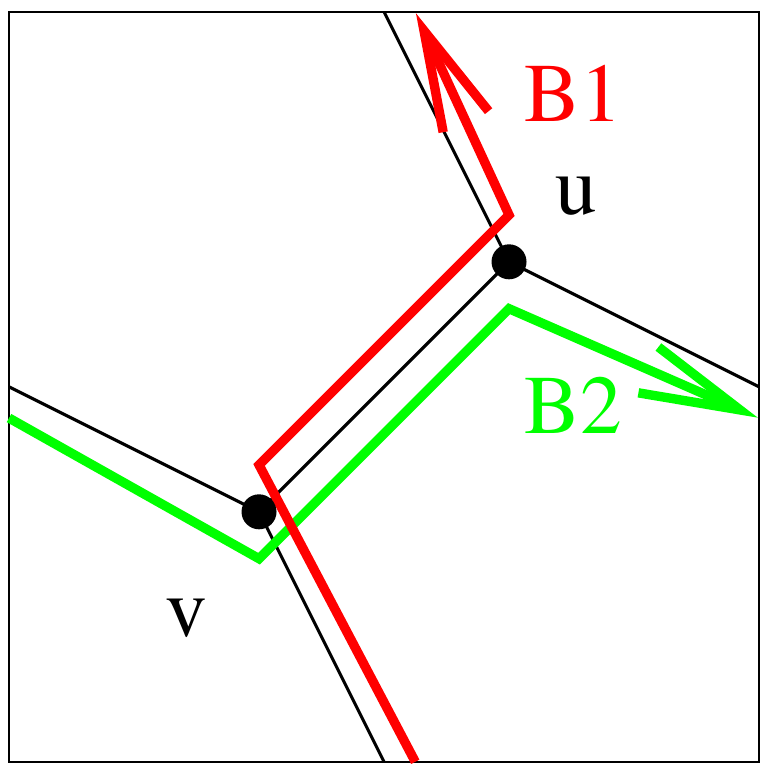}
\caption{Intersection of the basis.}
\label{fig:replicating0}
\end{figure}

\begin{figure}[!h]
\center
\begin{tabular}{cc}
\includegraphics[scale=0.3]{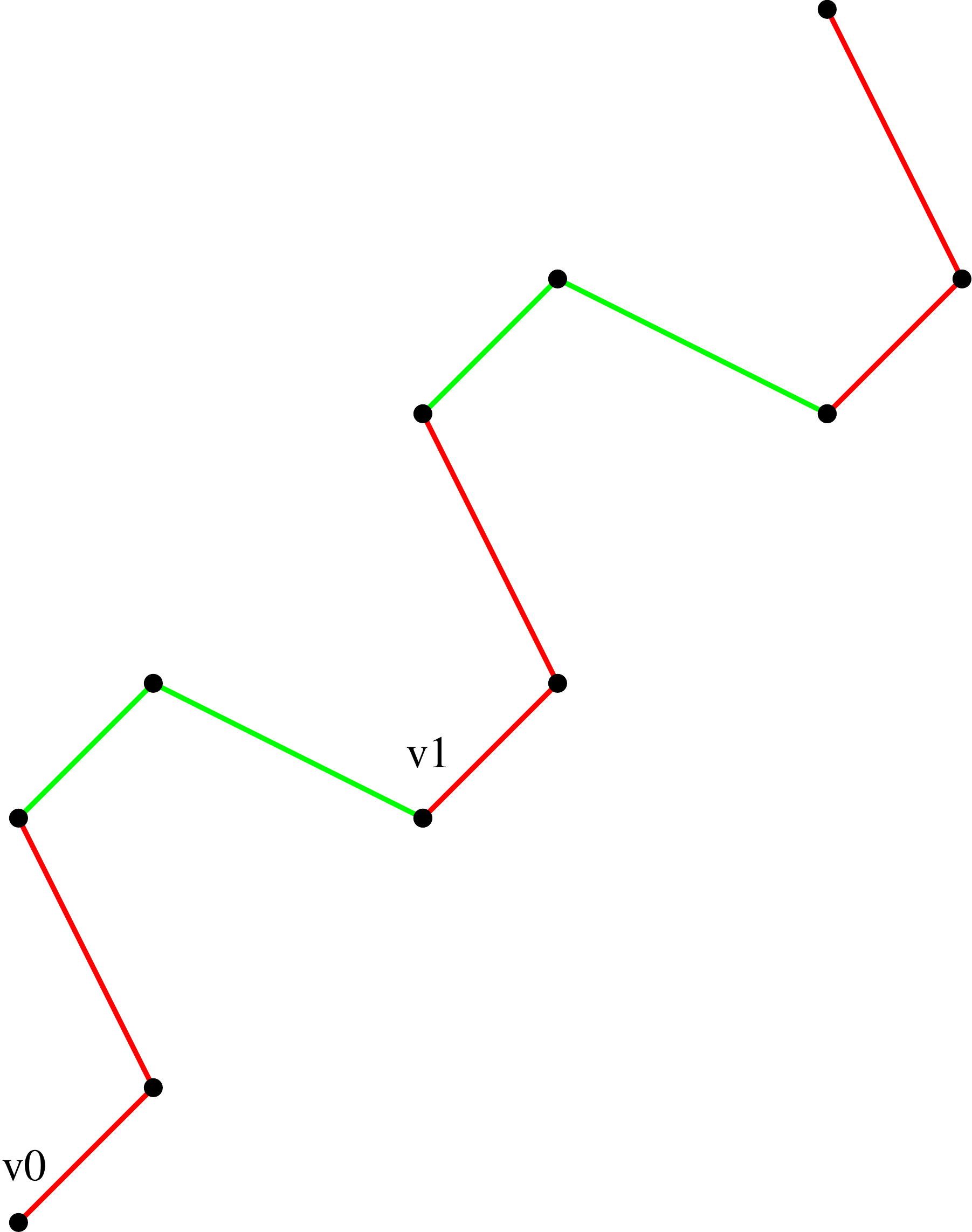} \ \ \ & \ \ \
\includegraphics[scale=0.3]{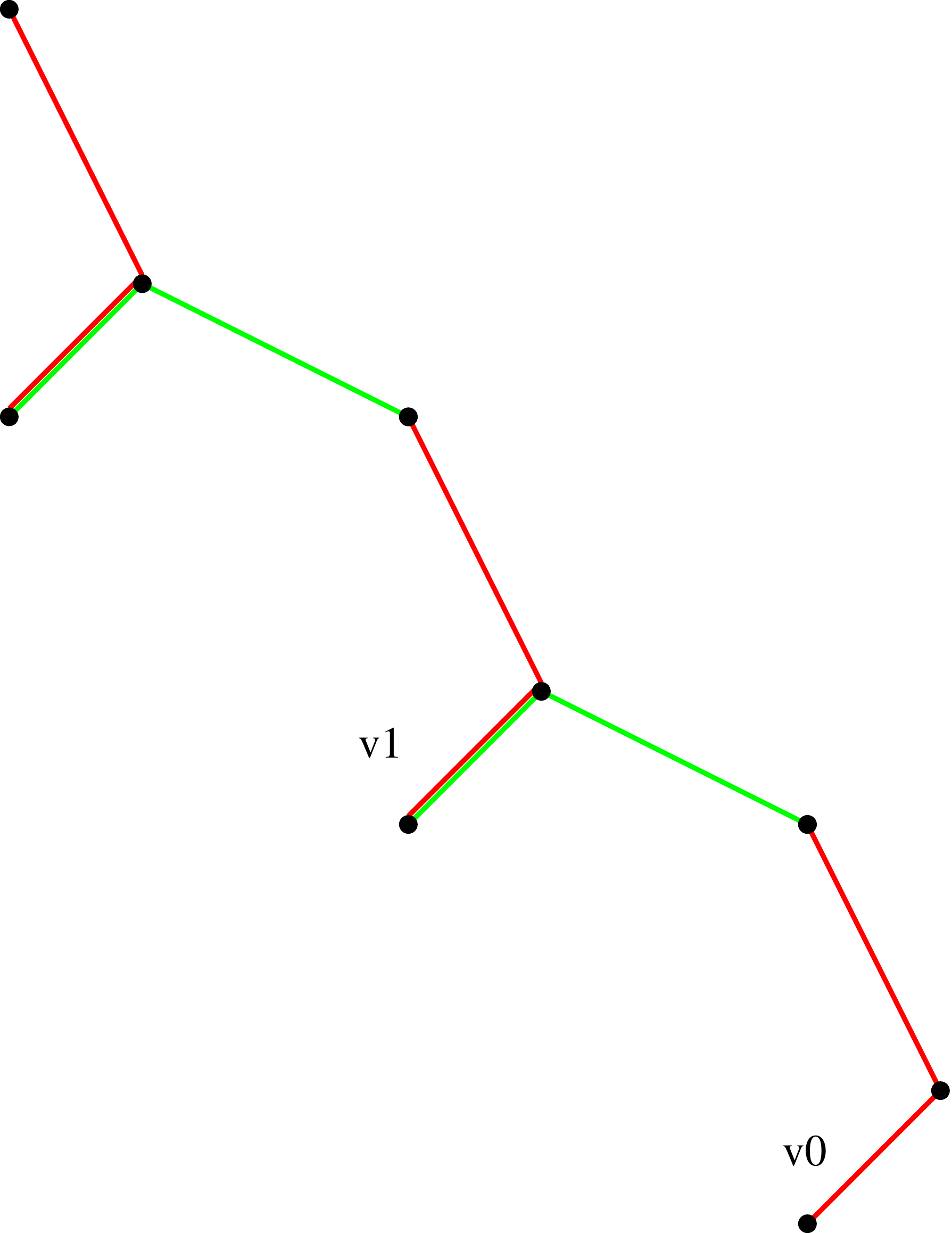} \\
$(k_1,k_2)=(1,1)$ \ \ \ & \ \ \ $(k_1,k_2)=(1,-1)$\\
\end{tabular}
\caption{Replicating ``$k_1 {B_1}$'' and ``$k_2B_{2}$'' in the universal cover.}
\label{fig:replicating}
\end{figure}

   We ``simplify'' $W^\infty$ by removing all the parts that consist 
   of going back and forth along a path (if any) and call $B^\infty$
   the obtained walk that is now without repetition of vertices. By
   the choice of $v$, the walk $B^\infty$ goes through copies of
   $v$. If $v_0,v_1$ are no more a vertex along $B^\infty$, because of
   a simplification at the transition from ``$k_2 {B_2}$'' to
   ``$k_1B_{1}$'', then we replace $v_0$ and $v_1$ by the next copies
   of $v$ along $W^\infty$, i.e., at the transition from
   ``$k_1 {B_1}$'' to ``$k_2B_{2}$''.

Since ${C}$ is homotopic to $k_1 {B_1} + k_{2}{B_{2}}$,
we can find an infinite path $C^\infty$, that corresponds to copies of
$C$ replicated, that does not intersect $B^\infty$ and situated on the
right side of $B^\infty$. Now we can find a copy $B'^\infty$ of
$B^\infty$, such that $C^\infty$ lies between $B^\infty$ and
$B'^\infty$ without intersecting them. We choose two copies $v'_0,v'_1$
of $v_0,v_1$ on $B'^\infty$ such that the vectors $v_0v_1$ and
$v'_0v'_1$  are equal.

Let $R_0$ be the region bounded by $B^\infty$ and $B'^\infty$.  Let $R_1$
(resp. $R_2$) be the subregion of $R_0$ delimited by $B^\infty$ and
$C^\infty$ (resp. by $C^\infty$ and $B'^\infty$).  We consider
$R_0,R_1,R_2$ as cylinders, where the lines $(v_0,v'_0),(v_1,v_1')$
(or part of them) are identified. Let $B,B',C'$ be the cycles of $R_0$
corresponding to $B^\infty, B'^\infty, C^\infty$ respectively. 

Let $x$ be the sum of the weights of the half-edges of ${M}$ incident
to $B$ and in the strict interior of $R_1$. Let $y$ be the sum of the
weights of the half-edges of ${M}$ incident to $B'$ and in the strict
interior of $R_2$.  Let $x'$ (resp. $y'$) be the sum of the weights of
the half-edges of ${M}$ incident to $C'$ and in the strict interior of
$R_2$ (resp. $R_1$). Note that $C'$ corresponds to exactly one copy of
$C$, so $\gamma(C)=x'-y'$.  Similarly, $B$ (and $B'$ as well) 
``almost'' corresponds to $k_1$ copies of $B_1$ followed by $k_2$ copies
of $B_2$, except for the fact that we may have removed a back and forth
part (if any).  In any case we have the following:

\begin{claimn}
\label{cl:computegamma}
  $k_1\,\gamma
(B_1) + k_2\,\gamma (B_2)=x-y$
\end{claimn}

 \begin{proofclaim}
   We prove the case where the common intersection of $B_1,B_2$ is a
   path (if the intersection is a single vertex, the proof is very
   similar and even simpler). We assume, by possibly 
   reversing one of $B_1$ or $B_2$, that $B_1,B_2$ are oriented the
   same way along their intersection, so we are in the situation of
    Figure~\ref{fig:replicating0}.  

Figure~\ref{fig:computegamma1} shows
   how to compute $k_1\,\gamma (B_1) + k_2\,\gamma (B_2)+y-x$ when
   $(k_1,k_2)=(1,1)$. Then, one can check that
   the weight of each half-edge of ${M}$ is counted exactly the same number
   of times positively and negatively. So everything compensates and we
   obtain $k_1\,\gamma (B_1) + k_2\,\gamma (B_2)+y-x=0$.

\begin{figure}[!h]
\center
\begin{tabular}{ccccc}
\includegraphics[scale=0.27]{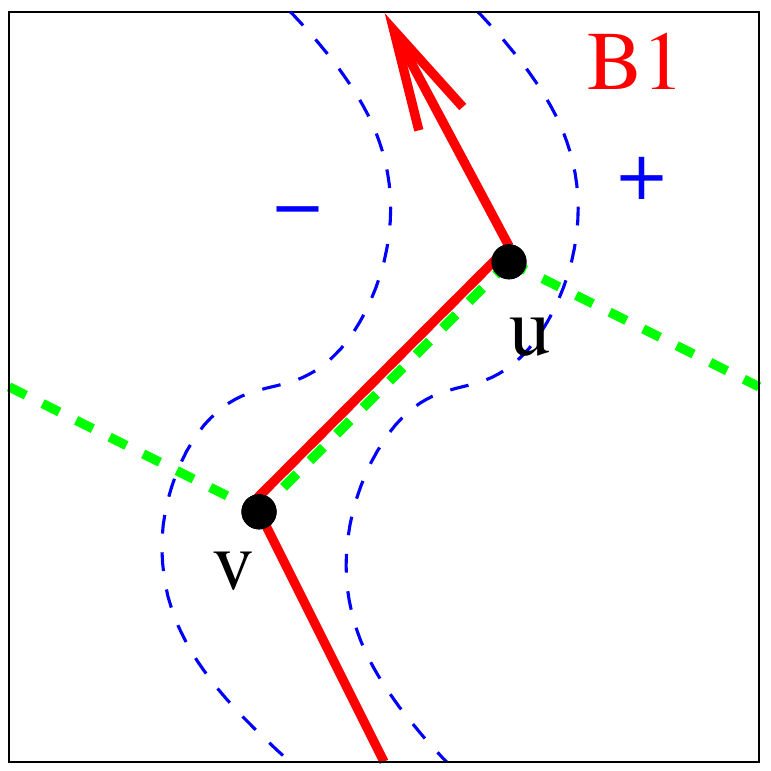}
&&\includegraphics[scale=0.27]{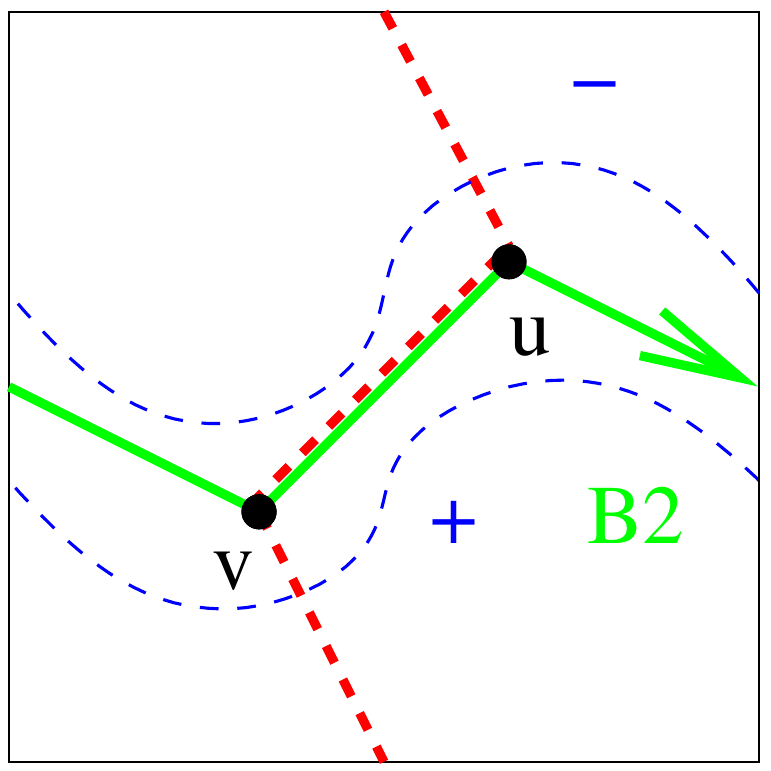}
&&\includegraphics[scale=0.27]{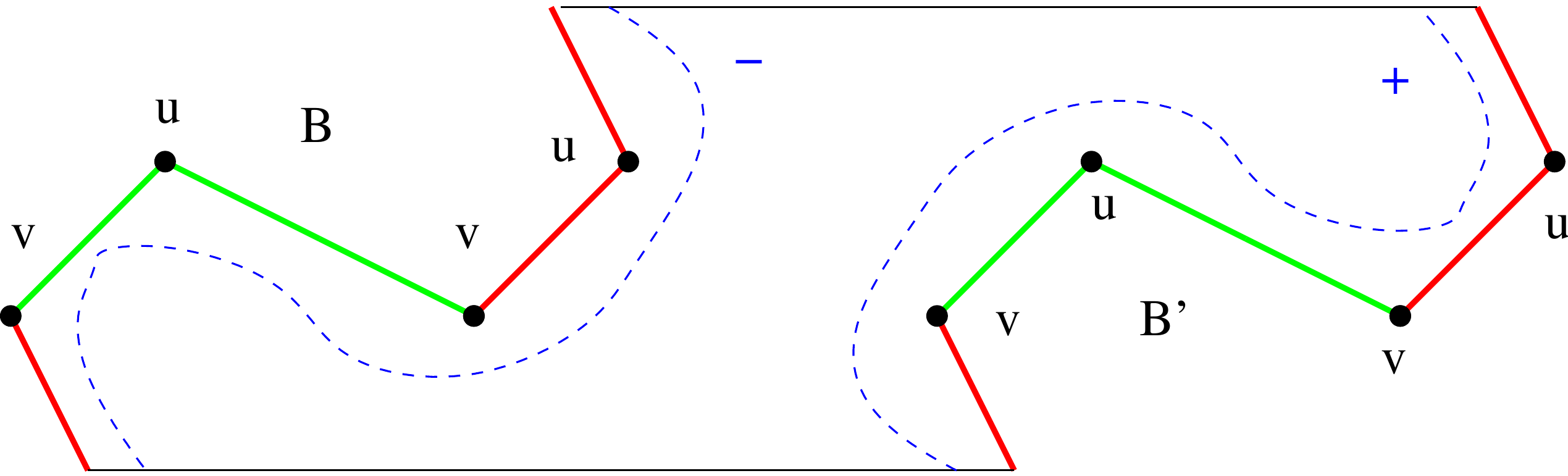} \\
$\gamma(B_1)$&$+$ & $\gamma (B_2)$ &$+$& $(y-x)$
\end{tabular}
\caption{Case $(k_1,k_2)=(1,1)$.}
\label{fig:computegamma1}
\end{figure}

Figure~\ref{fig:computegamma2} shows how to compute
$k_1\,\gamma (B_1) + k_2\,\gamma (B_2)+y-x$ when
$(k_1,k_2)=(1,-1)$. As above, most of the things compensate but, in
the end, we obtain $k_1\,\gamma (B_1) + k_2\,\gamma (B_2)+y-x$ equals
the sum of the weights of the half-edges incident to $u$ minus the sum
of the weights of the half-edges incident to $v$.  Since the sum of the
weights of the half-edges at each vertex is equal to $d$, we 
again conclude that $k_1\,\gamma (B_1) + k_2\,\gamma (B_2)+y-x=0$.

\begin{figure}[!h]
\center
\begin{tabular}{ccccc}
\includegraphics[scale=0.3]{gammabasis-5}
&&\includegraphics[scale=0.3]{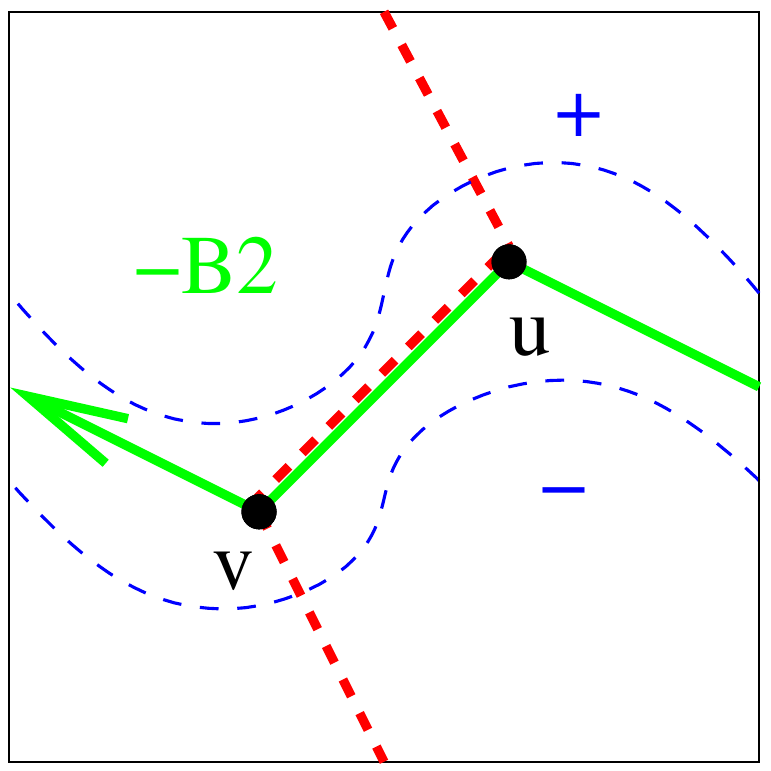}
&&\includegraphics[scale=0.3]{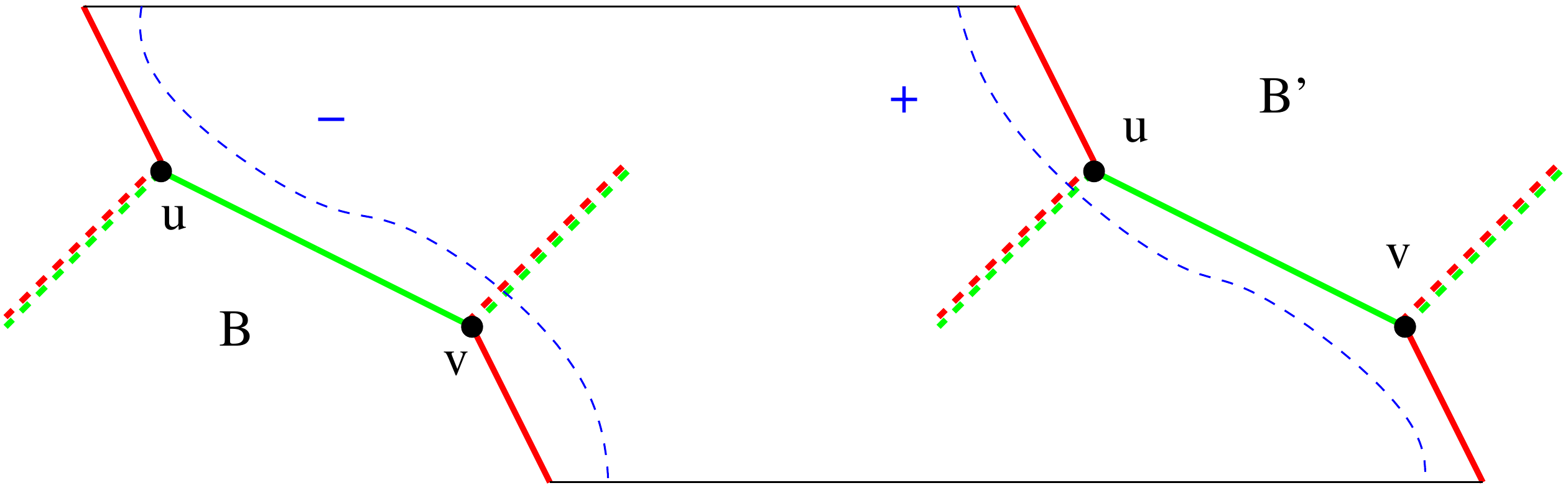}\\
$\gamma(B_1)$&$+$ & $(-\gamma (B_2))$ &$+$& $(y-x)$
\end{tabular}

\caption{Case $(k_1,k_2)=(1,-1)$.}
\label{fig:computegamma2}
\end{figure}

One can easily be
convinced that when $|k_1|\geq 1$ and $|k_2|\geq 1$ then the same arguments
 apply. The only difference is that the red or green part of the
figures in the universal cover would be
longer (with  repetitions of $B_1$ and $B_2$). These parts being 
``smooth'', they do not affect the way we compute the  equality.
Finally, if one of $k_1$ or $k_2$ is equal to zero, the analysis
is much  simpler and the conclusion holds.
 \end{proofclaim}

  For $i\in\{0,1,2\}$, let $G_i$ be the cylinder map made of all the
  vertices and edges of ${M}^\infty$ that are in the cylinder
  region $R_i$. 
  Let $k$ (resp. $k'$) be the length of $B$ (resp. $C'$).  Let
  $n_1,m_1,f_1$ be respectively the number of vertices, edges and
  faces of $G_1$.  Since $G_1$ is a $d$-angulation we have
  $2m_1=df_1+(k+k')$.  The total weight of the edges of $G_1$ is
  $(d-2)m_1=dn_1-(x'+y)$.  Combining these equalities with Euler's
  formula $n_1-m_1+f_1=0$, one obtains $k+k'=x'+y$. Similarly, by
  considering $G_2$, one obtains $k+k'=x+y'$.  Thus 
  $x'+y=x+y'$, which gives 
  $\gamma(C)=k_1\,\gamma (B_1) + k_2\,\gamma (B_2)$ using the claim.
\end{proof}

Lemma~\ref{lem:gammahomology} implies the following:

\begin{lemma}
\label{lem:gamma0all}
Let $M$ be a $d$-toroidal map endowed with a
$\frac{d}{d-2}$-orientation. If  the $\gamma$-score of two
  non-contractible non-homotopic cycles of
  $M$ is $0$, then the orientation is balanced.
\end{lemma}

\begin{proof}
  Consider two non-contractible  non-homotopic cycles $C,C'$ of
  ${M}$, each with a chosen traversal direction, such that $\gamma(C)=\gamma(C')=0$.
  Consider a homotopy basis $\{B_1,B_2\}$ of
$M$, such that $B_1,B_2$ are non-contractible cycles whose
intersection is a single vertex or a path. Note that one can easily obtain such a basis
by  considering a spanning tree $T$ of $M$, and a spanning tree
$T^*$ of $M^*$ that contains no edges dual to $T$.  By Euler's
formula, there are exactly $2$ edges in $M$ that are not in $T$ nor
dual to edges of $T^*$. Each of these edges forms a unique cycle
with $T$. These two cycles, given with
any traversal direction, form the wanted basis.

Let $k_1,k_2,k'_1,k'_2\in \mathbb Z^4$, such that $C$ (resp. $C'$) is
homotopic to $k_1 B_1+k_2 B_2$ (resp.  $k'_1 B_1+k'_2 B_2$).  Since
$C$ is non-contractible we have $(k_1,k_2)\neq (0,0)$. By possibly 
exchanging $B_1,B_2$, we can assume, without loss of generality that $k_1\neq 0$.  By
Lemma~\ref{lem:gammahomology}, we have
$k_1 \gamma(B_1)+k_2 \gamma(B_2)=\gamma(C)=0=\gamma(C')=k'_1
\gamma(B_1)+k'_2 \gamma(B_2)$.  So
$\gamma(B_1)=(-k_2/k_1) \gamma(B_2)$ and thus
$(-k_2 k'_1/k_1 + k'_2)\gamma(B_2)=0$. So $k'_2=k_2 k'_1/k_1$ or
$\gamma(B_2)=0$.  Suppose by contradiction, that $\gamma(B_2)\neq 0$.
Then $(k'_1,k'_2)= \frac{k'_1}{k_1} (k_1,k_2)$, and $C'$ is homotopic
to $\frac{k'_1}{k_1} C$.  Since $C$ and $C'$ are both non-contractible
cycles, it is not possible that one is homotopic to a multiple of the
other, with a multiple different from $-1,1$. So $C, C'$ are
homotopic, a contradiction.  So $\gamma(B_2)= 0$ and thus
$\gamma(B_1)=0$.  Then by Lemma~\ref{lem:gammahomology} we have $\gamma(C)=0$
for  any
non-contractible cycle $C$ of $M$, and thus
the orientation is balanced. 
\end{proof}

\subsection{Existence of balanced toroidal $\frac{d}{d-2}$-orientations}
\label{sec:existence}

The main goal of this section is to prove the following existence result:

\begin{proposition}
\label{th:existencebalanced}
Any toroidal $d$-angulation with essential girth $d$ admits a
balanced $\frac{d}{d-2}$-orientation.
\end{proposition}

In the case of toroidal triangulations, essentially toroidal
3-connected maps, or essentially 4-connected toroidal triangulations,
the proof of existence of analogous ``balanced orientations'' can be
done by doing edge-contractions until reaching a map with few
vertices (see~\cite{LevHDR,BoLe18}).  We do not know if 
such a strategy could be applied for $d\geq 5$ (indeed the contraction of an 
edge in a $d$-toroidal map  results in some faces of size strictly
less than $d$). So we use a different technique in the current paper.

The method consists in defining orientations that are ``totally
unbalanced'' ---which we call biased orientations---   
then taking a linear combinations of these biased 
orientations to obtain a balanced orientation  but with rational
weights, and finally proving that the orientation that is minimal and
$\gamma$-equivalent to it is a balanced orientation with integer
weights.

\subsubsection{Biased orientations}

Consider a $d$-toroidal map $M$, and let $C$ be a non-contractible
cycle of $M$ of length $k$ given with a traversal direction.  A
\emph{biased orientation w.r.t.~$C$} is a $\frac{d}{d-2}$-orientation
of $M$ such that $\gamma(C)= 2k$. Note that in a
$\frac{d}{d-2}$-orientation of $M$, the sum of the weights of the
half-edges incident to vertices of $C$ is $dk$ and the sum of the
half-edges that are on $C$ is $(d-2)k$. So we have
$\gamma_L(C)+\gamma_R(C)=dk-(d-2)k= 2k$.  Thus a
$\frac{d}{d-2}$-orientation of $M$ is a {biased orientation w.r.t.~$C$
  if and only if all the half-edges incident to the left side of $C$
  have weight $0$.

The goal of this section is to prove the following lemma:
  
\begin{lemma}
\label{lem:biased}
Let $M$ be a $d$-toroidal map and $C$ a non-contractible cycle of $M$ that is shortest
 in its homotopy class and is given with a traversal direction. 
Then $M$ admits a biased orientation w.r.t.~$C$.
\end{lemma}

To prove Lemma~\ref{lem:biased} we need to introduce some more
general terminology concerning $\alpha$-orientations.

If $S$ is a subset of vertices of a graph $M$, then $E[S]$ denotes the
set of edges of $M$ with both ends in $S$. We need the
following lemma\footnote{This lemma can be
seen as an application of Hall's theorem regarding the existence of a 
perfect matching in the bipartite graph obtain from $G$ by copying
$\beta(e)$ times each edge $e$, then subdividing once each edge of the
resulting graph, and finally copying $\alpha(v)$ times each initial
vertex of $G$.} from~\cite{BF12}: 

\begin{lemma}[\cite{BF12}]
\label{lem:demandgeneral}
A graph $G$ admits an $\frac{\alpha}{\beta}$-orientation if and only if
$\sum_{e\in E(G)}\beta(e)=\sum_{v\in V(G)}\alpha(v)$, and, for every
subset of vertices $S$ of $G$, we have
$\sum_{e\in E[S]}\beta(e)\leq \sum_{v\in S}\alpha(v)$.
\end{lemma}

Consider a non-contractible cycle $C$ of $M$ that is a shortest cycle
in its class of homotopy and given with a traversal direction.
Consider the annular map $A$ obtained from $M$ by cutting $M$ along $C$
and open it as a planar map where vertices of $C$ are duplicated to
form the outer face and a \emph{special} inner face of $A$. Without loss of generality,
we assume that $A$ is represented such that the special inner 
face is on the left side of $C$.  Let
$\alpha:V(A)\to \mb{N}$ be such that $\alpha(v)=0$ if $v$ is an
outer-vertex of $A$ and $\alpha(v)=d$ otherwise.  Let
$\beta:E(A)\to \mb{N}$ be such that $\beta(e)=0$ if $e$ is an
outer-edge of $A$ and $\beta(e)=(d-2)$ otherwise.  Then one can
transform any $\frac{\alpha}{\beta}$-orientation of $A$ to a biased
orientation of $M$ by gluing back the two copies of $C$ and giving to
the half-edges of $C$ the weight they have on the special face of $A$.
Indeed, it is clear by the definition of $A$ and the choice of
$\alpha,\beta$, that in the obtained $\frac{d}{d-2}$-orientation of
$M$ all the weights on half-edges incident to the left side of $C$ are
equal to $0$, and thus the orientation is biased w.r.t.~$C$ by the above
discussion.  So the existence of a biased orientation
(Lemma~\ref{lem:biased}), is reduced to the existence of an 
$\frac{\alpha}{\beta}$-orientation of $A$. 
It is proved in Theorem~24 of~\cite{BF12b} that $A$ admits an 
$\frac{\alpha}{\beta}$-orientation, where the proof is done first in the bipartite case
(case of even $d$) 
using Lemma~\ref{lem:demandgeneral}, and then the general case is derived
from the bipartite case using a subdivision argument.  
We reproduce here in the general case the arguments given in~\cite{BF12b} for the 
bipartite case, for the sake of completeness and 
 since this is one of the key ingredients to obtain a balanced orientation of $M$.

\begin{lemma}[Theorem~24 in~\cite{BF12b}]
\label{lem:annular}
  The annular map $A$ admits an $\frac{\alpha}{\beta}$-orientation.
\end{lemma}

\begin{proof}
  It is not difficult to check that by Euler formula that the first
  condition of Lemma~\ref{lem:demandgeneral} is satisfied. Let us now
  prove that the second condition of the lemma is also satisfied.

  Let $S$ be any subset of vertices of $A$. Suppose first that $A[S]$,
  the subgraph of $A$ induced by $S$, is connected. We consider two
  cases whether $S$ contains some outer vertices of $A$ or not.

 \begin{itemize}
 \item \emph{$S$ contains at least one outer vertex of $A$:} 

   Let $S'$ be the set of vertices obtained by adding to $S$ all the
   outer vertices of $A$. Since $\alpha$ equals to $0$ for outer
   vertices, we have
   $\sum_{v\in S}\alpha(v)= \sum_{v\in S'}\alpha(v)$. Moreover, $E[S]$
   is a subset of $E[S']$, so
   $\sum_{e\in E[S]}\beta(e)\leq \sum_{e\in E[S']}\beta(e)$.

   Let $n',m',f'$ be the number of vertices, edges and faces of
   $A'=A[S']$.  Euler's formula says that $n'-m'+f'=2$. The outer face
   of $A'$ has size $k$. Since $C$ is a
   shortest cycle in its class of homotopy, the inner face of $A'$
   containing the special face of $A$ has size at least $k$.  Moreover
   $M$ is a $d$-angulation, so all the other inner faces of $A$ have
   size at least $d$. So finally   $2m'\geq d\,(f'-2)+2k$.
   By combining the two (in)equalities, we obtain
   $d\,n'-(d-2)\,m'-2k\geq 0$. So
   $\sum_{v\in S}\alpha(v)-\sum_{e\in E[S]}\beta(e)\geq \sum_{v\in
     S'}\alpha(v)-\sum_{e\in E[S']}\beta(e)=d(n'-k)-(d-2)(m'-k)\geq
   0$.

 \item \emph{$S$ does not contain any outer vertices of $A$:} 

Let $n',m',f'$ be the
   number of vertices, edges and faces of $A'=A[S]$. Then Euler's
   formula says that $n'-m'+f'=2$. The planar map $A'$ has at most two faces
   that can be of size strictly less than $d$: its outer face, and the face
   of $A'$ containing the special face of $A$. Note that these two faces are not necessarily
   distinct and can  also be of size more than $d$. In any case
   we have  $2m'> d\,(f'-2)$.  By combining the two (in)equalities, we
   obtain $d\,n'-(d-2)\,m'> 0$. So
   $\sum_{v\in S}\alpha(v)-\sum_{e\in E[S]}\beta(e)=dn'-(d-2)m'>
   0$.
 \end{itemize}

 In both cases, the second condition of Lemma~\ref{lem:demandgeneral}
 is satisfied when $A[S]$ is connected.  If $A[S]$ is not connected,
 then we can sum over the different connected components to obtain the
 result.
\end{proof}

By the above remarks, Lemma~\ref{lem:annular} implies Lemma~\ref{lem:biased}. 

\subsubsection{Linear combinations of biased orientations}

Consider a $d$-toroidal map $M$ and $B_1,B_2$ two non-contractible
non-homotopic cycles of $M$ that are both shortest cycles in their
respective class of homotopy. Suppose that $B_1,B_2$ are given with a
traversal direction. Let $k_1$ (resp. $k_2$) be the length of $B_1$
(resp. $B_2$).

Consider $D_1, D_2, D_3, D_4$ the four $\frac{d}{d-2}$-orientations of
$M$ that are biased with respect to~$B_1, -B_1, B_2, -B_2$ respectively.  The
$\gamma$-score of $B_1,B_2$ in these four orientations are given in  Table~\ref{tab:gamma} where $a,b$ are
integers in $\{-2k_2,\ldots,2k_2\}$ and $c,d$ are integers in 
$\{-2k_1,\ldots,2k_1\}$.

\begin{table}[!h]
\center
\begin{tabular}{|c|c|c|c|c|}
\hline
  & $D_1$ & $D_2$ & $D_3$ & $D_4$ \\ \hline
$\gamma(B_1)$ & $2k_1$ & $-2k_1$ & $c$ & $d$ \\\hline
$\gamma(B_2)$ & $a$ & $b$ & $2k_2$ & $-2k_2$ \\ \hline
\end{tabular}
\caption{$\gamma$-score of the orientations $D_1, D_2, D_3, D_4$.}
\label{tab:gamma}
\end{table}

For $1\leq i\leq 4$, let ${ w_i}$ be the \emph{weight function} of
$D_i$. i.e.,~the function defined on the half-edges of $M$ such that
the weight of a half-edge $h$ is ${ w_i}(h)$ in the
$\frac{d}{d-2}$-orientation $D_i$. Let $k=2k_1k_2$. Let ${ w}$ be the
weight function defined on the set of half-edges of $M$  by the
following:

$${ w}=\left\lbrace
\begin{array}{ll}
(2k+bc)k_2\, \times\, { w_1}+(2k-ac)k_2\, \times\,  { w_2}-(a+b)k\, \times\,  { w_3} & \mbox{if $a+b<0$}\\
{ w_1}+{ w_2} & \mbox{if $a+b=0$}\\
(2k-bd)k_2\, \times\,  { w_1}+(2k+ad)k_2\, \times\,  { w_2}+(a+b)k\, \times\,  { w_4} & \mbox{if $a+b>0$}\\
\end{array}
\right.$$

Note that in all cases, with weight function $w$, the $\gamma$-score
of both $B_1,B_2$ is zero. Indeed, we have:
$$
\begin{bmatrix}
  (2k+bc)k_2 & (2k-ac)k_2 & -(a+b)k & 0\\
  1 & 1 & 0 & 0 \\
  (2k-bd)k_2 & (2k+ad)k_2 & 0 & (a+b)k \\
\end{bmatrix}
\times
\begin{bmatrix}
2k_1 & a \\
-2k_1 & b \\
c & 2k_2 \\
d & -2k_2
  \end{bmatrix}
  =
  \begin{bmatrix}
    0 & 0 \\
    0 & a+b \\
    0 & 0 \\
    \end{bmatrix}.
  $$

Note also that in all cases, for $1\leq i\leq 4$ the coefficient of $w_i$ is in $\mathbb{N}$,   
hence $w(h)\in\mathbb{N}$ for every half-edge $h$ of $M$.  
We denote by $\sigma$ the
sum of the coefficients, i.e.,
$$\sigma=\left\lbrace
\begin{array}{ll}
(2k+bc)k_2+(2k-ac)k_2-(a+b)k & \mbox{if $a+b<0$},\\
2 & \mbox{if $a+b=0$},\\
(2k-bd)k_2+(2k+ad)k_2+(a+b)k & \mbox{if $a+b>0$}.\\
\end{array}
\right.$$
Note that $\sigma\geq 1$ in all cases. 

Then the total $w$-weight at any vertex (resp. edge) of $M$ equals
$\sigma d$ (resp. $\sigma(d-2)$). Hence $w$ is the weight function of
a $\frac{\sigma d}{\sigma (d-2)}$-orientation $D^\sigma$ of $M$. In a
sense ${D^\sigma}/{\sigma}$, obtained from  $D^\sigma$ by dividing all
the weights by $\sigma$,
is a $\frac{d}{d-2}$-orientation of $M$
but with rational weights instead of integers. Note that the proof of
Lemma~\ref{lem:gamma0all} is not using the fact that the weights are
integers thus the conclusion holds with rational weights as well.

We have defined the linear combination of biased orientations 
in such a way that we precisely have $\gamma(B_1)=\gamma(B_2)=0$ for the
orientation $D^{\sigma}$.
A variant of
Lemma~\ref{lem:gamma0all} with rational weights
 implies that $D^{\sigma}$ is a balanced
$\frac{\sigma d}{\sigma (d-2)}$-orientations and
${D^\sigma}/{\sigma}$ can be viewed as a balanced
$\frac{d}{d-2}$-orientation of $M$ but with rational weights.  So we
almost have what we are looking for, except for the rational weights
that we would like to be integers.

\subsubsection{Integrality by minimality}

We use the same terminology as in the previous subsection.

Let $M$ be a $d$-toroidal map, with a distinguished face $f_0$.  By
Corollary~\ref{theo:general_gamma}, the map $M$ has a  unique minimal
$\frac{\sigma d}{\sigma(d-2)}$-orientation $D_{\min}^\sigma$ that is
$\gamma$-equivalent to $D^{\sigma}$, i.e. that is balanced. In the  next
lemma, we now prove that
the weights of  $D_{\min}^\sigma$   are multiple of $\sigma$. So
${D_{\min}^\sigma}/{\sigma}$, obtained from  $D_{\min}^\sigma$ by
dividing all the weights by $\sigma$, is a balanced $\frac{d}{d-2}$-orientation
of $M$ with integer weights and thus this proves Proposition~\ref{th:existencebalanced}.

\begin{lemma}
\label{lem:multiplesigma}
  All the weights of $D_{\min}^\sigma$ are multiples of $\sigma$.
\end{lemma}

\begin{proof}
Since the total weight of an edge is a multiple of $\sigma$, 
for each edge $e\in M$ either its two half-edges are not multiple 
of $\sigma$ or they are both multiple of $\sigma$. We denote by $Q$
the set of edges with weights (on both half-edges) not multiple 
of $\sigma$, and let $M_Q$ be the embedded graph induced by edges in $Q$
and their incident vertices. Note that
  $M_Q$ is embedded on the torus but is not necessarily a
  map as some of its faces may not be homeomorphic to an open disk.
  Since the total weight at any vertex is a multiple of $\sigma$, a vertex of $M$ 
can not be incident to a single edge in $Q$, hence all the vertices of $M_Q$ 
 have degree at least $2$. 

 Suppose by contradiction, that $M_Q$ has at least two faces (the embedded subgraph $M_Q$ is not 
necessarily a map, a `face' refers here to a connected component of the torus cut by $M_Q$). 
  Let $f$ be a face of $M_Q$ not containing $f_0$.  Let $F$ be
  the set of edges on the border of $f$. The weights of the half-edges
  of $F$ are not multiple of $\sigma$. So none of their weights is
  equal to $0$. So in the underlying biorientation of $M$, all edges
  of $F$ are bioriented. Thus, in the $\sigma (d-2)$-expansion $H$ of
  $M$, the set $S$ of faces of $H$ within $f$ is such that every edge on the boundary of $S$ has a face in $S$ on its right, contradicting the minimality of $D_{\min}^\sigma$. So $M_Q$ has a unique face.

  Since the vertices of $M_Q$ have degree at least $2$ and $M_Q$ has a
  unique face, the embedded toroidal graph $M_Q$ has to be one of the graphs depicted
  in Figure~\ref{fig:threecases}, i.e. it is either a non-contractible
  cycle, or the union of two non-contractible cycles that are
  edge-disjoint and intersect at a unique vertex, or it is the union of
  three  edge-disjoint paths such that the union of any two of
  these paths forms a non-contractible cycle. 

  In any case, there exists a non-contractible cycle $C$ of $M$ such
  that on each side of $C$ there is a single incident half-edge in
  $Q$.  This implies that the sum of the weights of incident
  half-edges on the left (resp. right) side of $C$ is not a multiple
  of $\sigma$.

\begin{figure}
\center
\includegraphics[scale=0.3]{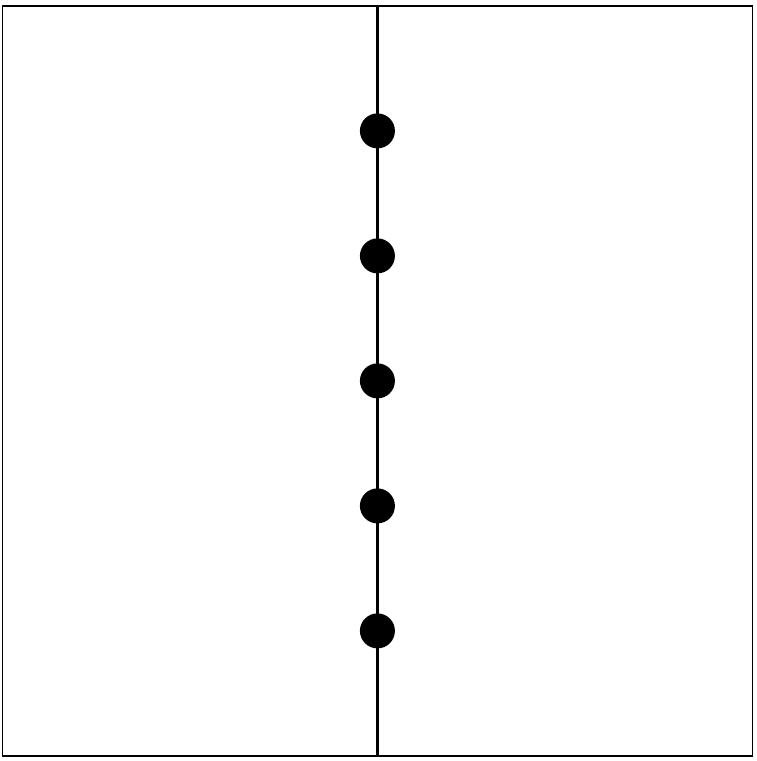} \
\ \includegraphics[scale=0.3]{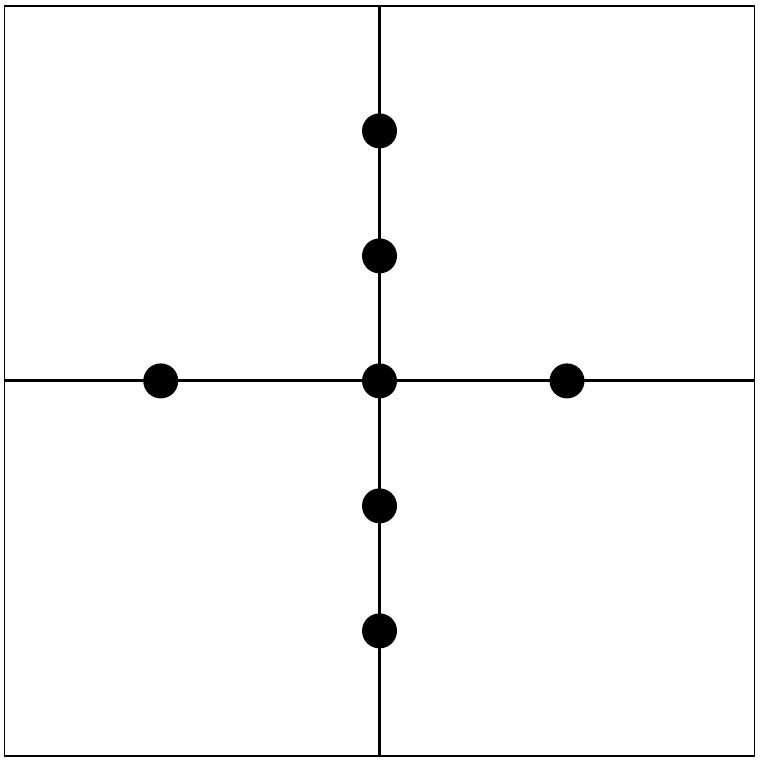} \
\ \includegraphics[scale=0.3]{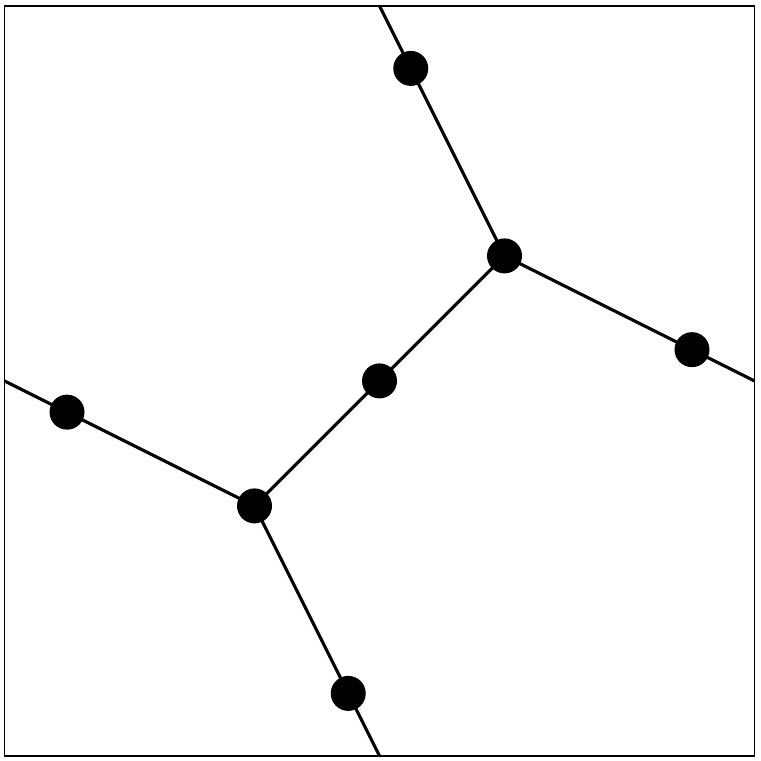}
\caption{The three possible cases for  $M_Q$.}
\label{fig:threecases}
\end{figure}

On the other hand, since $D_{\min}^\sigma$ is
  balanced we have $\gamma (C)=0$. Let $\ell$ be the length of $C$, so that 
  the sum of the weights of the half-edges of $M$ incident to each
  side of $C$ is equal to 
$\frac{1}{2}(\sigma d \ell-\sigma (d-2) \ell) =\sigma \ell$.
This is  a multiple of
   $\sigma$, giving a contradiction.
\end{proof}

\subsection{Bipartite case}
\label{sec:evencase}

  For the particular case where $d$ is even and the map is bipartite, we
can prove the existence of balanced orientations with even weights, as
discussed below.

Consider a $d$-toroidal map $M$ where $d$ is even, i.e. $d=2b$ with
$b\geq 2$. Note that $d/2=b$ and $(d-2)/2=b-1$.  So if the
weights of a $\frac{d}{d-2}$-orientation of $M$ are even, they can be
divided by two to obtain a \emph{$\frac{b}{b-1}$-orientation} of $M$, i.e., an
$\NN$-biorientation where every vertex has weight $b$ and every edge has weight $b-1$. 
Then one can ask the question of existence
of balanced $\frac{b}{b-1}$-orientations in that case.  The answer to
this question is given as follows.

\begin{proposition}
  \label{th:evencase}
  When $d=2b$ with $b\geq 2$, a toroidal $d$-angulation $M$ with
  essential girth $d$ admits a balanced $\frac{b}{b-1}$-orientation if
  and only if $M$ is bipartite.  In this case, for any choice of a
  distinguished face $f_0$ of $M$, the unique balanced
  $\frac{d}{d-2}$-orientation that is minimal has all its weights that
  are even (hence is a balanced $\frac{b}{b-1}$-orientation upon
  dividing the weights by $2$).
\end{proposition}

\begin{proof}
  If $M$ admits a balanced $\frac{b}{b-1}$-orientation we want to show
  that $M$ is bipartite. Since the face-degrees of $M$ are even it is
  enough to check that every non-contractible cycle $C$ of $M$ has
  even length.  Recall that $\gamma_R(C)$ (resp. $\gamma_{L}(C)$) is
  the sum of the weights of the half-edges incident to the right
  (resp. left) side of $C$. Since the orientation is balanced, we have
  $\gamma_R(C)=\gamma_L(C)$. Denoting by $k$ the length of $C$, the
  sum of the weights of all the half-edges of $C$ is equal to
  $(b-1)k$. The sum of the weights of all the half-edges incident to
  vertices of $C$ is $bk$. Hence
  $bk=(b-1)k+\gamma_R(C)+\gamma_L(C)=(b-1)k+2\gamma_R(C)$.  So
  $k=2\gamma_R(C)$ and thus $k$ is even.

  Now suppose that $M$ is bipartite, and consider an
  arbitrary face $f_0$ of $M$. By Proposition~\ref{th:existencebalanced},
   $M$ admits a balanced $\frac{d}{d-2}$-orientation. By
  Corollary~\ref{theo:general_gamma} we can consider the unique minimal
  $\frac{d}{d-2}$-orientations $D$ that is balanced.  We have the following:

  \begin{claimn}
\label{cl:even}
The weights of $D$ are  even.    
  \end{claimn}

  \begin{proofclaim}
The proof follows the same arguments as the proof of
Lemma~\ref{lem:multiplesigma}. Since each edge has even total weight, either 
its two half-edges have both even weights, or they have both odd weights. We let $Q$ be 
the set of edges with odd weights, and assume for contradiction that $Q$ is not empty.
Let  $M_Q$ be the embedded graph induced by the edges in $Q$ 
and their incident vertices. Since every vertex has even total weight, it can not be incident
to a single edge in $Q$, hence all vertices of $M_Q$ have degree at least~$2$. 

  Suppose by contradiction that $M_Q$ has at least two faces.
  Let $f$ be a face of $M_Q$ not containing $f_0$.  Let $F$ be
  the set of edges on the border of $f$. The weights of the half-edges
  of $F$ are odd, hence non-zero. Hence, in the
  underlying biorientation of $M$, all edges of $F$ are
  bioriented. Thus, in the $(d-2)$-expansion $H$ of $M$, the set of
  faces $S$ of $H$ corresponding to $f$ is such that every edge on the
  boundary of $S$ has a face in $S$ on its right, contradicting the
  minimality of $D$.  So $M_Q$ has a unique face.

  Since the vertices of $M_Q$ have degree at least $2$ and $M_Q$ has a
  unique face, we again have the property that $M_Q$ is in one of the
  configurations shown in Figure~\ref{fig:threecases}. In any case,
  there exists a non-contractible cycle $C$ of $M$ that has a single
  edge in $Q$ on each side. Hence $\gamma_L(C)$ is odd. On the other
  hand, since each edge has weight $d-2$ and each vertex has weight
  $d$, we have $\gamma_L(C)+\gamma_R(C)=2\ell$, with $\ell$ the length
  of $C$.  Since the orientation is balanced, we have
  $\gamma_L(C)=\ell$; and since the map is bipartite $\ell$ is even,
  contradicting the fact that $\gamma_L(C)$ is odd.
\end{proofclaim}

The claim ensures that  
all the weights of $D$ are even. Thus, dividing all the weights of $D$
by $2$, one obtains a balanced $\frac{b}{(b-1)}$-orientation of $M$.
\end{proof}

\section{Bijective results}
In this section we state our main bijective results. Similarly as in the planar case~\cite{BF12,BF12b}, our starting point is 
a `meta-bijection' $\Phi_+$ in any genus $g$ between a family of oriented maps and a family of decorated unicellular maps. 
The families are defined in Section~\ref{sec:terminology} and $\Phi_+$ is presented in Section~\ref{sec:bijPhi+} in the oriented setting, and then extended in Section~\ref{sec:weightbi}
} to the weighted bioriented setting. In Section~\ref{sec:bijdangul} we then specialize $\Phi_+$ to the balanced $\frac{d}{d-2}$-orientations studied in Section~\ref{sec:balanced}, and obtain a bijection for toroidal $d$-angulations of essential girth $d$ (Theorem~\ref{theo:bij_dang}) which admits a parity specialization in the bipartite case (Corollary~\ref{theo:bij_bip_2bang}). Each of these two bijections can be further extended to a bijection for toroidal maps of fixed essential girth with a certain root-face condition (Theorem~\ref{theo:bij_maps_d}, and Theorem~\ref{theo:bij_maps_2b} in the bipartite case, both stated in Section~\ref{sec:bij_extended} without proofs, which are delayed to Section~\ref{sec:proofs_theorems}).

\subsection{Terminology for oriented maps and mobiles}\label{sec:terminology}

Consider a face-rooted map $M$ of genus $g\geq 0$. Suppose that $M$ is
given with an orientation of its edges such that every vertex has at
least one outgoing edge.  For an edge $e\in M$, the \emph{rightmost
  walk} starting from $e$, is the (necessarily unique and eventually
looping) walk starting from $e$ by following the orientation of $e$,
then taking at each step the rightmost outgoing edge, i.e., for any
pair $e',e''$ of consecutive edges along the walk, all edges between
$e'$ and $e''$ in \ccw order around their common vertex are ingoing.

An orientation  of $M$ is called a \emph{right orientation} if the
following conditions are satisfied:

\begin{itemize}
\item
every vertex has at least one outgoing edge,
\item
  for every edge $e$ of $M$, the rightmost walk starting from $e$ eventually loops
  on the contour of the root-face $f_0$ with $f_0$ on its right side.
\end{itemize}

For $d\geq 1$ and $g\geq 0$,  
we denote by $\cOgd$ the family of  right orientations
of face-rooted maps of genus $g$ whose root-face has degree $d$. 

Let us now define the unicellular maps 
to be set in bijective correspondence with $\cOgd$. 
A \emph{mobile of genus $g$} is defined as a 
unicellular map of genus $g$ that is bipartite (it has black vertices
and white vertices and every edge connects a black 
vertex to a white vertex) such that each corner at a black
vertex is allowed to carry additional dangling half-edges called
\emph{buds}, represented as outgoing arrows. 
The \emph{excess} of a mobile $T$ is the number of edges 
minus the number of buds in $T$, and the family
of mobiles of genus $g$ and excess $d$ is denoted by $\cMgd$.

\subsection{Bijection $\Phi_+$ between $\cOgd$ and $\cMgd$.}\label{sec:bijPhi+}

Let $d\geq 1$. Similarly as in the planar case developed in~\cite{BF12}
we adapt the bijection from~\cite{BC11} into a
bijection\footnote{It should also be possible, for any $d\leq
    0$,  
to adapt the bijection
from~\cite{BC11} into a bijection between
 the family of genus $g$ mobiles of 
 excess $d$
 and
a 
 well-characterized family of genus $g$ oriented map, but we will not need it here to get our bijections
for toroidal maps with prescribed essential girth.}   
between $\cOgd$ and $\cMgd$ (see Section~\ref{sec:proof_phi} for proof 
details). For $O\in\cOgd$ we denote by $\Phi_+(O)$
the embedded graph obtained by inserting a black vertex in each
face of $O$, then applying 
the local rule of Figure~\ref{fig:local_rule} to 
every edge of $O$ (thereby creating an edge and a bud), and finally
erasing the  isolated black vertex
in the root-face of $O$ (since the root-face contour is directed
clockwise, this black vertex is incident to $d$ buds and no edge). 
See Figure~\ref{fig:bij_ori} for an example. 

\begin{figure}[!h]
\begin{center}
\includegraphics[width=4cm]{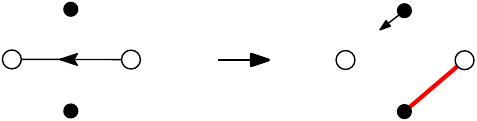}
\end{center}
\caption{The local rule applied by the bijection $\Phi_+$ 
 to each edge.}
\label{fig:local_rule}
\end{figure}
\begin{figure}[!h]
\begin{center}
\includegraphics[width=12cm]{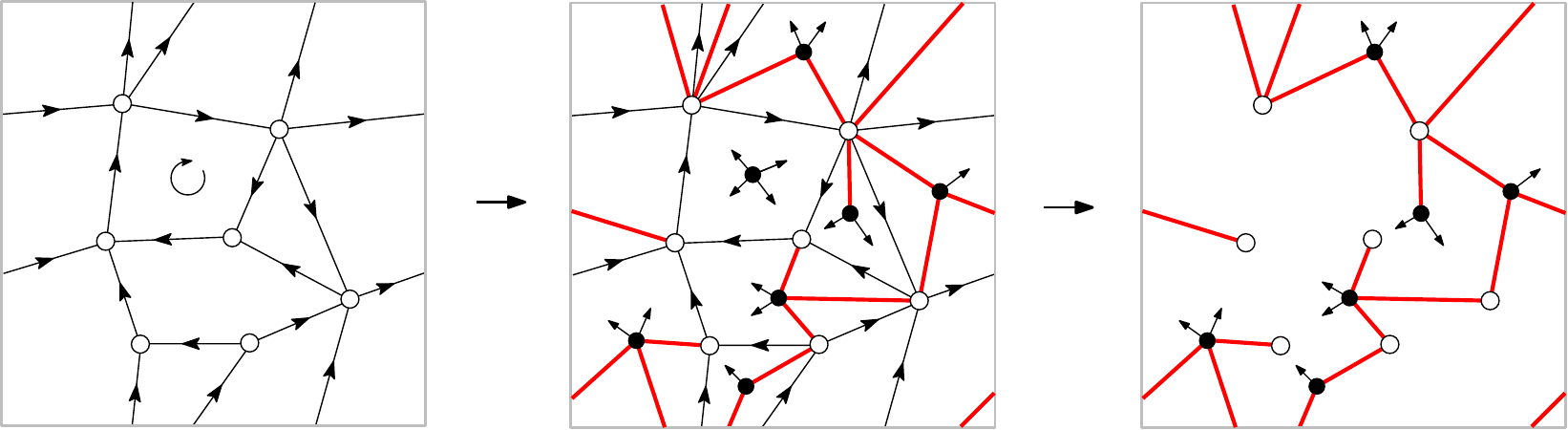}
\end{center}
\caption{The bijection $\Phi_+$  from a toroidal orientation
 in $\cO_4^{1}$
to a toroidal mobile of excess $4$ (the root-face is indicated by the small clockwise circular arrow).}
\label{fig:bij_ori}
\end{figure}

\begin{theorem}[Oriented case]\label{theo:phi_ori}
For $d\geq 1$ and $g\geq 0$, 
the mapping $\Phi_+$, 
with the local rule of Figure~\ref{fig:local_rule},
is a bijection between the family $\cO_d^{g}$ of 
oriented maps and the family $\cMgd$ of mobiles.
\end{theorem}

The proof of Theorem~\ref{theo:phi_ori} is delayed to
Section~\ref{sec:proof_phi}.
 
The inverse mapping $\Psi_+$ is done as follows. Starting from a
mobile $T\in\cMgd$, we insert an \emph{ingoing bud} in every corner of
a black vertex $u$ that is just after an edge (not a bud) in
counterclockwise order around $u$.  Since $T$ has excess $d$, there
are $d$ more ingoing buds than outgoing buds. We then match the
outgoing and ingoing buds according to a walk (with the face on our
right) around the unique face of $T$, considering outgoing buds as
opening parentheses and ingoing buds as closing parentheses. Every
matched pair yields a directed edge, and we are left with $d$
unmatched ingoing buds (all in the same face
of the obtained figure), which we call the \emph{exposed buds}  
of $T$. For each such bud, the consecutive half-edge in clockwise order
around the incident black vertex is called an \emph{exposed half-edge} of $T$.  

We then join the exposed buds to a newly created vertex
$v_{\infty}$, see Figure~\ref{fig:closure} for an example.  Let $X$ be
the oriented map obtained after erasing the edges of $T$ and the white
vertices; and let $O$ be the dual map endowed with the face-rooted
dual orientation (that is, for every edge $e\in O$, with $e^*\in X$
the dual edge, we orient $e$ from the left side of $e^*$ to the right
side of $e^*$), where the root-face is taken to be the face dual to
$v_{\infty}$.  Then $\Psi_+$ is the mapping that maps $T$ to $O$ (it
is quite easy to check that $\Phi_+(\Psi_+(T))=T$ when superimposing
$O$, $X$ and $T$).

\begin{figure}[!h]
\begin{center}
\includegraphics[width=\linewidth]{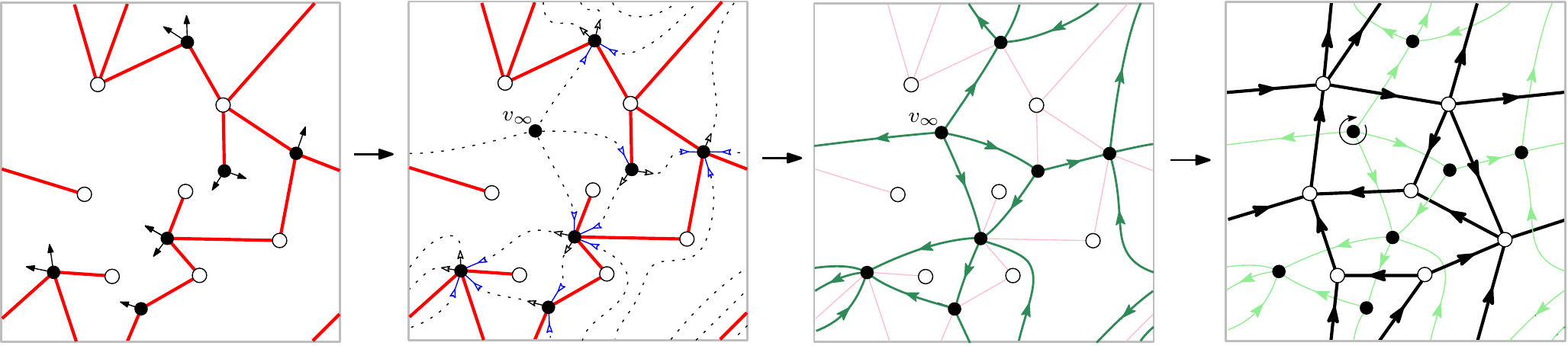}
\end{center}
\caption{The inverse mapping $\Psi_+$: from a mobile in $\cT_4^{1}$  
to an orientation in $\cO_4^{1}$.} 
\label{fig:closure}
\end{figure}

\subsection{Extension of $\Phi_+$ to the weighted bioriented setting}
\label{sec:weightbi}

Similarly as in~\cite{BF12} we may now extend this bijection to 
the context of biorientations, and then to weighted
biorientations. Recall from Section~\ref{sec:preliminaries}, that in a bioriented map $M$, 
every half-edge receives a direction (ingoing or outgoing).
For $i\in\{0,1,2\}$ an edge is said to be \emph{$i$-way} if it has
$i$ outgoing half-edges among its two incident half-edges.  
For $O$ a bioriented map, 
the \emph{induced} oriented map $O'=\mu(O)$ is obtained by 
replacing each $2$-way edge by a double edge (enclosing a face of degree $2$) directed counterclockwise, and inserting a vertex of (out)degree
 $2$ in the middle of each $0$-way edge, see the left column of 
Figure~\ref{fig:bij_biori} for an example.   
For $d\geq 1$ and $g\geq 0$ we can now extend 
the definition of the families $\cOgd$
to the bioriented setting:  a face-rooted bioriented
map  is said to belong to $\cOgd$ if  the induced
oriented face-rooted map is in $\cOgd$.

\begin{figure}[!t]
\begin{center}
\includegraphics[width=12cm]{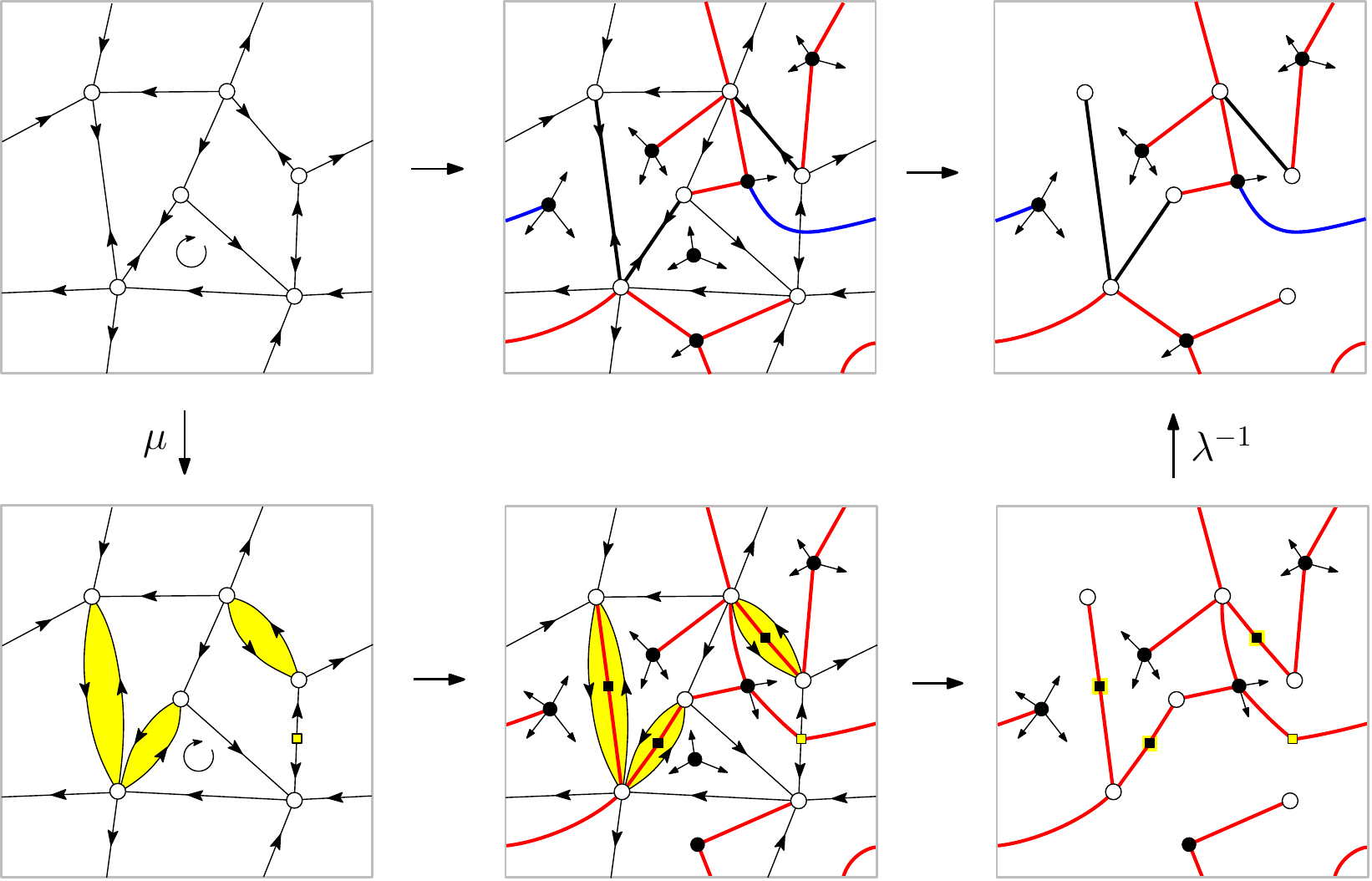}
\end{center}
\caption{Top-row: the bijection $\Phi_+$ from a biorientation
 in $\cO_3^{1}$
to a bimobile in $\cT_3^{1}$,  applying the 
local rules of Figure~\ref{fig:local_rule_biori}. The bottom-row shows 
that the construction amounts to applying the bijection $\Phi_+$ in the oriented setting, upon blowing each 2-way edge
into a \ccw 2-cycle and inserting a sink of degree $2$ in the middle of every 0-way edge.}
\label{fig:bij_biori}
\end{figure}

Let us now formulate 
rightmost walks directly on the biorientation to be a bit more explicit on the properties
that a biorientation needs to satisfy to be in $\cOgd$.  Consider a
face-rooted map $M$ of genus $g\geq 0$.  Suppose that $M$ is given
with a biorientation such that every vertex
has at least one outgoing half-edge.  For an outgoing half-edge $h$
of $M$, we define the \emph{rightmost walk} from $h$ as the (necessarily
unique and eventually looping) sequence
of half-edges starting from $h$, and at each step taking the 
 opposite half-edge and then the rightmost outgoing half-edge at the
 current vertex.
 
A biorientation  of $M$ is called a \emph{right biorientation} if the
following conditions are satisfied:

\begin{itemize}
  \item  every vertex
has at least one outgoing half-edge, 
\item
  for every outgoing half-edge $h$, the rightmost walk starting from $h$ loops
  on  the contour of the root-face $f_0$ with $f_0$ on its right
  side.
\end{itemize}

Thus with this definition, a face-rooted bioriented
map  belongs to $\cOgd$ if and only if
it is a right biorientation, it has genus $g$ and the degree of the root-face is $d$.

As illustrated in Figure~\ref{fig:bij_biori} (forgetting for now the second and third drawing of the top-row),
 $\Phi_+\circ\mu$ induces a bijection between bioriented maps
in $\cOgd$ and mobiles in $\cMgd$ where some vertices
of degree $2$ are marked as square vertices (square black 
vertices correspond to the 2-way edges, square white vertices
correspond to the 0-way edges).

We call \emph{bimobile} of genus $g$ a unicellular map of genus $g$
with two kinds of vertices, white or black (this time, black-black
edges and white-white edges are allowed), and such that each corner at
a black vertex might carry additional dangling half-edges called buds.
(Note that a mobile is a special case of bimobile, where all the edges
are black-white.)  The \emph{excess} of a bimobile is the number of
black-white edges plus twice the number of white-white edges, minus
the number of buds.  We now extend the definition of the family
$\cMgd$ to bimobiles: a bimobile of genus $g$ is said to belong to
$\cMgd$ if its excess is $d$.  For $T$ a bimobile, the \emph{induced
  mobile} $\lambda(T)$ is obtained by inserting in each white-white
edge a square black vertex of degree $2$, and inserting in each
black-black edge a square white vertex of degree $2$ together with two
buds at the incident edges, as shown in
Figure~\ref{fig:local_rule_lambda}. Clearly $\lambda(T)$ has the same
excess as $T$. As shown in Figure~\ref{fig:bij_biori} the mapping
$\lambda^{-1}\circ\Phi_+\circ\mu$ thus yields a bijection from
bioriented maps in $\cO_d^{g}$ to bimobiles in $\cMgd$ (it just
amounts to marking some \ccw faces of degree $2$ and some sinks of
degree $2$ in the bijection of Theorem~\ref{theo:phi_ori}). By a slight abuse of notation we 
refer to $\lambda^{-1}\circ\Phi_+\circ\mu$ as $\Phi_+$ (adapted to the bioriented setting). 
It is easy to see that the effect of $\lambda^{-1}$, of $\mu$, and of
 the local rules of Figure~\ref{fig:local_rule} can be
shortcut as the local rules shown in Figure~\ref{fig:local_rule_biori}
applied to the three types of edges (0-way, 1-way, or 2-way), so that,
given a biorientation $O$ in $\cO_d^{g}$, $\Phi_+(O)$ is obtained
after applying these rules to every edge of $O$, and then deleting the
isolated black vertex in the root-face. 

\begin{figure}[!h]
\begin{center}
\includegraphics[width=10cm]{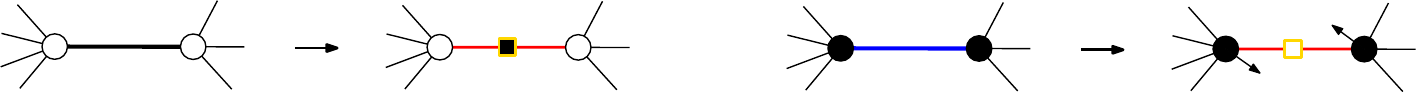}
\end{center}
\caption{The rules to obtain a mobile $\lambda(T)\in\cT_d^{g}$
 (where the new added vertices, of degree $2$, are distinguished
as square)  
from a bimobile $T\in\cT_d^{g}$.} 
\label{fig:local_rule_lambda}
\end{figure}

\begin{figure}[!h]
\begin{center}
\includegraphics[width=\linewidth]{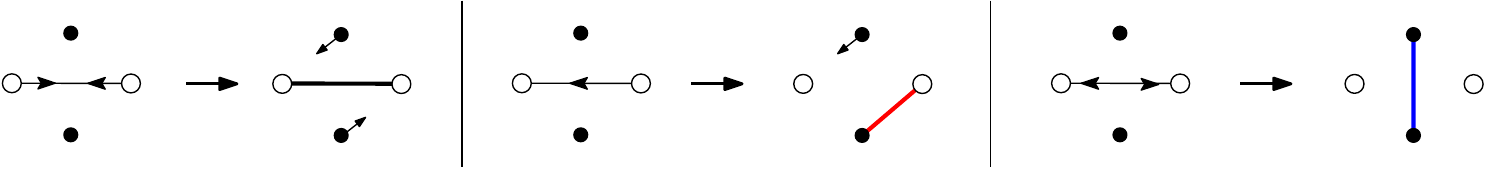}
\end{center}
\caption{The local rules applied to each edge by the bijection $\Phi_+$ 
in the bioriented setting.}
\label{fig:local_rule_biori}
\end{figure}

We obtain:

\begin{corollary}[Extension to the bioriented setting]
For $d\geq 1$ and $g\geq 0$, 
the mapping $\Phi_+$, 
with the local rules of Figure~\ref{fig:local_rule_biori},
is a bijection between the family $\cO_d^{g}$ of 
bioriented maps and the family $\cMgd$ of bimobiles.
\end{corollary}

Finally, similarly as in the planar case~\cite{BF12}, the bijection is
directly extended to the weighted setting.  A $\ZZ$-biorientation of a
map is a biorientation where every half-edge is given a value in
$\ZZ$, which is in $\ZZ_{>0}$ (strictly positive) if the half-edge is outgoing and in $\ZZ_{\leq 0}$ (negative or zero)  
if the half-edge is ingoing.  A $\ZZ$-bimobile is  a
bimobile where every non-bud half-edge is given a value in
$\ZZ$, which is in $\ZZ_{>0}$ if the half-edge is incident to a white
vertex and in $\ZZ_{\leq 0}$ if
 the (non-bud) half-edge is incident to a black vertex.
 
A $\ZZ$-bioriented face-rooted map is said to belong to $\cOgd$
if the underlying unweighted face-rooted bioriented map belongs
to $\cOgd$; and a $\ZZ$-bimobile $T$ is said to belong to 
$\cMgd$ if the underlying unweighted bimobile is in $\cMgd$.  

For a $\ZZ$-bioriented map, the \emph{weight} of a vertex $v$ 
is the sum of the weights of the outgoing half-edges at $v$,
and the \emph{weight} of a face $f$ is the sum of the weights
of the ingoing half-edges that have $f$ on their left (traversing
the half-edge toward its incident vertex); 
and the \emph{weight} of an edge $e$ is the sum of the 
weights of its two half-edges.  
For a $\ZZ$-bimobile, the \emph{weight} of a vertex $v$
is the sum of the weights of the incident half-edges,
and the \emph{weight} of an edge $e$ is the sum of the weights
of its two half-edges.  
We extend the bijection $\Phi_+$ to the weighted bioriented setting by
the rules of Figure~\ref{fig:local_rule_weighted_biori}.

\begin{figure}[!h]
\begin{center}
\includegraphics[width=\linewidth]{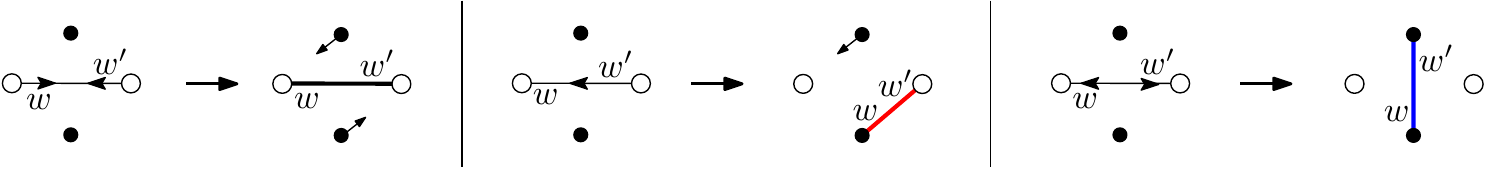}
\end{center}
\caption{The local rules applied to each edge by the bijection $\Phi_+$ 
in the weighted bioriented setting.}
\label{fig:local_rule_weighted_biori}
\end{figure}

Then we obtain the following:

\begin{corollary}[Extension to the weighted bioriented setting]\label{theo:bij_weighted_biori}
For $d\geq 1$ and $g\geq 0$, 
the mapping $\Phi_+$, 
with the local rules shown in Figure~\ref{fig:local_rule_weighted_biori},
is a bijection between the family $\cO_d^{g}$ of 
$\ZZ$-bioriented maps and the family $\cMgd$ of $\ZZ$-bimobiles.
\end{corollary}

An example is given in Figure~\ref{fig:bij_weighted_biori} (the weights are omitted in the middle drawing).

\begin{figure}[!h]
\begin{center}
\includegraphics[width=12cm]{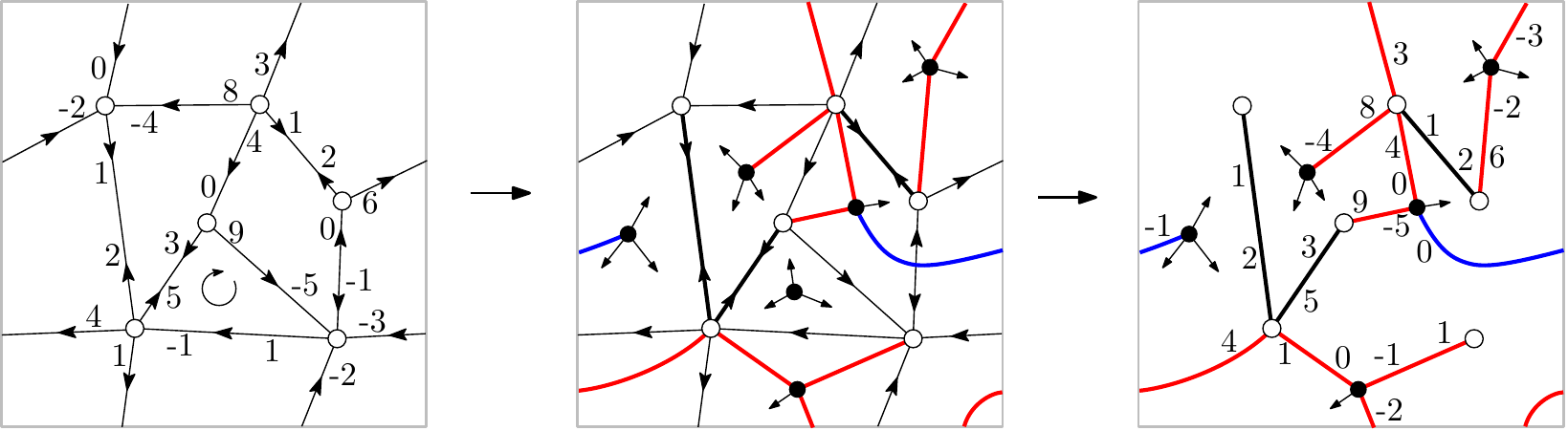}
\end{center}
\caption{Example of the bijection $\Phi_+$ from a $\ZZ$-biorientation
 in $\cO_3^{1}$
to a $\ZZ$-bimobile in $\cT_3^{1}$.}
\label{fig:bij_weighted_biori}
\end{figure}

As in the planar case~\cite{BF12}, for $O$ a $\ZZ$-bioriented map in
$\cOgd$ and $T=\Phi_+(O)$ the corresponding $\ZZ$-bimobile, several
parameters can be traced:
\begin{itemize}
\item
each vertex $v$ of $O$ corresponds to a white vertex $w$ of $T$:  
 the outdegree of $v$ corresponds to the degree of $w$ and the weight of $v$ is the same as the weight of $w$,
\item
each non-root face $f$ of $O$ corresponds to a black vertex $b$
of $T$ of the same degree and same weight,
\item
each edge $e$ of $O$ corresponds to an edge of $T$ of the same weight.
\end{itemize}

When all the weights of a $\ZZ$-biorientation are in $\ZZ_{\geq 0}$ then we
have an $\NN$-biorientation, as defined in
Section~\ref{sec:preliminaries}.  Note that an $\NN$-biorientation is
a $\ZZ$-biorientation where all the ingoing half-edges have weight
$0$.  The corresponding $\ZZ$-bimobiles are called
\emph{$\NN$-bimobiles}  (these are the $\ZZ$-bimobiles where the half-edges at black vertices have weight
$0$).

We will use specializations of the  weighted formulation of
$\Phi_+$ (Corollary~\ref{theo:bij_weighted_biori}) in order to obtain bijections for 
$d$-toroidal maps (relying on $\frac{d}{d-2}$-orientations, so we are in the 
$\NN$-bioriented setting), and 
more generally for toroidal maps of essential 
girth $d$ with a root-face of degree $d$ (relying on a generalization of $\frac{d}{d-2}$-orientations in the $\ZZ$-bioriented setting). 

\subsection{Necessary condition for $\ab$-orientations to be  right biorientations}

We prove here that minimality is a necessary condition for an $\ab$-orientation to be a right biorientation:

\begin{lemma}
  \label{lem:necmin}
  If a face-rooted $\ab$-oriented map belongs to $\cOgd$, then it is
  minimal.
\end{lemma}
\begin{proof}
  Suppose by contradiction that a face-rooted map $M$
  has an $\ab$-orientation $X$ in $\cO_d^{g}$ that is non-minimal.  Let
  $f_0$ be the root face of $M$.  
By  definition of non-minimality, there exists a non-empty set $S$ of faces of $M$, not containing
  $f_0$, such that every edge on the boundary of $S$ is either simply directed with a face in
  $S$ on its right or is bidirected. 
Hence,  while walking \cw on the contour of $S$, each half-edge that is
  encountered just after a vertex is outgoing.  Consider such a half-edge
  $h$ and let $W$ be the rightmost walk starting from $h$.  Then $W$
  necessarily stays in $S$ union its contour (it can not escape), and moreover if it loops on
  the contour of $S$, then it does so with $S$ on its right side.
Since $S$ does not contain $f_0$, 
 $W$ can not eventually loop on the contour of $f_0$ with $f_0$
  on the right side, a contradiction.
\end{proof}

\subsection{Bijection for toroidal $d$-angulations of essential
  girth~$d$}
\label{sec:bijdangul}

Let $d\geq 3$.   We define a toroidal $\frac{d}{d-2}$-mobile as an 
$\NN$-bimobile of genus $1$, where every white vertex has weight $d$,
every edge has weight $d-2$ and every black vertex has degree~$d$. 
We denote by $\mathcal U_d$ the family of these $\NN$-bimobiles.  
(Note that there is no black-black edges in an element of  $\mathcal U_d$.)
A simple counting argument gives:

\begin{lemma}\label{claim:excess}
Every $\NN$-bimobile in $\mathcal U_d$ has excess $d$.
\end{lemma}
\begin{proof}
For $T\in \mathcal U_d$, let $\nbw$ be the number of black-white edges, $\nww$ the number
of white-white edges, $e=\nbw+\nww$ the total number of edges, 
$\nb$ the number of black vertices, $\nw$ the number of white vertices, and $k$ the number of buds. 
By definition the excess
of $T$ is $\nbw+2\nww-k$, so we want to prove that this quantity equals $d$. 
Since $T$ is unicellular, Euler's formula gives $e=\nb+\nw+1$. Since every 
white vertex has weight $d$, every black vertex has weight $0$, and every edge has weight $d-2$
we have $d\nw=(d-2)e$. Since every black vertex has degree $d$ we have $d\nb=\nbw+k$. 
Hence we have at the same time $d(\nw+\nb)=(d-2)e+\nbw+k$ and $d(\nw+\nb)=de-d$, 
so that $2e=k+\nbw+d$, and thus $\nbw+2\nww-k=d$.  
\end{proof}

Clearly the bijection $\Phi_+$ specializes as a bijection between
face-rooted toroidal $d$-angulations endowed with a
$\frac{d}{d-2}$-orientation in $\cO_d^{1}$, and the family
$\mathcal U_d$. 

\begin{figure}
\begin{center}
\includegraphics[width=\linewidth]{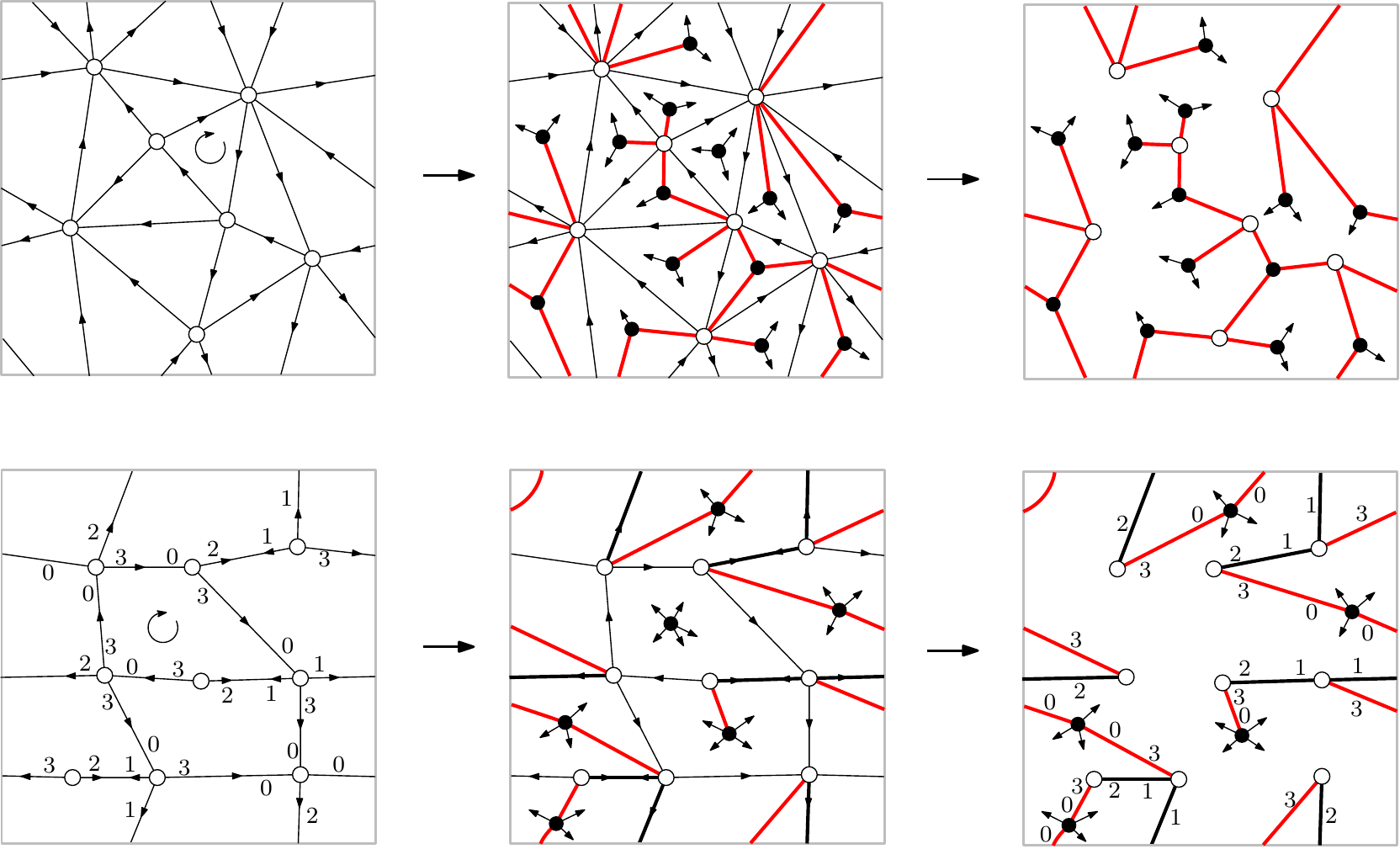}
\end{center}
\caption{
Left: 
a toroidal $d$-angulation in $\cF_d$
 endowed with its unique balanced $\frac{d}{d-2}$-orientation in $\cO_{d}^{1}$ 
($d=3$ for the top example, $d=5$ for the
 bottom example).  
Right: the associated  $\NN$-bimobile.} 
\label{fig:bij_dangulations}
\end{figure}

Consider a toroidal $\frac{d}{d-2}$-mobile $T$ and a
cycle $C$ of $T$ with a traversal direction. Let $w_{L}(C)$
(resp. $w_R(C)$) be the total weight of half-edges incident to a white
vertex of $C$ on the left (resp. right) side of $C$; and let
$s_{L}(C)$ (resp. $s_R(C)$) be the number of half-edges, including
buds, incident to a black vertex of $C$ on the left (resp. right)
side of $C$.  We define $\gamma_{L}(C)=w_{L}(C)+s_{L}(C)$,
$\gamma_R(C)=w_R(C)+s_R(C)$, and define the $\gamma$-score of $C$ as 
$\gamma(C)=\gamma_{R}(C)-\gamma_{L}(C)$.  Then $T$ is called
\emph{balanced} if the $\gamma$-score of any non-contractible cycle of
$T$ is $0$. We denote by $\mathcal U_d^{Bal}$ the subset of elements
of $\mathcal U_d$ that are balanced.
We will show (Lemma~\ref{lem:phispebal} in Section~\ref{sec:proofbijdangspe}) that $\Phi_+$
specializes into a ``balanced version'' of the bijection, i.e., a
bijection between face-rooted toroidal $d$-angulations endowed with a
balanced $\frac{d}{d-2}$-orientation in $\cO_d^{1}$,  and the
family $\mathcal U_d^{Bal}$.

We denote by $\cF_d$ the family of face-rooted $d$-toroidal maps 
such that the only $d$-angle enclosing the root-face is its contour. We will show 
(Lemma~\ref{lem:canonical_dori} in Section~\ref{sec:proofbijdangspe})  that a face-rooted
toroidal $d$-angulation $M$ has a balanced $\frac{d}{d-2}$-orientation
in $\cO_d^{1}$ if and only if $M\in\cF_d$, and in that case $M$ has a
unique balanced $\frac{d}{d-2}$-orientation in $\cO_d^{1}$, which is
the minimal one (by Lemma~\ref{lem:necmin}). 
Thus we obtain the
following bijection:

\begin{theorem}[Toroidal $d$-angulations of essential
  girth~$d$]\label{theo:bij_dang}
  For $d\geq 3$, there is a bijection between the map family $\cF_d$
 and the $\NN$-bimobile family $\mathcal U_d^{Bal}$. Every non-root face of the map corresponds to a black vertex in the associated $\NN$-bimobile.
\end{theorem}

Two examples are given in Figure~\ref{fig:bij_dangulations} for $d=3$ and $d=5$.

We now give the statement for bipartite maps. Let $b\geq 2$, and
$d=2b$. We denote by $\chF_{2b}$  the subfamily of maps in $\cF_{2b}$ that
are bipartite.  Proposition~\ref{th:evencase} ensures that a face-rooted
$d$-toroidal map $M$ is bipartite if and only if all the weights of
the unique minimal balanced $\frac{d}{d-2}$-orientation of $M$ are even. 
Hence, in the bijection of Theorem~\ref{theo:bij_dang},
$M\in\cF_{2b}$ is bipartite if and only if all half-edge weights in
the associated  $\NN$-bimobile are even. We formalize this
simplification as follows.

We define a \emph{$\frac{b}{b-1}$-mobile} as an $\NN$-bimobile of genus
$1$, where every white vertex has weight $b$, every edge has weight
$b-1$ and every black vertex has degree $2b$. The family of $\frac{b}{b-1}$-mobiles
is denoted by $\hat{\mathcal U}_b$. 
 Note that for $T\in\hat{\mathcal U}_b$, the $\NN$-bimobile $T'$ obtained from $T$ 
 by doubling every half-edge weight is an element
of $\mathcal U_{2b}$ (in particular, $T$ must have excess $2b$). 
We say that $T$ is \emph{balanced} if $T'$ is
balanced and we denote by $\hat{\mathcal U}_{b}^{Bal}$ the subset of elements
of $\hat{\mathcal U}_b$ that are balanced.  Thus we obtain the following
bijection:

\begin{figure}
\begin{center}
\includegraphics[width=\linewidth]{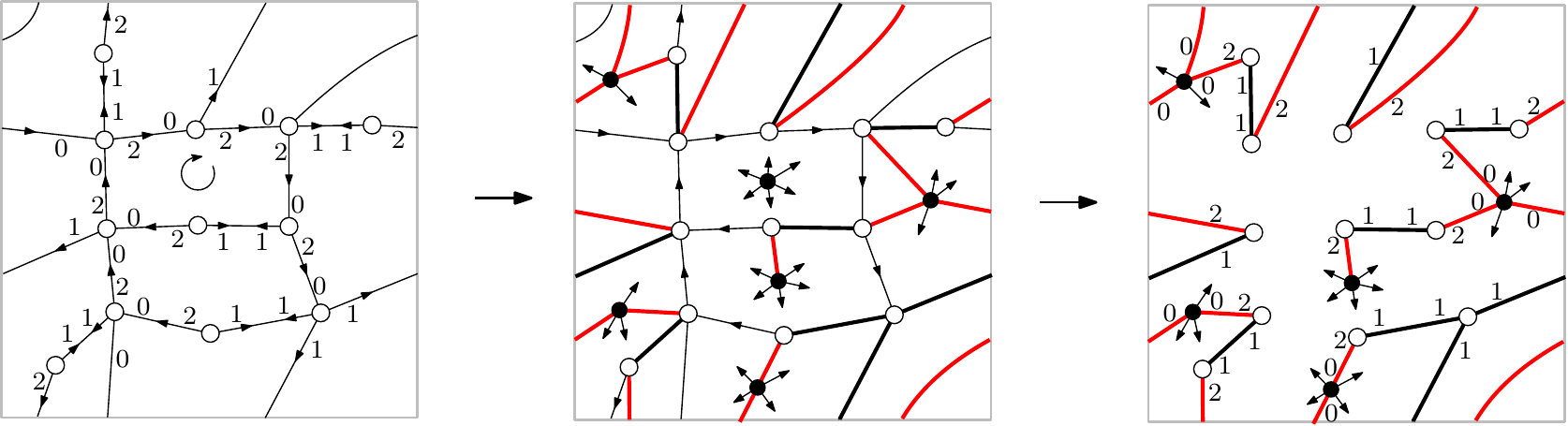}
\end{center}
\caption{Left: a bipartite toroidal face-rooted $2b$-angulation 
in $\chF_{2b}$ ($b=3$ in the example), endowed with its unique balanced $\frac{b}{b-1}$-orientation in $\cO_{2b}^{1}$. Right: the associated  $\NN$-bimobile.} 
\label{fig:bij_2bangulations}
\end{figure}

\begin{corollary}[Bipartite toroidal $2b$-angulations of essential
  girth~$2b$]\label{theo:bij_bip_2bang}
  For $b\geq 2$, there is a bijection between
the map family $\chF_{2b}$ and the $\NN$-bimobile family $\hat{\mathcal U}_b^{Bal}$. 
Every non-root face of the map corresponds to a
  black vertex in the associated $\NN$-bimobile.
\end{corollary}

An example is shown in Figure~\ref{fig:bij_2bangulations} for $b=3$.

\subsection{Extension to toroidal maps of essential girth $d$
with a root-face of degree $d$}\label{sec:bij_extended}
We state here a generalization of  
Theorem~\ref{theo:bij_dang} (resp. Corollary~\ref{theo:bij_bip_2bang}) 
to toroidal face-rooted maps 
of girth $d$ (resp. bipartite maps of girth $2b$)
with root-face degree $d$ (resp. $2b$). Note that we allow here all
 faces, except the root face, to have degree larger than $d$.
These results can be seen as toroidal counterparts of the 
bijections obtained in \cite{BF12b} for planar maps 
of girth $d$ with a root-face of degree $d$. 

Let $d\geq 1$. We denote by $\cL_d$ the family of face-rooted toroidal maps of essential
  girth~$d$, such that the root-face contour is a maximal $d$-angle.

  We now define the mobiles that will be set in bijection with maps in
  $\cL_d$. Recall that a $\ZZ$-bimobile is a bimobile with integer
  weights at the non-bud half-edges, which are in $\ZZ_{>0}$ (resp. in $\ZZ_{\leq 0}$)
  when the incident vertex is white (resp. black).  We define a
  \emph{toroidal $\frac{d}{d-2}$-$\ZZ$-mobile} as a $\ZZ$-bimobile
  of genus $1$ with weights in $\{-2,\ldots,d\}$ such that every white
  vertex has weight $d$, every edge has weight $d-2$ and every black
  vertex of degree $i$ has weight $-i+d$ (hence $i\geq d$). We denote
  by $\mathcal V_d$ the family of these $\ZZ$-bimobiles.  (Note that
  for $d\leq 3$, an element of $\mathcal V_d$ has no white-white edge,
  while for $d\geq 3$, it has no black-black edge.)  A counting
  argument similar to the one for proving Lemma~\ref{claim:excess}
  ensures that every $T\in\mathcal V_d$ has excess $d$.  Consider $T$
  in $\mathcal V_d$ and a cycle $C$ of $T$ with a traversal
  direction. Let $w_L(C)$ (resp. $w_R(C)$) be the total weight of
  half-edges incident to vertices (black or white) of $C$ on the left
  (resp. right) side of $C$. Let $s_L(C)$ (resp.  $s_R(C)$) be the
  total number of half-edges, including buds, incident to black
  vertices of $C$ on the left (resp. right) side of $C$.  We define
  $\gamma_L(C)=w_L(C)+s_L(C)$, $\gamma_R(C)=w_R(C)+s_R(C)$, and the
  $\gamma$-score of $C$ by $\gamma(C)=\gamma_R(C)-\gamma_L(C)$. Then
  $T$ is called \emph{balanced} if the $\gamma$-score of any
  non-contractible cycle of $T$ is $0$.  We denote by
  $\mathcal V_d^{Bal}$ the subset of elements of $\mathcal V_d$ that
  are balanced (see the left-part of
  Figure~\ref{fig:mobiles_larger_faces} for an example).

\begin{figure}[!h]
\begin{center}
\includegraphics[width=12cm]{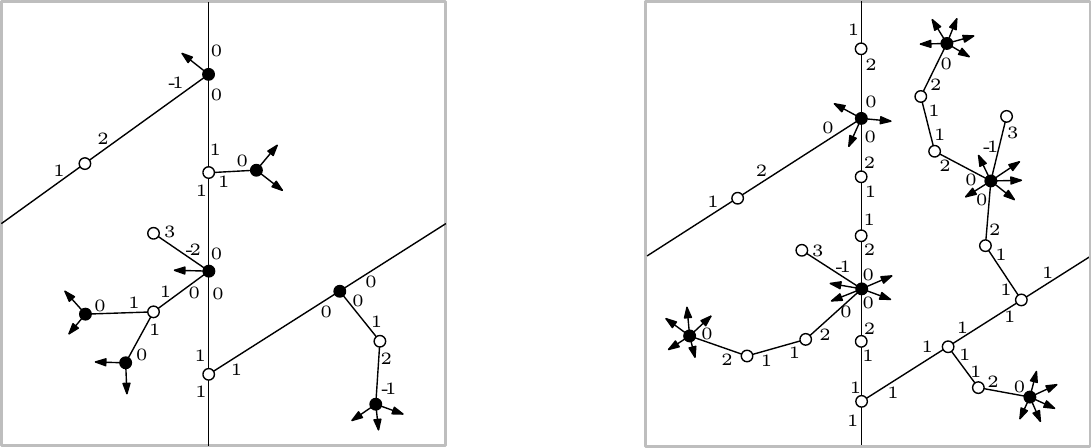}
\end{center}
\caption{Left: a $\ZZ$-bimobile in $\mathcal V_3^{Bal}$;
apart from the root-face (which has degree $3$), 
the corresponding toroidal map 
has $4$ faces of degree $3$, $2$ faces of degree $4$ and one face of degree $5$.
 Right: a $\ZZ$-bimobile in $\hat{\mathcal V}_3^{Bal}$;  
apart from the root-face (which has degree $6$), 
the corresponding toroidal bipartite map 
has $4$ faces of degree $6$ and $2$ 
faces of degree $8$.
 }
\label{fig:mobiles_larger_faces}
\end{figure}

\begin{theorem}[toroidal maps]\label{theo:bij_maps_d}
  For $d\geq 1$, there is a bijection between the map 
family $\cL_d$ and the $\ZZ$-bimobile family $\mathcal V_d^{Bal}$.  
Every non-root face in the map corresponds to a 
  black vertex of the same degree in the associated $\ZZ$-bimobile.
\end{theorem}

The proof of Theorem~\ref{theo:bij_maps_d} is delayed to Section~\ref{sec:proofs_theorems}. 

We now give the statement for bipartite maps.
Let $b\geq 1$. We denote by $\chL_{2b}$
 the subfamily of maps in $\cL_{2b}$ that are
bipartite. We define a \emph{toroidal $\frac{b}{b-1}$-$\ZZ$-mobile} 
as   a 
 $\ZZ$-bimobile of genus $1$  with weights in
$\{-1,\ldots,b\}$, all black vertices of even degree, such that
every white vertex has weight $b$, every edge has weight $b-1$ and
every black vertex of degree $2i$ has weight $-i+b$ (hence $i\geq
b$). The family of these $\ZZ$-bimobiles is denoted by $\hat{\mathcal V}_b$. 
(Note that for $b\leq 2$,
an element of $\hat{\mathcal V}_b$
 has no white-white edge,
while for $b\geq 2$, it has no black-black edge.)  
  Note also that the
$\ZZ$-bimobile $T'$ obtained from an element $T$ in  $\hat{\mathcal V}_b$
by doubling every half-edge weights is an element of ${\mathcal V}_{2b}$
(in particular, $T$ must have excess $2b$). 
We say that $T$ is 
\emph{balanced} if $T'$ is balanced and denote by $\hat{\mathcal V}_{b}^{Bal}$
 the subset of elements
of $\hat{\mathcal V}_b$ that are balanced
(see the right-part of Figure~\ref{fig:mobiles_larger_faces} 
 for an example).

\begin{theorem}[bipartite toroidal maps]\label{theo:bij_maps_2b}
For $b\geq 1$, there is a bijection between the map family  $\chL_{2b}$ 
and the $\ZZ$-bimobile family $\hat{\mathcal V}_b^{Bal}$.  
Every non-root face in the map corresponds to a
  black vertex of the same degree in the associated $\ZZ$-bimobile.
\end{theorem}
The proof is again delayed to Section~\ref{sec:proofs_theorems}.

Similarly as for $d$-angulations, in the bijection
of Theorem~\ref{theo:bij_maps_d} 
the map $M$ is bipartite if and only if the half-edge
weights in the corresponding $\ZZ$-bimobile $T$ are even, and 
upon dividing the weights by $2$ the $\ZZ$-bimobile one obtains is the one
associated to $M$ by the bijection of Theorem~\ref{theo:bij_maps_2b} 
(which can thus be seen as a parity specialization of the bijection
of Theorem~\ref{theo:bij_maps_d}).

\section{Counting results}\label{sec:counting}
For $d\geq 1$, let $\cM_d'$ (resp. $\cM_d$) be the family of rooted  (resp. face-rooted)
toroidal maps of essential girth $d$ with a root-face of degree $d$.  
In Section~\ref{sec:gen_expression} we express  the generating function of $\cM_d'$ (with control on the face-degrees) in terms of generating functions of balanced toroidal $\frac{d}{d-2}$-$\ZZ$ mobiles. To do this, we rely on the bijections obtained so far (Theorems~\ref{theo:bij_maps_d} and~\ref{theo:bij_maps_2b}) and on a decomposition of maps in $\cM_d'$ into a toroidal part and a planar part by cutting along a certain $d$-angle (the `maximal' one) enclosing the root-face. Then, in Sections~\ref{sec:bij_deriv_triang} and~\ref{sec:bij_deriv_quadrang} we show that the generating function
of $\frac{d}{d-2}$-$\ZZ$ mobiles can be expressed in certain specific cases (we show the approach on essentially 
simple triangulations and bipartite quadrangulations). 

\subsection{A general expression in terms 
of balanced mobiles}\label{sec:gen_expression}

For $M\in\cM_d$, recall that a $d$-angle of $M$
 is a contractible closed walk of length $d$, and it is called maximal if its enclosed
area is not contained in the enclosed area of another $d$-angle. 

\begin{lemma}
  \label{lem:maxdisjoint}
   Two distinct maximal $d$-angles of a map $M\in\cM_d$ always
  have disjoint interiors.
    \end{lemma}
\begin{proof}
Let us first reformulate the definition of  a $d$-angle. We define a \emph{region} of $M$
as given by $R=V'\cup E'\cup F'$ where $V',E',F'$ are subsets of the vertex-set, edge-set and face-set of $M$,  
 such that if $v\in V'$ then the edges incident to $v$ are in $E'$, and if $e\in E'$ then 
the faces incident to $e$ are in $F'$. Note that the union (resp. intersection) of two regions is also a region. 
A \emph{boundary-edge-side} of $R$ is an incidence face/edge of $M$ such that the face is in $F'$ and the edge
is not in $E'$. The \emph{boundary-length} of $R$, denoted by $\ell(R)$, is the number of boundary-edge-sides of $R$.  
A \emph{disk-region} is a region $R$ homeomorphic to an open disk. A $d$-angle thus corresponds
to the (cyclic sequence of) boundary-edge-sides of a disk-region $R$ such that $\ell(R)=d$; and it is \emph{maximal}
if there is no other disk-region $\bar{R}$ of boundary-length $d$ such that $R\subset \bar{R}$. 
 
We thus have to show that for two distinct disk-regions $R_1,R_2$ both 
enclosed by maximal $d$-angles, we have $R_1\cap R_2=\emptyset$.
It is easy to see that for any two regions $S_1,S_2$ we have $\ell(S_1)+\ell(S_2)=\ell(S_1\cup S_2)+\ell(S_1\cap S_2)$
(any incidence face/edge of $M$ has the same contribution to $\ell(S_1)+\ell(S_2)$ as to 
$\ell(S_1\cup S_2)+\ell(S_1\cap S_2)$).
Assume $R_1\cap R_2\neq \emptyset$.  
Since $R_1$ and $R_2$ are disk-regions, $R_1\cap R_2$ 
is a disjoint union of disk-regions $D_1,\ldots,D_k$, and we have
\[
2d=\ell(R_1)+\ell(R_2)=\ell(R_1\cup R_2)+\sum_{i=1}^k\ell(D_i).
\]
Since $M$ has essential girth $d$, we have $\ell(D_i)\geq d$ for each $1\leq i\leq k$. 
Hence we must have $k=1$ (we use $\ell(R_1\cup R_2)\geq 1$ to exclude the case $k=2$). 
Since $R_1\cap R_2$ is a disk-region, the union $R_1\cup R_2$ must also be 
a disk-region, hence $\ell(R_1\cup R_2)\geq d$. But  $\ell(R_1\cup R_2)=2d-\ell(D_1)\leq d$, hence
$\ell(R_1\cup R_2)=d$. Thus $R_1\cup  R_2$ is enclosed by a $d$-angle, contradicting the fact 
that $R_1$ and $R_2$ are enclosed by maximal $d$-angles. 
\end{proof}

\begin{figure}[h!]
\begin{center}
\includegraphics[width=10cm]{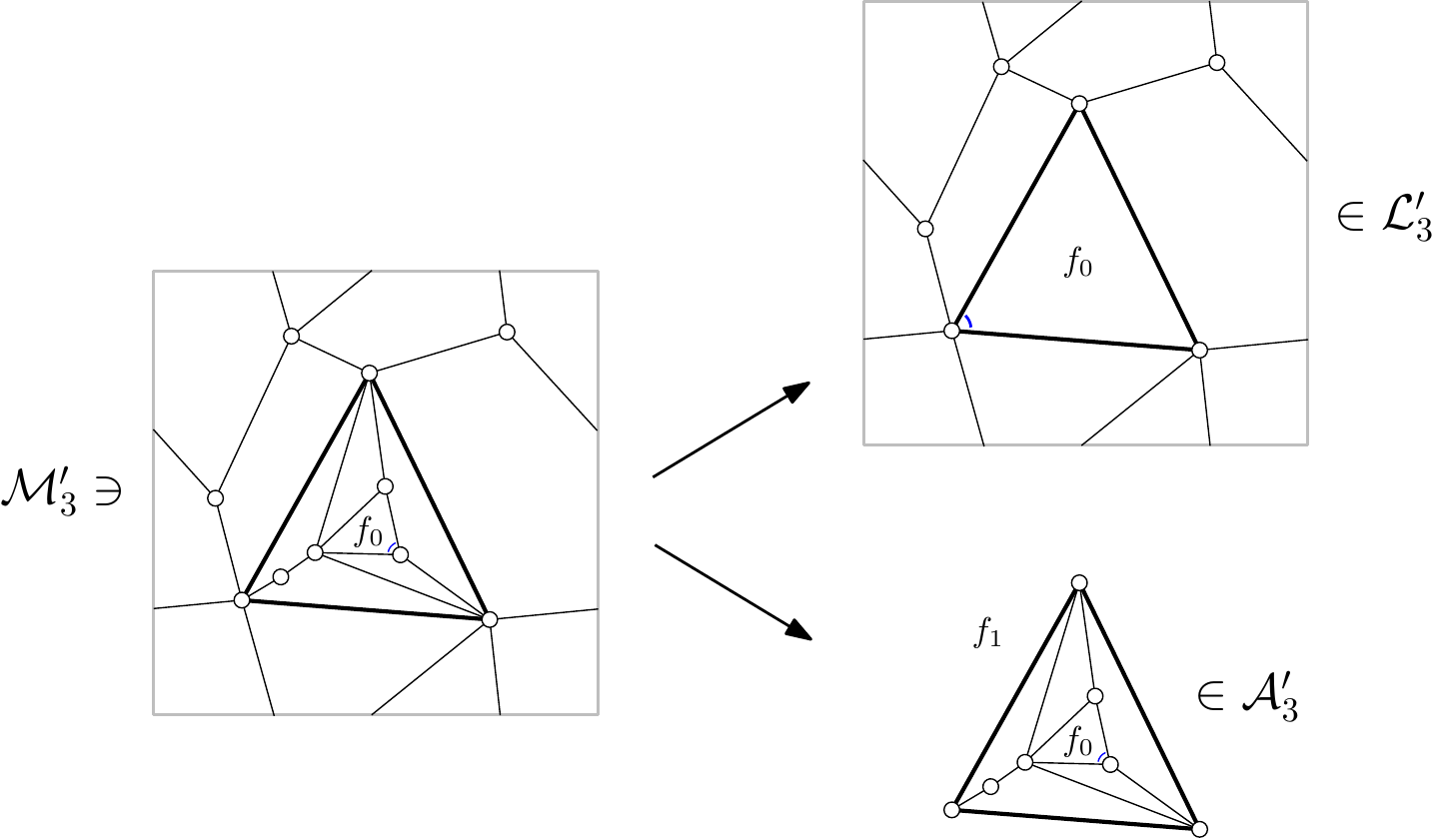}
\end{center}
\caption{Left: a map $M\in\cM_3'$. Cutting along the root-$3$-angle of $M$, one obtains a map $L\in\cL_3'$ and a map 
$A\in\cA_3'$. One of the $3$ vertices on the root-$3$-angle can be canonically chosen (i.e., the first of the $3$ vertices that is reached in a left-to-right depth-first-search starting from the root-corner), and this vertex is taken as the one incident to the root-corner of~$L$. The correspondence thus obtained is bijective.}
\label{fig:decomp_ML}
\end{figure}

Every $M\in\cM_d$ is rooted in a face $f_0$ of degree $d$, 
so $f_0$ is included in a maximal $d$-angle, and 
Lemma~\ref{lem:maxdisjoint} ensures that $M$ has a \emph{unique} maximal
$d$-angle enclosing the root-face. This $d$-angle is called the \emph{root-$d$-angle}. 
Consider the operation of cutting 
along the root-$d$-angle $C$ of $M$.  This operation yields two maps (one on each
side of $C$): a toroidal map $L$ with a marked face of degree $d$ and
a planar map $A$ with two marked faces $f_0,f_1$ each of degree~$d$. 

Recall that $\cL_d$ is the subfamily of $\cM_d$ where the root-face contour is a maximal $d$-angle; we denote by $\cL_d'$ the family of rooted toroidal maps such that the underlying face-rooted map is in $\cL_d$.  
 Moreover we let $\cA_d'$ be the family
of planar maps of girth $d$ with two marked faces $f_0,f_1$ of degree
$d$, and a marked corner in $f_0$ (we consider $f_1$ as the outer face). Then the previous decomposition at
the root-$d$-angle yields (see Figure~\ref{fig:decomp_ML})
\begin{equation}\label{eq:MdPrime}
\cM_d'\simeq \cL_d'\times\cA_d'.
\end{equation}

Let $M_d\equiv M_d(z;x_d,x_{d+1},\ldots)$ 
(resp. $L_d\equiv L_d(z;x_d,x_{d+1},\ldots)$) be the generating
function of maps in $\cM_d'$ 
(resp. in $\cL_d'$),  
with $z$ dual to the number of vertices and $x_i$ 
dual to the number of non-root faces of degree $i$.
And let $A_d\equiv A_d(z;x_d,x_{d+1},\ldots)$
 be the generating
function of maps in $\cA_d'$, with $z$ dual
to the number of vertices not incident to $f_1$,
and $x_i$ dual to the number of non-marked 
faces of degree $i$.  
Then by~\eqref{eq:MdPrime}
we have
\begin{equation}
M_d=L_d\cdot A_d. 
\end{equation}

The generating function $A_d$ has already 
been computed bijectively in~\cite{BF12b}, it reads:
$$
A_d=(1+W_0)^d,
$$
where $W_0$ is part of a finite set $W_{-1},W_0,\ldots,W_{d-1}$
of series (in the variables $z,x_d,x_{d+1},\ldots$)  that are specified by the
 system\footnote{We use the usual bracket notation: 
if $P=\sum_ka_ku^k$, then $[u^k]P=a_k$.}:

\begin{equation}\label{eq:syst}
\left\{
\begin{array}{ll}\ds
W_j=z\ \!h_{j+2}(W_1,\ldots,W_{d-1}) & \textrm{for all } j \textrm{ in }[-1\,..\,d-3],\\[.1cm]
\ds W_j=[u^{j+2}]\sum_{i\geq d}x_iu^i(1+W_0+u^{-1}W_{-1}+u^{-2})^{i-1} & \textrm{for all } j \textrm{ in }\{d-2,d-1\},\\
\end{array}
\right.
\end{equation}
where $h_j$ denotes the multivariate polynomial in the variables $w_1,w_2,\ldots$ defined by:
\begin{equation}\label{eq:def_hj}
h_j(w_1,w_2,\ldots)=[t^j]\frac{1}{1-\sum_{i>0}t^i w_i}=\sum_{r\geq 0}\sum_{\substack{i_1,\ldots,i_r>0\\ i_1+\cdots+i_r=j}}w_{i_1}\cdots w_{i_r}.
\end{equation}

Regarding $L_d$, let $\cT_d\equiv \mathcal V_d^{Bal}$ be the 
family of balanced toroidal
 $\frac{d}{d-2}$-$\ZZ$-mobiles and let $\cT_d'$ be the
family of objects in $\cT_d$ where 
one of the $d$ exposed half-edges is marked 
(see Section~\ref{sec:bijPhi+} for the definition of exposed half-edges).  
Then the bijection of Theorem~\ref{theo:bij_maps_d}    
directly yields a bijection between $\cL_d'$ 
and $\cT_d'$ (indeed the bijection of Theorem~\ref{theo:bij_maps_d}
relies on the general bijection given in Corollary~\ref{theo:bij_weighted_biori}, for which 
there is a natural 
1-to-1 correspondence
between the $d$ corners in the root-face and 
the $d$ exposed half-edges). Hence 
$L_d$ is also the generating function of balanced toroidal 
$\frac{d}{d-2}$-$\ZZ$-mobiles with a marked exposed half-edge, with $z$ dual to the 
number of white vertices and $x_i$ dual to 
 the number of black vertices of degree $i$. 

For a unicellular map $M$ of positive genus, 
the \emph{core} $C$ of $M$ is obtained from $M$
by successively deleting leaves, until there is 
no leaf (so $C$ has all its vertices of degree at
least $2$; the deleted edges form trees attached
at vertices of $C$). In $C$ we call \emph{maximal chain}
a path $P$ whose extremities have degree larger than $2$
and all non-extremal vertices of $P$ have degree $2$. 
Then the \emph{kernel} $K$ of $M$ is obtained
from $C$ by replacing every maximal chain by an edge.
In genus $1$ it is known that the kernel of a unicellular
map is either made of one vertex with two loops (double loop) 
or is made of $2$ vertices and $3$ edges joining them
(triple edge). 

Hence there are two types of toroidal mobiles, those 
where the associated kernel is the triple edge, called of 
type I, and those where the associated kernel 
is the double loop, called of type II.  
Let $G_d\equiv G_d(z;x_d,x_{d+1},\ldots)$ 
(resp. $H_d\equiv H_d(z;x_d,x_{d+1},\ldots)$ 
be the generating function of elements of type I (resp. type II) in $\cT_d$ 
and with a marked half-edge in the associated kernel. And let 
$\tG_d\equiv \tG_d(z;x_d,x_{d+1},\ldots)$ 
(resp. $\tH_d\equiv \tH_d(z;x_d,x_{d+1},\ldots)$ 
be the generating function of elements of type I (resp. type II) in $\cT_d$ 
and with a marked exposed half-edge.  
In all these generating functions, $z$
is dual to the number of white vertices and
 $x_i$ is dual to the 
number of black vertices of degree $i$.  
We have $L_d=(\tG_d+\tH_d)$, so
by what precedes $M_d=A_d\cdot(\tG_d+\tH_d)$; 
and by a classical double-counting
argument we have $\tG_d=\frac{d}{6}G_d$ and $\tH_d=\frac{d}{4}H_d$. 
Hence we obtain the following expression of $M_d$
in terms of generating functions of balanced toroidal $\frac{d}{d-2}$-$\ZZ$-mobiles:
\begin{proposition}\label{prop:series_exp_Md}
For $d\geq 1$, the generating function $M_d$ is given by
$$
M_d=d\cdot A_d\cdot(\tfrac{1}{6}G_d+\tfrac{1}{4}H_d).
$$
\end{proposition}

Very similarly we can obtain a general expression in the  bipartite case. For $b\geq 1$, let $\hM_{2b}$, $\hL_{2b}$
and $\hA_{2b}$ be the generating functions gathering
(respectively) the terms of $M_{2b}$, $L_{2b}$ and $A_{2b}$
given by bipartite maps.

Then, specializing~\eqref{eq:MdPrime} to bipartite maps 
yields
$$
\hM_{2b}=\hL_{2b}\cdot\hA_{2b}.
$$
In addition the generating function $\hA_{2b}$ has been given
an explicit expression in~\cite{BF12b},  
it reads: $$\hA_{2b}=(1+V_0)^{2b}$$ where $V_0$ is part
of a finite set $\{V_0,\ldots,V_{b-1}\}$ of generating functions 
specified by the system:

\begin{equation}\label{eq:syst_bip}
\left\{
\begin{array}{ll}\ds
V_j=z\ \!h_{j+1}(V_1,\ldots,V_{b-1}) & \textrm{for all } j \textrm{ in }[0\,..\,b-2],\\[.1cm]
  \ds V_{b-1}=\sum_{i\geq b}x_{2i}\binom{2i-1}{i-b}(1+V_0)^{b+i-1}  &\\
                                                                      
\end{array}
\right.
\end{equation}

Let $\hG_{2b}\equiv \hG_{2b}(z;x_2,x_4,\ldots)$ (resp.
$\hH_{2b}\equiv \hH_{2b}(z;x_2,x_4,\ldots)$) be the generating
function of balanced toroidal $\frac{b}{b-1}$-$\ZZ$-mobiles of type
I (resp. type II) with a marked half-edge in the associated kernel,
with $z$ dual to the number of white vertices and $x_{2i}$ dual to the
number of black vertices of degree $2i$.  By the very same arguments
as to prove Proposition~\ref{prop:series_exp_Md}, we can express
$\hM_{2b}$ in terms of generating functions of balanced toroidal
$\frac{b}{b-1}$-$\ZZ$-mobiles:
\begin{proposition}\label{prop:series_exp_Mdprime}
For $b\geq 1$, the generating function $\hM_{2b}$ is given by
$$
\hM_{2b}=2b\cdot \hA_{2b}\cdot(\tfrac{1}{6}\hG_{2b}+\tfrac{1}{4}\hH_{2b}).
$$
\end{proposition}

Propositions~\ref{prop:series_exp_Md} and~\ref{prop:series_exp_Mdprime} ensure that the enumeration 
of rooted toroidal maps (resp. bipartite toroidal maps) of essential
girth $d$ (resp. $2b$) and root-face degree $d$ (resp. $2b$),  
with control  on the face-degrees,  
amounts to counting balanced toroidal $\frac{d}{d-2}$-$\ZZ$-mobiles
(resp. $\frac{b}{b-1}$-$\ZZ$-mobiles) with control on the degrees of the 
black vertices. We show in the next two sections that 
this can be carried out for essentially simple triangulations
and for essentially simple bipartite quadrangulations,
yielding the two simple generating function expressions stated next:

\begin{proposition}[essentially simple triangulations]\label{prop:triang}
Let $t_n$ be the number of essentially simple rooted toroidal triangulations with $n$ vertices. 
Then
$$
\sum_{n\geq 1}t_nz^n=\frac{r}{(1-3r)^2},
$$
where $r\equiv r(z)$ is given by $r=z(1+r)^4$.
\end{proposition}  

\begin{proposition}[essentially simple quadrangulations]\label{prop:quadrang}
Let $q_n$ be the number of rooted toroidal quadrangulations
with $n$ vertices (and also $n$ faces) that are 
essentially simple and bipartite. Then
$$
\sum_{n\geq 1}q_nz^n=\frac{r^2}{(1+2r)(1-2r)^2},
$$
where $r\equiv r(z)$ is given by $r=z(1+r)^3$. 
\end{proposition}

Similar calculations could be carried out for bipartite quadrangulations and for  
essentially loopless triangulations. 
The expression for the series of rooted toroidal bipartite quadrangulations (counted by vertices) is
 $F(z)=\frac{r^2(1+3r)}{(1+r)(1-3r)^2}$ where $r\equiv r(z)$ is given by $r=z(1+3r)^2$.  
Bijective derivations of this formula have been given in~\cite{CMS09,Lep18}.  
And the expression for the series of rooted toroidal essentially loopless triangulations (counted by vertices) 
is $G(z)=\frac{r(1+2r)}{(1-4r)^2}$ where $r\equiv r(z)$ is given by $r=z(1+2r)^3$. By a classical substitution approach~\cite[Sec.2.9]{GoJa83} the series $F(z)$ can be related to the series of Proposition~\ref{prop:quadrang} (and similarly the series $G(z)$ can be related to the series of Proposition~\ref{prop:triang}), so that one expression can be deduced from the other one (however via some algebraic manipulations, so a bijective derivation of one expression does not yield
a bijective derivation of the other expression via this approach).

Calculations for toroidal $d$-angulations
of essential girth $d\geq 5$ seem much more technical. In principle the line of approach we follow in the next two subsections is doable (see~\cite{Ch09} where it is carried out for constellations and hypermaps of arbitrarily large face-degrees) and should at least yield an algebraic expression, but likely a complicated one, whereas it is to be expected that the final expression should be simple\footnote{Indeed, combining the substitution approach
in~\cite{BG15} to deal with girth constraints, together with the expressions obtained from the topological
recursion approach for toroidal maps
 with no girth constraint~\cite{Eyn,CF16}, 
it should be possible to show that when the face-degrees
are bounded (i.e., for some fixed $N$, 
the face-degree variables $x_{2i}$ are 
taken to be $0$ for $i>N$), the generating function $\hM_{2b}$ has
a rational expression in terms of the series $V_{0},\ldots,V_{b}$
and the variables $x_{2b},\ldots,x_{2N}$, and 
a similar rationality property should hold for $M_d$.}. 
In this perspective it would 
be helpful to have a better combinatorial explanation
of the simplicity of the generating function expressions obtained in 
Propositions~\ref{prop:triang} and~\ref{prop:quadrang}.

\subsection{Bijective derivation of Proposition~\ref{prop:triang}}\label{sec:bij_deriv_triang}

In this section we compute the generating function $T(z)$
of rooted toroidal triangulations that have essential girth $3$
(or equivalently, that are essentially simple), with $z$ dual to the number
 of vertices.   
Note that, for $d=3$, a toroidal $\frac{d}{d-2}$-mobile $T$ has all its edges of weight~$1$, hence all edges
are black-white with weight~$1$ on the half-edge incident to the white extremity. Since
white vertices have weight $3$, they have degree $3$. Hence, for $d=3$ the toroidal $\frac{d}{d-2}$-mobiles identify  
to toroidal mobiles (edges are black-white, buds are at black vertices only) where every vertex (white or black)
 has degree $3$, which we call \emph{3-regular toroidal mobiles}, see Figure~\ref{fig:mobile_triang} (1st drawing) for an example.  
Note that such mobiles must be of type I, since in type II
the unique vertex in the kernel must have degree at least $4$. 
A 3-regular toroidal mobile $T$ is called \emph{balanced} if 
every cycle of $T$ has the same number of incident half-edges
on the left side as on the right side. 
Let $N(z)$ be the generating function of balanced 3-regular toroidal
mobiles with a marked half-edge in the associated kernel.

\begin{figure}[!h]
\begin{center}
\includegraphics[width=13cm]{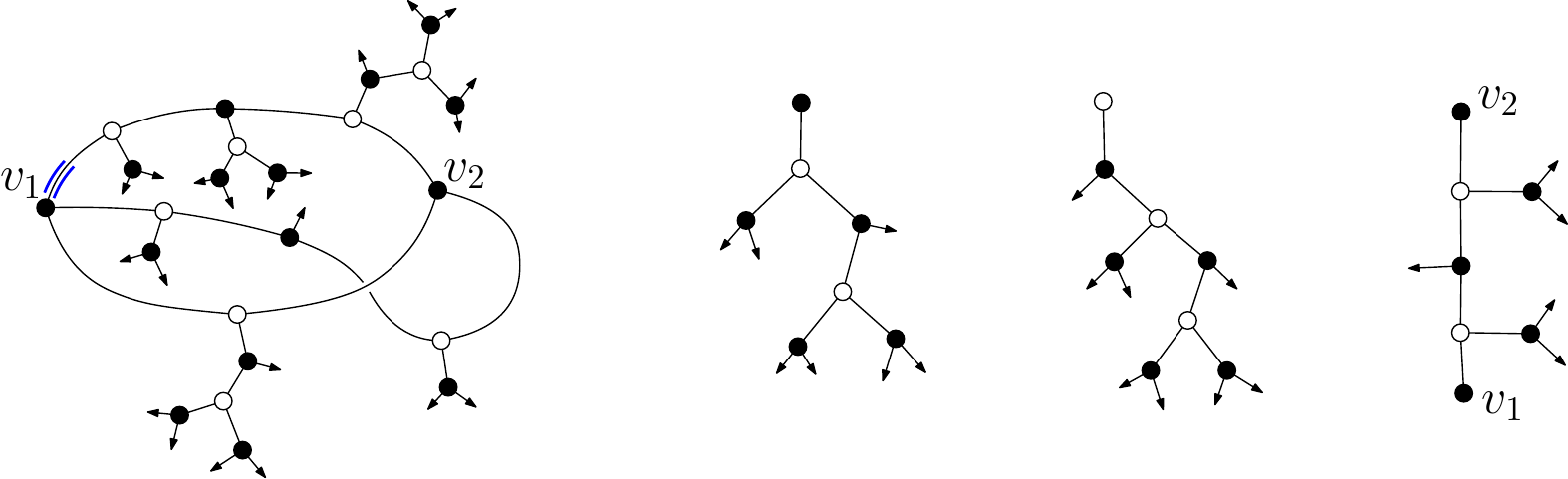}
\end{center}
\caption{From left to right: a toroidal 3-regular mobile $T$ 
counted by $\Nbb(z)$ (where the marked half-edge of the kernel
is indicated); a rooted R-mobile; a rooted S-mobile;
and a bi-rooted 3-regular mobile (the second branch of $T$, both roots
are black).}
\label{fig:mobile_triang}
\end{figure}

When setting $x_i=\delta_{i=3}$ in system~\eqref{eq:syst}, one obtains $W_0=zW_1^2$ and
$W_1=(1+W_0)^2$. Let $R\equiv R(z)$ and $S\equiv S(z)$ be given by 
$R=1+W_0$ and $S=W_1$. So $R,S$ satisfy the system
$\{R=1+zS^2, S=R^2\}$. Then by Proposition~\ref{prop:series_exp_Md},
we have:
$$
T(z)=\frac{1}{2}R(z)^3N(z),
$$

Note that the generating function $N(z)$ splits as
$$
N(z)=\Nbb(z)+\Nwb(z)+\Nbw(z)+\Nww(z)=\Nbb(z)+2\Nbw(z)+\Nww(z),
$$
depending on the colors of the two vertices $v_1,v_2$ of the 
kernel (with $v_1$ the one incident to the marked half-edge),
where the second equality follows from $\Nbw(z)=\Nwb(z)$, 
since $v_1$ and $v_2$ play symmetric roles.         

We now define a \emph{rooted} mobile as a planar 
mobile with a marked vertex 
that is a leaf, called the \emph{root} (it is allowed
for a rooted mobile to be just made of a black vertex with a single
incident bud).  
And we define a \emph{bi-rooted}
 mobile as a mobile with two marked vertices $v_1,v_2$ 
that are leaves, called \emph{primary root} and \emph{secondary root}. 
A rooted or bi-rooted mobile is called \emph{3-regular} if 
all its non-root vertices have degree $3$.  

An \emph{$R$-mobile} (resp. \emph{$S$-mobile}) is defined  
as a rooted 3-regular mobile where the root is black (resp. white),
see 2nd and 3rd drawing in Figure~\ref{fig:mobile_triang}.  
By a decomposition at the root, one checks that 
$R$ is the generating function of $R$-mobiles
and $S$ is the generating function of $S$-mobiles, 
with $z$ dual to the number of non-root white vertices.

For a bi-rooted mobile, the path connecting the two roots is called
the \emph{spine}, the traversal direction being from the primary to
the secondary root (see the fourth drawing of
Figure~\ref{fig:mobile_triang}). For each non-extremal vertex $v$ of
the spine, the \emph{balance} at $v$ is defined as the number of
half-edges (including buds) incident to $v$ on the left side of the
spine, minus the number of half-edges (including buds) incident to $v$
on the right side of the spine.  And the \emph{balance} of the
bi-rooted mobile is defined as the total balance over all non-extremal
vertices of its spine.  For a bi-rooted 3-regular mobile, the balance
at each vertex of the spine is either $+1$ or $-1$, so that the
sequence of balances along the spine is naturally encoded by a
directed path with steps in $\{-1,+1\}$, and the final height of the
path is the balance of the rooted bimobile, see
Figure~\ref{fig:path_triang}.

\begin{figure}[!h]
\begin{center}
\includegraphics[width=13cm]{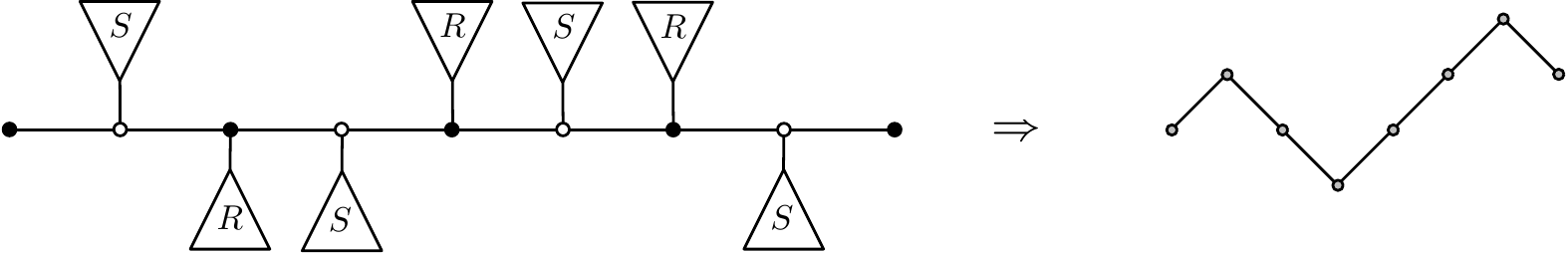}
\end{center}
\caption{Left: a bi-rooted mobile of balance $1$ 
(generic notation for
mobiles hanging from the spine). Right: the associated path
with steps $+1$ or $-1$ (up or down), ending at height $1$.}
\label{fig:path_triang}
\end{figure}

Clearly  
a 3-regular toroidal mobile $T$ 
(with a marked half-edge in the kernel)  
 decomposes into an ordered triple of  
 bi-rooted 3-regular mobiles (one for each edge of the kernel),   
and $T$ is balanced if and only if the $3$ bi-rooted mobiles have the 
same balance.  
Hence, if for $i\in\mathbb{Z}$ we 
let $\Kbbi(z)$, $\Kbwi(z)$, $\Kwwi(z)$ be the generating functions 
of bi-rooted 3-regular mobiles of balance $i$ 
where $v_1,v_2$ are black/black (resp.
black/white, white/white), and with $z$ dual 
to the number of non-root white vertices, then we find
$$
\Nbb(z)=\sum_{i\in\mathbb{Z}}\Kbbi(z)^3,\ 
\ \Nbw(z)=z\sum_{i\in\mathbb{Z}}\Kbwi(z)^3,\ \ 
\Nww(z)=z^2\sum_{i\in\mathbb{Z}}\Kwwi(z)^3.
$$

For $i\in\zZ$, let $p_{n,i}$ be the number
 of walks of length $n$ with steps in $\{-1,1\}$,
starting at $0$ and ending at $i$ (note that 
$p_{n,i}=0$ if $i\neq n\mathrm{\ mod\ }2$). 
We also define the generating
function of walks ending at $i$ as 
$$
\Pi(t)=\sum_{n\geq 0}p_{n,i}t^{\lfloor n/2\rfloor}.
$$ 
We clearly have for $i\in\zZ$, 
$$
\Kwwi(z)=0\ \mathrm{for}\ i\ \mathrm{even},\ \ \  
\Kwwi(z)=R\cdot\Pi(t)\Big|_{t=zRS}\ \ \mathrm{for}\ i\ \mathrm{odd}.
$$
$$
\Kbbi(z)=0\ \mathrm{for}\ i\ \mathrm{even},\ \ \  
\Kbbi(z)=zS\cdot\Pi(t)\Big|_{t=zRS}\ \ \mathrm{for}\ i\ \mathrm{odd}.
$$
$$
\Kbwi(z)=0\ \mathrm{for}\ i\ \mathrm{odd},\ \ \ \  
\Kbwi(z)=\Pi(t)\Big|_{t=zRS}\ \ \mathrm{for}\ i\ \mathrm{even}.
$$
Let $B(t)=P^{(0)}(t)$ be the generating
function of bridges (walks ending at $0$), 
and let $U(t)$ be the generating function of
non-empty Dyck walks (i.e., bridges of positive 
length never visiting negative values). 
Then $U\equiv U(t)$ is classically given by
$$
U=t\cdot(1+U)^2,
$$
and then (looking at the first return to $0$
for non-empty bridges), $B\equiv B(t)$ satisfies
the equation $B=1+2t(1+U)B$, so that 
$$
B=\frac1{1-2t\cdot(1+U)}.
$$
Then we have $\Pi(t)=P^{(-i)}(t)$ for $i<0$, 
and for $i>0$ we have (by a classical decomposition at the last
visits to $0,1,\ldots,i-1$, see~\cite{PoSc04}) 
$$
\Pi(t)=B\cdot(1+U)^i\cdot t^{\lfloor i/2\rfloor}.
$$
Hence we have
\begin{align*}
\Nww(z)&=z^2R^3
\sum_{\substack{i\in\zZ\\i\ \mathrm{odd}}}\Pi(t)^3\Big|_{t=zRS}\\
&=2z^2R^3\frac{B^3\cdot(1+U)^3}{1-t^3(1+U)^6}\Big|_{t=zRS}=2z^2R^3\frac{B^3\cdot(1+U)^3}{1-U^3}\Big|_{t=zRS}
\end{align*}
The last expression can be written in terms of $U$
uniquely. Indeed, all involved quantities can be expressed in terms of $U$: we have
$$
t=\frac{U}{(1+U)^2}=\frac1{U+2+U^{-1}},\ \ B=\frac1{1-2t(1+U)}=\frac{1+U}{1-U},
$$
and $t=zRS=zR^3=(R-1)/R=1-1/R$, so that 
$$
R=\frac{1}{1-t}=\frac{(1+U)^2}{(U^2+U+1)},\ \ \ z=\frac{R-1}{R^4}=\frac{(U^2+U+1)^3U}{(U+1)^8}.
$$
Overall we find
$$
\Nww(z)=\frac{2U^2(U^2+U+1)^2}{(U-1)^4(1+U)^4}=\frac{2(U+1+U^{-1})^2}{(U-2+U^{-1})^2(U+2+U^{-1})^2}.
$$
Similarly as in~\cite{CMS09}, we obtain an expression that is rational
in $U+U^{-1}$ and so it is also rational in $t$ since
$U+U^{-1}=1/t-2$, and then rational in $R$ since $t=1-1/R$. Finally, we obtain
$$
\Nww(z)=\frac{2(R-1)^2}{(3R-4)^2R^2}.
$$
Similarly we find
\begin{align*}
\Nbb(z)&=2z^3S^3\frac{B^3\cdot(1+U)^3}{1-t^3(1+U)^6}\Big|_{t=zRS}=2z^3S^3\frac{B^3\cdot(1+U)^3}{1-U^3}\Big|_{t=zRS}  \\
&=\frac{2(U+1+U^{-1})^2}{(U-2+U^{-1})^2(U+2+U^{-1})^3}=\frac{2(R-1)^3}{(3R-4)^2R^3}
\end{align*}
and 
\begin{align*}
\Nbw(z)&=zB^3\Big(\frac{2}{1-U^3}-1\Big)\Big|_{t=zRS}\\
&=\frac{(U+1+U^{-1})^2(U-1+U^{-1})}{(U-2+U^{-1})^2(U+2+U^{-1})^2}=\frac{(1-R)(2R-3)}{(3R-4)^2R^2}
\end{align*}

We can now conclude the proof of Proposition~\ref{prop:triang}. 
The sum of the $3$ contributions above (with the $3$rd contribution multiplied by $2$) gives
 $N(z)=\frac{2(R-1)}{(4-3R)^2R^3}$, so that we obtain 
$$
T(z)=\frac1{2}N(z)R^3=\frac{(R-1)}{(4-3R)^2},
$$
which gives the stated expression upon writing $r$ for $R-1$ (so 
that $r$ is given by $r=z(1+r)^4$). 

\subsection{Bijective derivation of Proposition~\ref{prop:quadrang}}\label{sec:bij_deriv_quadrang}

We now compute the generating function $Q(z)$
(with $z$ dual to the number of vertices) 
of rooted toroidal quadrangulations that are 
bipartite and essentially simple (we will overrule
here some notation from the previous section). 
For $b=2$, a toroidal $\frac{b}{b-1}$-mobile $T$ has all its edges of weight~$1$, hence all edges
are black-white with weight~$1$ on the half-edge incident to the white extremity. Since
white vertices have weight $4$, they have degree $4$. Hence, for $b=2$ the toroidal $\frac{b}{b-1}$-mobiles identify  
to toroidal mobiles where black vertices have degree $4$ and 
white vertices have degree $2$, which we call \emph{$(4,2)$-regular toroidal mobiles}.  Such a
mobile is called \emph{balanced} if, for 
every cycle, it has the same number of incident
half-edges (including buds) on the left side
as on the right side. Let $\NI(z)$ 
(resp. $\NII(z)$) be the generating function of
toroidal balanced $(4,2)$-regular mobiles
of type I (resp. type II), with $z$ dual to the number of white vertices. 

When setting $x_i=\delta_{i=4}$ in system~\eqref{eq:syst_bip}, one
obtains $V_0=zV_1$ and $V_1=(1+V_0)^3$. Let $R\equiv R(z)$ be given by $R=1+V_0$, so $R$ satisfies 
 $R=1+zR^3$. Then by
Proposition~\ref{prop:series_exp_Mdprime}, we have:
 
$$
Q(z)=R(z)^4\cdot\Big(\frac{2}{3}\NI(z)+\NII(z)\Big).
$$

A rooted or bi-rooted (planar) mobile is called \emph{$(4,2)$-regular}
if the root-vertex is black and all the non-root vertices have degree
$4$ if black and degree $2$ if white. Rooted $(4,2)$-regular mobile
are shortly called R-mobiles; again it is easy to check that $R(z)$ is
the generating function of R-mobiles, with $z$ dual to the number of
white vertices.  For a bi-rooted $(4,2)$-regular mobile the balance at
each black vertex of the spine is in $\{-2,0,+2\}$, so that the sequence of
balances along the spine is now encoded by a path with increments in
$\{-1,0,+1\}$, the final value of the path giving half of the total
balance (see Figure~\ref{fig:path_quad_mobile}).  Let $p_{n,i}$ be
the number of such paths of length $n$ ending at $i$, and let
$\Pi(t)=\sum_{n\geq 0}p_{n,i}t^{n}$ be the generating function for
walks ending at $i$, and let $B\equiv B(t)=P^{(0)}(t)$ be the
generating function of those ending at $0$, called \emph{bridges}.

\begin{figure}[!h]
\begin{center}
\includegraphics[width=13cm]{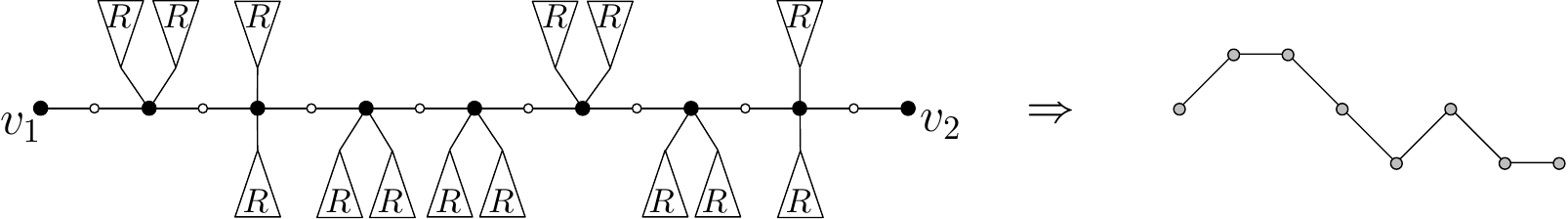}
\end{center}
\caption{Left: a bi-rooted $(4,2)$-regular mobile of 
balance $-1$ (generic notation
for the mobiles hanging from the spine). Right: the associated
path with steps in $\{-1,0,1\}$, ending at height $-1$.}
\label{fig:path_quad_mobile}
\end{figure}

A mobile counted by $\NII(z)$ (see Figure~\ref{fig:mobile_quad}
for an example) clearly
decomposes  into a pair
of bi-rooted $(4,2)$-regular mobiles
both of balance $0$ (one bi-rooted mobile for each 
of the two edges of the kernel), which gives
$$
\NII(z)=z^2B^2\Big|_{t=zR^2}.
$$

\begin{figure}[!h]
\begin{center}
\includegraphics[width=10cm]{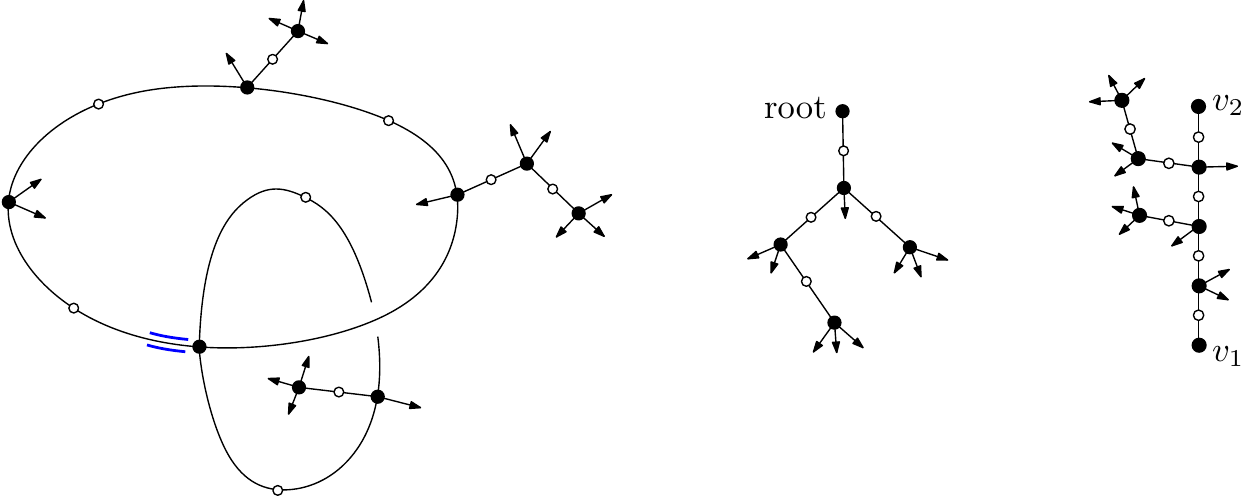}
\end{center}
\caption{From left to right: a toroidal $(4,2)$-regular mobile $M$ 
counted by $\NII(z)$; an R-mobile; 
and a bi-rooted $(4,2)$-regular mobile (the second branch of $M$).} 
\label{fig:mobile_quad}
\end{figure}

Let $C\equiv C(t)$ be the generating function
of walks counted by $B(t)$ that never 
visit  values in $\ZZ_{<0}$ (called \emph{Motzkin excursions}), and let $U(t):= tC(t)$. 
Note that $U\equiv U(t)$ is given by the equation
$$
U=t\cdot(1+U+U^2).
$$
Again our aim will be to express all generating
functions rationally in terms of $U$. 
We have
$$
t=\frac{1}{U+1+U^{-1}},
$$
and moreover we have $t=zR^2=(R-1)/R=1-1/R$,
which gives 
$$
R=\frac1{1-t}=\frac{U^2+U+1}{U^2+1},\ \ \ z=\frac{R-1}{R^3}=\frac{(U^2+1)^2U}{(U^2+U+1)^3}.
$$
Note that $B$ satisfies the equation
$B=1+(t+2tU)B$ (obtained by looking at the 
first return to $0$), which gives
$$
B=\frac{1}{1-t-2tU}=\frac{U^2+U+1}{(1-U)(1+U)}.
$$
We thus obtain the following expression for $\NII(z)$ in terms of $U$:
$$
\NII(z)=\frac{(U^2+1)^4U^2}{(U^2+U+1)^4(U-1)^2(U+1)^2}=\frac{(U+U^{-1})^4}{(U+1+U^{-1})^4(U-2+U^{-1})(U+2+U^{-1})}.
$$
We obtain an expression that is rational
in $U+U^{-1}$ and so it is also rational in $t$ since
$U+U^{-1}=1/t-1$, and then rational in $R$ since $t=1-1/R$. Finally, we obtain

$$
\NII(z)=\frac{(R-1)^2}{(2R-1)(3-2R)R^4}.
$$

Regarding mobiles counted by $\NI(z)$, note that the two vertices $v_1,v_2$
 of the kernel $\kappa$ have to be black
(since white vertices have degree $2$),
and moreover, for $i\in\{1,2\}$, $v_i$ has  
exactly one corner (denoted $\nu_i$) that carries an attached R-mobile. Note that 3 situations
can arise in a \ccw walk (of length $6$ since
$\kappa$ has $3$ edges) around the unique face of 
$\kappa$: $\nu_2$ is either (a) just after
$\nu_1$, (b) or $3$ steps after $\nu_1$, (c) or just before $\nu_1$. 
Let $\NIi(z)$, $\NIii(z)$, $\NIiii(z)$ be the respective
contributions to $\NI(z)$. 
The first and last situations are clearly symmetric
(up to exchanging the roles of $v_1$ and $v_2$), 
hence $\NIi(z)=\NIiii(z)$. 

In case (b), the mobile is made of $3$ bi-rooted mobiles (one for each
branch connecting $v_1$ to $v_2$) 
of the same excess $i\in\zZ$, plus two attached 
R-mobiles (those at $\{\nu_1,\nu_2\}$). Hence 
$$
\NIii(z)=3R(z)^2\sum_{i\in\zZ}z^3\Pi(t)^3\Big|_{t=zR^2}
$$
where the factor $3$ accounts for the choice of the marked
half-edge of $\kappa$, the factor $R(z)^2$ accounts for the attached R-mobiles
at $v_1$ and $v_2$, and each of the $3$ factors $z\Pi(t)\big|_{t=zR^2}$
accounts for each of the $3$ branches connecting $v_1$ to $v_2$.

Note that $\Pi(t)=P^{-i}(t)$ for $i<0$, and for $i>0$ a 
decomposition at the last visits to $0$, to $1,\ldots,i-1$,    ensures that 
$$
\Pi(t)=B(t)\cdot U(t)^{i}.
$$
Hence we have
$$
\NIii(z)=3z^3R(z)^2B(t)^3\sum_{i\in\zZ}U(t)^{3|i|}\Big|_{t=zR^2}
=3z^3R(z)^2B(t)^3\Big(1+\frac{2U(t)^3}{1-U(t)^3}\Big)\Big|_{t=zR^2}.
$$
Again we can express everything rationally in terms of $U$, and find
$$
\NIii(z)=\frac{3(U+U^{-1})^4(U-1+U^{-1})}{(U-2+U^{-1})^2(U+1+U^{-1})^5(U+2+U^{-1})}=\frac{3(R-1)^3(2-R)}{(2R-1)R^5(3-2R)^2}.
$$
Finally, in case (a), it is easy to see that two of the 
bi-rooted mobiles from $v_1$ to $v_2$ have same balance $i\in\zZ$, 
while the bi-rooted mobile for the remaining branch has balance $i-1$
(see 1st drawing of Figure~\ref{fig:types_NI} for an example).  

\begin{figure}[!h]
\begin{center}
\includegraphics[width=10cm]{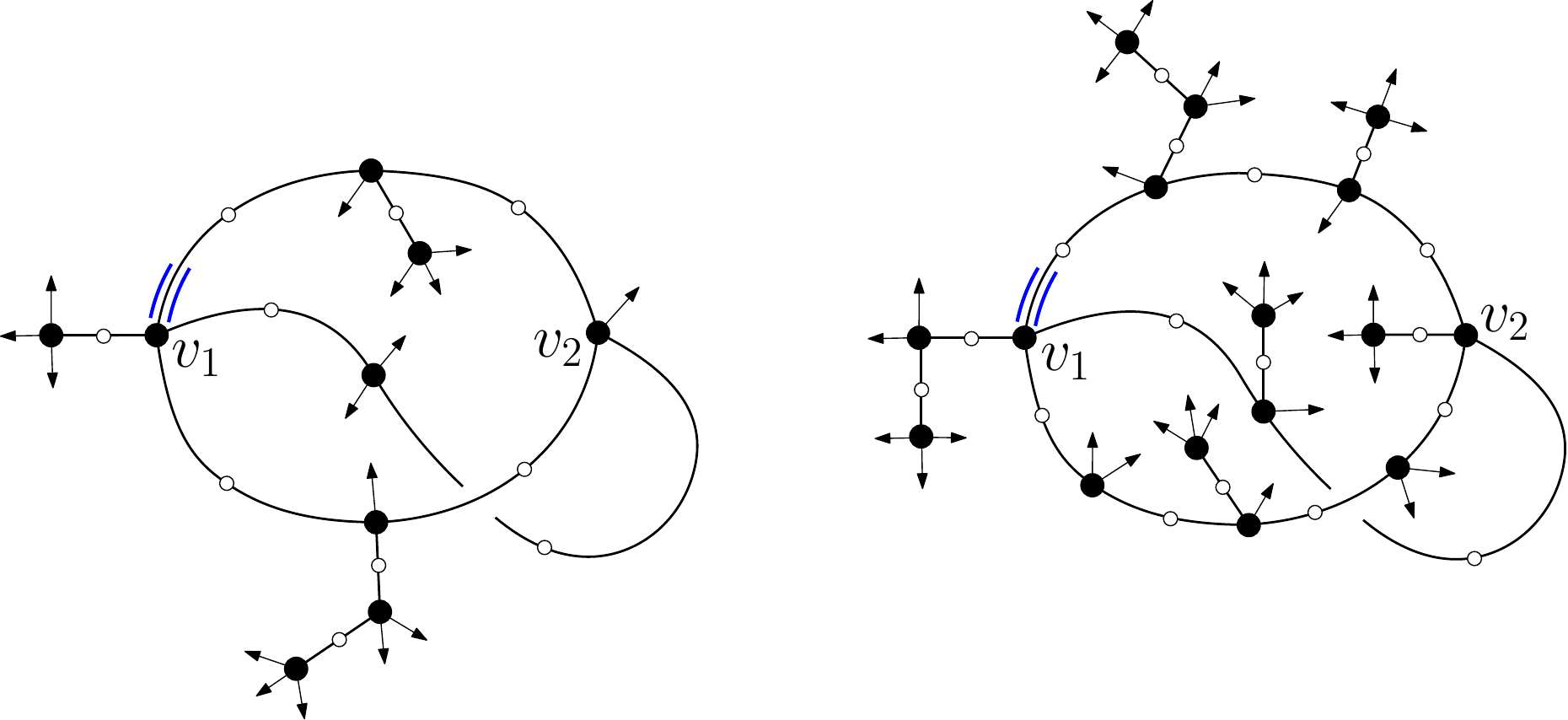}
\end{center}
\caption{Left: a toroidal $(4,2)$-regular mobile  
counted by $\NIi(z)$; Right: a toroidal $(4,2)$-regular mobile 
counted by $\NIii(z)$.} 
\label{fig:types_NI}
\end{figure}

Hence 
\begin{align*}
\NIi(z)=3R(z)^2\sum_{i\in\zZ}z^3\Pi(t)^2P^{(i-1)}(t)\Big|_{t=zR^2}&=3R(z)^2\sum_{i\in\zZ}z^3B(t)^3U^{2|i|+|i-1|}(t)\Big|_{t=zR^2}\\
&=3R(z)^2z^3B(t)^3\frac{U(t)+U(t)^2}{1-U(t)^3}\Big|_{t=zR^2}
\end{align*}
Again we rewrite the expression in terms of $U$ and then $R$, finding
$$
\NIi(z)
=\frac{3(U+U^{-1})^4}{(U-2+U^{-1})^2(U+1+U^{-1})^5(U+2+U^{-1})}
=\frac{3(R-1)^4}{(2R-1)R^5(3-2R)^2}.
$$
We thus get
$$
\NI(z)=2\NIi(z)+\NIii(z)=\frac{3(R-1)^3}{R^4(2R-1)(3-2R)^2}.
$$

We thus obtain
$$
Q(z)=R(z)^4\cdot\Big(\frac{2}{3}\NI(z)+\NII(z)\Big)=\frac{(R-1)^2}{(2R-1)(3-2R)^2},
$$
which concludes the proof of Proposition~\ref{prop:quadrang}, upon writing $r=R-1$
(so that $r$ is given by $r=z(1+r)^3$).

\section{Proofs of the bijections}
\subsection{Proof of Theorem~\ref{theo:phi_ori}}\label{sec:proof_phi}
\subsubsection{The Bernardi-Chapuy bijection} 
Similarly as in the planar case~\cite{BF12}, 
the proof of Theorem~\ref{theo:phi_ori}
is by a reduction (in the dual setting) to the bijection of
Bernardi and Chapuy~\cite{BC11}. In a \emph{rooted map} 
(of genus $g\geq 0$), the convention adopted here  
 is to indicate the root-corner by an artificial ingoing half-edge $\hat{h}$, see the top-left drawing in 
Figure~\ref{fig:bijection_BC}. 
For $M$ a rooted
map, an orientation of $M$ is called a \emph{co-left orientation}
if for any edge $e$ of $M$ there is a (necessarily unique) 
sequence of half-edges $h_1,h_1',\ldots,h_k$ such that 
\begin{itemize} 
\item
$h_1$ is the ingoing part of $e$, and $h_k=\hat{h}$, 
\item for every $i\in[1..k-1]$, $h_i$ and $h_i'$ are opposite on the
  same edge, with $h_i$ the ingoing part and $h_i'$ the outgoing part;
  in addition all the half-edges between $h_i'$ and $h_{i+1}$
  (excluded) in clockwise order around their common incident vertex
  are outgoing.
\end{itemize}

For $g\geq 0$, let $\cRg$ be the family of co-left
orientations of rooted maps of genus $g$.  Bernardi and Chapuy give 
in~\cite[Section 7]{BC11} a bijection between $\cT^{g}_1$ (mobiles of genus $g$ and excess $1$) and $\cRg$. 

We first describe the mapping $\Psi$ from $\cT^{g}_1$ to $\cRg$.  
For $T\in\cT^{g}_1$, the \emph{partial closure} of $T$ is the figure obtained as follows (see the middle drawing
in the top-row of 
Figure~\ref{fig:bijection_BC}): 
\begin{itemize}
\item
for each edge $e=(u,v)\in T$, with $u$ the black extremity
and $v$ the white extremity, insert an ingoing bud 
in the corner just after $e$ in counterclockwise order around $u$
(since $T$ has excess $1$, there are one more ingoing
buds than outgoing buds); 
\item
match the outgoing and ingoing buds according to a walk (with the face on our right)
 around the unique face in $T$, considering 
outgoing buds as opening parentheses and ingoing 
buds as closing parentheses; each matched pair yields a new directed
edge, and the unique unmatched ingoing bud is called \emph{exposed} (in~\cite[Section 7]{BC11}
they call \emph{balanced blossoming mobile} the mobile $T$ plus the unique exposed ingoing bud).  
\end{itemize}

\begin{figure}
\begin{center}
\includegraphics[width=10cm]{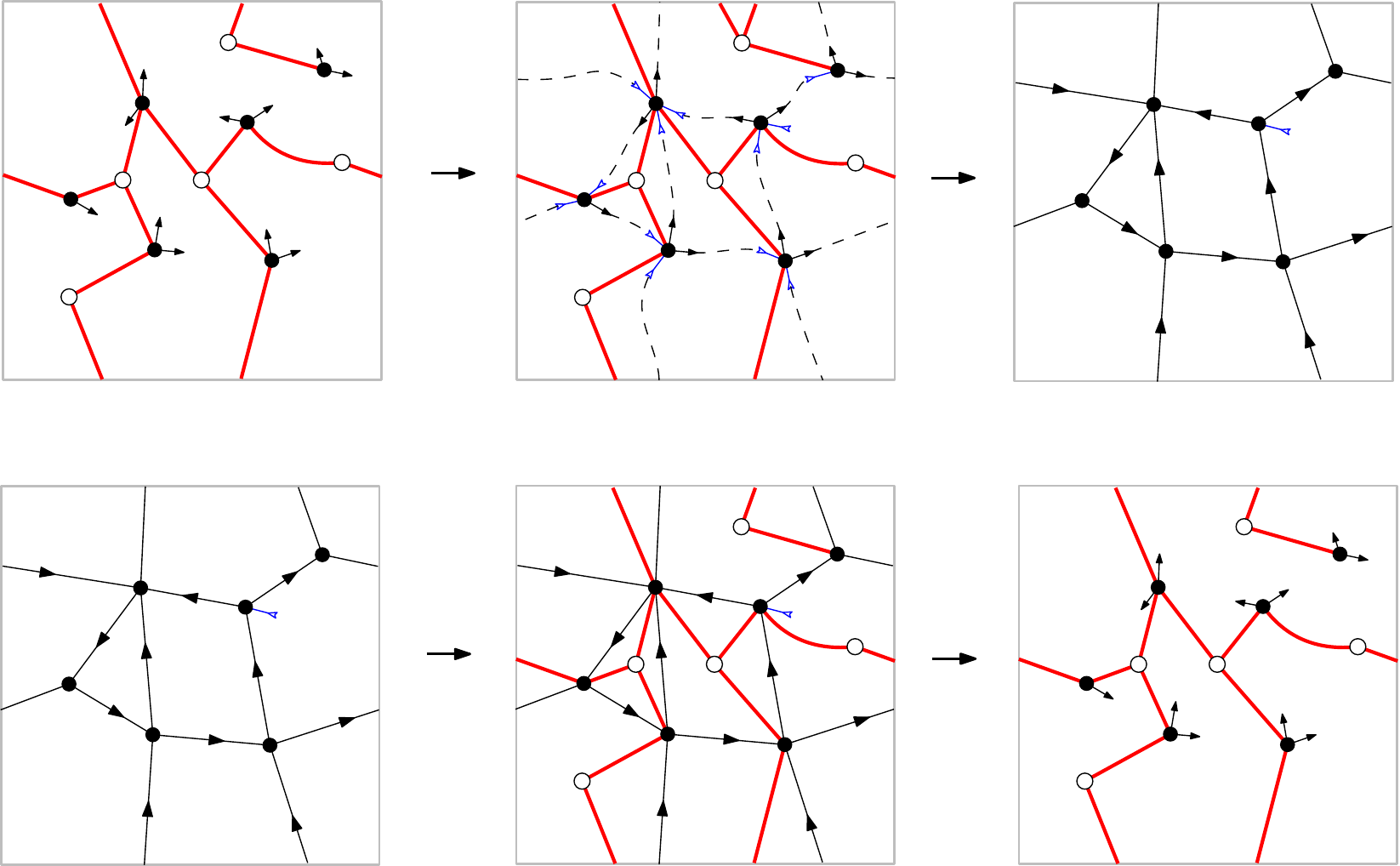}
\end{center}
\caption{The Bernardi-Chapuy bijection  between $\cT_1^{g}$ and $\cRg$ ($g=1$ in the example). 
The top-row shows the mapping $\Psi$ from $\cT^{g}_1$ to $\cR_g$. 
 The bottom-row shows the mapping $\Phi$ from $\cR_g$ to $\cT^{g}_1$.}
\label{fig:bijection_BC}
\end{figure}

Then, $M:=\Psi(T)$ is obtained from the partial closure of $T$ by erasing all the white vertices
of $T$, all the edges of $T$, and declaring the single
exposed ingoing bud as the root of the obtained oriented
map $M$, see the top-row of Figure~\ref{fig:bijection_BC}.  

Conversely, for $M$ an oriented map in
$\cRg$ (whose vertices are considered as black), $T=\Phi(M)$ is
obtained as follows (see the bottom-row of
Figure~\ref{fig:bijection_BC}):
\begin{itemize}
\item
Insert a white vertex in each face of $M$, 
\item
For each ingoing half-edge $h$ of $M$ (including the root half-edge),
create a new edge connecting the vertex incident to $h$ to the 
white vertex in the face on the left of $h$ (looking from $h$
toward the vertex incident to $h$),
\item
Delete all the ingoing half-edges, and declare the outgoing 
half-edges as buds.
\end{itemize}

\subsubsection{Deducing the bijectivity of $\Phi_+$} 
We now explain how the bijectivity of $\Phi_+$ in
Theorem~\ref{theo:phi_ori} can be deduced from properties of the
bijections $\Psi/\Phi$ and properties of the relevant oriented maps. 
\begin{figure}
\begin{center}
\includegraphics[width=10cm]{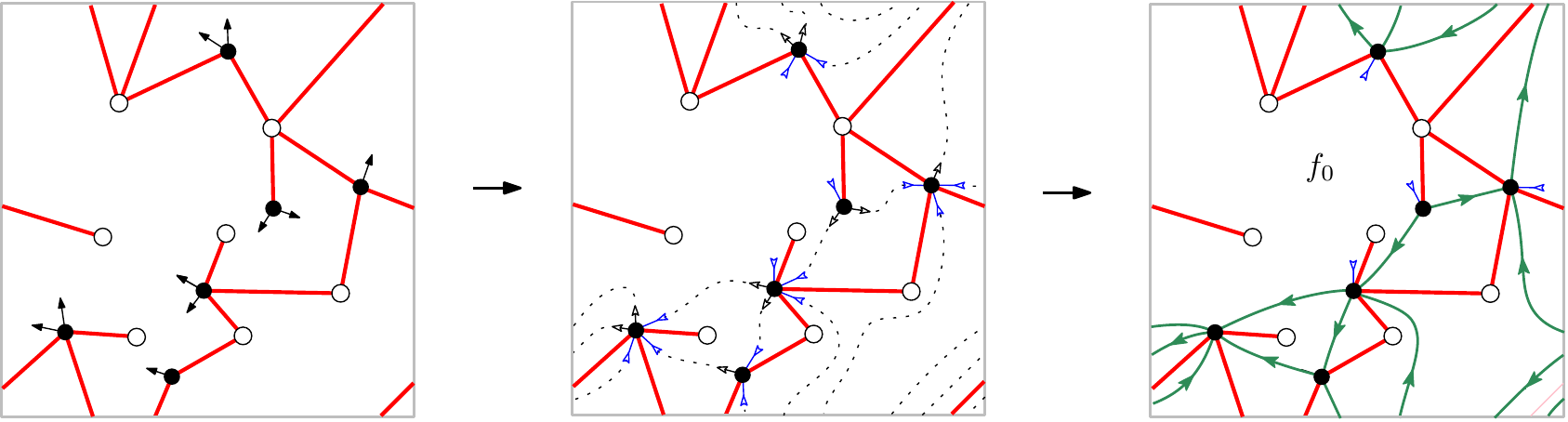}
\end{center}
\caption{The partial closure of a mobile of excess $4$.}
\label{fig:partial_closure}
\end{figure}
A first remark is that, for $d\geq 1$ and $T\in\cTgd$, the \emph{partial closure} of $T$ can be performed exactly in the same 
way as for $d=1$. One obtains a map (made of $T$, the new white vertices, and the new edges created by  
matching outgoing buds with ingoing buds) with $d$ unmatched ingoing buds incident to a same face, see 
Figure~\ref{fig:partial_closure} for an example.

For $d\geq 1$, let $\cRgd$ be the subfamily of $\cRg$ where the
root-vertex has $d$ outgoing half-edges and a single ingoing half-edge
(the root).  For $M\in\cRgd$ let $\iota(M)$ be the underlying
vertex-rooted oriented map (i.e., we delete the root ingoing half-edge
but record that the incident vertex is distinguished), and let $\cSgd$
be the family of vertex-rooted oriented maps of genus $g$ that is the
image of $\cRgd$ by the mapping $\iota$.  For two oriented maps $M,M'$
in $\cRgd$ we write $M\sim M'$ if $\iota(M)=\iota(M')$, so that
$\cSgd\equiv\cRgd/\!\!\sim$. Moreover let $\cUgd$ be
the subfamily of mobiles in $\cT_1^{g}$ that are associated to maps in
$\cRgd$.  Let $T'\in\cUgd$ and let $M'=\Psi(T')$, with $v$ the root-vertex of $M'$.
Since $v$ has indegree $1$ and outdegree $d$ in $M'$, the vertex $v$
is a leaf in $T'$ ---it is incident to a single edge $e$--- with $d$
attached buds. If we delete $v$ together with the attached edge and
buds we clearly obtain a mobile in $\cTgd$; we denote by $\iota(T')$ this mobile. 

For two mobiles $T',U'$ in $\cUgd$ we write $T'\sim U'$ if
$\iota(T')=\iota(U')$.  Conversely, for $T\in\cTgd$, let $G$ be the
 partial closure of $T$; and let $f_0$ be the 
face of $G$ containing the $d$ unmatched ingoing buds. It is easy to
see that $f_0$ has exactly $d$ corner that are at a white vertex;
indeed there is one such corner before each unmatched ingoing bud in a
\cw walk around $f_0$ (i.e., walking with the interior of $f_0$ on the
right).  Then we obtain all the mobiles $T'$ such that $\iota(T')=T$
as follows (see Figure~\ref{fig:reduction_mobiles}): choose a white
corner $c$ in $f_0$, and then attach an edge at $c$ (inside $f_0$)
connected to a new black vertex $v$, and attach $d$ buds at $v$.

\begin{figure}[!h]
\begin{center}
\includegraphics[width=10cm]{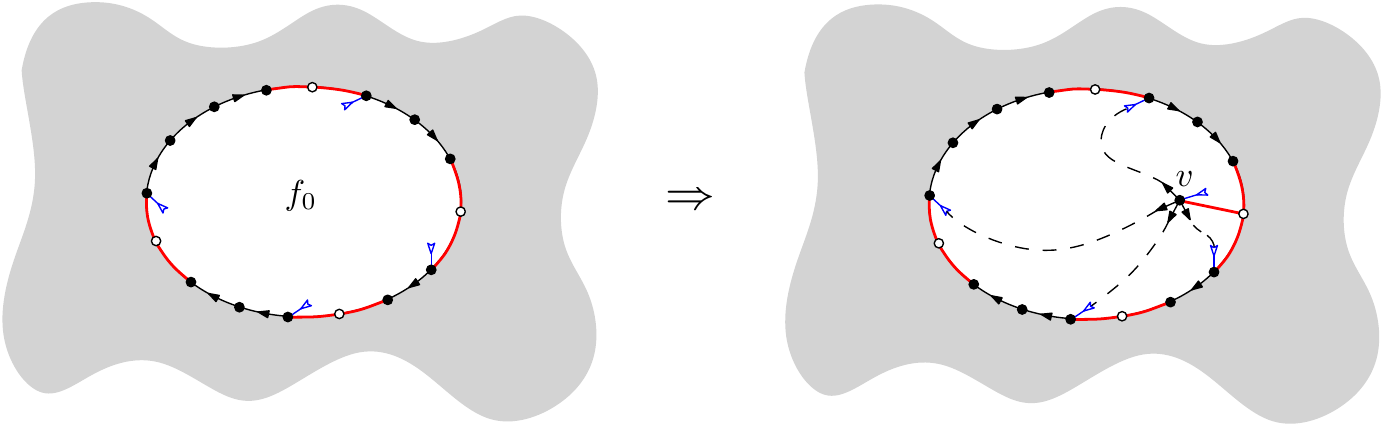}
\end{center}
\caption{Lifting a mobile in $\cTgd$ to a mobile in $\cUgd$.}
\label{fig:reduction_mobiles}
\end{figure}

From the preceding discussion, it is clear that the bijection $\Psi/\Phi$ 
between $\cUgd$ and $\cRgd$ respects the equivalence relations $\sim$,
i.e. $\Phi(M')\sim\Phi(N')$ for $M'\sim N'$ and $\Psi(T')\sim\Psi(U')$ for $T'\sim U'$. 
Since $\cSgd\equiv\cRgd/\!\sim$ and $\cTgd\equiv \cUgd/\!\sim$, we conclude that
 $\Psi/\Phi$ induces a bijection between $\cTgd$ and $\cSgd$.

Moreover the duality property of $\cRg$ (see~\cite[Lemma 8.1]{BC11}) implies   
that $\cOgd$ is the image of $\cSgd$ by duality 
(for $M\in\cSgd$ and $M^*$ the dual face-rooted map, every edge
$e^*\in M^*$ is directed from the left-side to the right-side
of the dual edge $e\in M$).  
Hence $\Phi$ induces a bijection between
$\cOgd$ and $\cTgd$ for every $d\geq 1$, which one can check 
to be precisely the bijection $\Phi_+$ described in Section~\ref{sec:bijPhi+}.

\subsection{Proof of Theorem~\ref{theo:bij_dang}}

\label{sec:proofbijdang}

In this section we prove Theorem~\ref{theo:bij_dang}, which will follow from two lemmas:
the first one (Lemma~\ref{lem:phispebal}) ensuring that the bijection $\Phi_+$ preserves the balancedness property,
and the second one (Lemma~\ref{lem:canonical_dori}) 
ensuring that the maps in $\cF_d$ identify to the face-rooted toroidal maps 
endowed with a balanced $\frac{d}{d-2}$-orientation in $\cO_d^{1}$.  

\subsubsection{Balanced specialization of  $\Phi_+$}
\label{sec:proofbijdangspe}

For $M$ a face-rooted map of genus $g$ (whose vertices are considered
as white), we define the \emph{star-completion} of $M$ as the map
$\aM$ obtained from $M$ by adding a black vertex $v_f$ inside each
non-root face $f$, and connecting $v_f$ to every vertex around $f$
(via every corner around $f$), so that $v_f$ has degree
$\mathrm{deg}(f)$ in $\aM$. The edges of $\aM$ belonging to $M$ are
called \emph{$M$-edges} and the edges incident to black vertices are
called \emph{star-edges}.

Let $d\geq 3$.
Let $M$ be a face-rooted toroidal $d$-angulation endowed with a
$\frac{d}{d-2}$-orientation $X$.  We extend $X$ to an 
$\N$-biorientation $X^\star$ of $\aM$ as follows : for each
half-edge $h$ of $\aM$, if $h$ is part of a $M$-edge, then it has the
same weight (thus the same orientation) as in $X$, if $h$ is part of
a star-edge, then it has weight $0$ if it is incident to a white
vertex, and weight $1$ if it is incident to a black  vertex
(thus star-edges are fully oriented from the black vertex toward the
white vertex).

\begin{lemma}\label{lem:augment}
  Let $M$ be a $d$-toroidal map endowed with a
  $\frac{d}{d-2}$-orientation $X$.  Then $X$ is balanced if and only if $X^\star$ is balanced.
 Moreover, if the $\gamma$-score of two
  non-contractible non-homotopic cycles of
  $\aM$ is $0$, then $X^\star$ is balanced.
\end{lemma}

\begin{proof}
  We start with the case of $d$ odd, which is a bit easier.  Let $M'$
  be the $d$-angulation obtained from $M$ where in each face $f$ of
  $M$ we insert a new vertex $v_f$, called a \emph{star-vertex},
  connected to every corner around $f$ via a path of length
  $\frac{d-1}{2}$, called a \emph{connection-path}, see
  Figure~\ref{fig:augment}(a).  Any $\frac{d}{d-2}$-orientation $X$ of
  $M$ can be extended to a $\frac{d}{d-2}$-orientation $X'$ of $M'$:
  for each connection-path $e_1,\ldots,e_{(d-1)/(2)}$ (which is
  traversed starting from the star-vertex extremity), we give weight
  $2i-1$ (resp. $d-2i-1$) to the first (resp. second) traversed
  half-edge of $e_i$.  

 \begin{figure}[!h]
\begin{center}
\includegraphics[width=12cm]{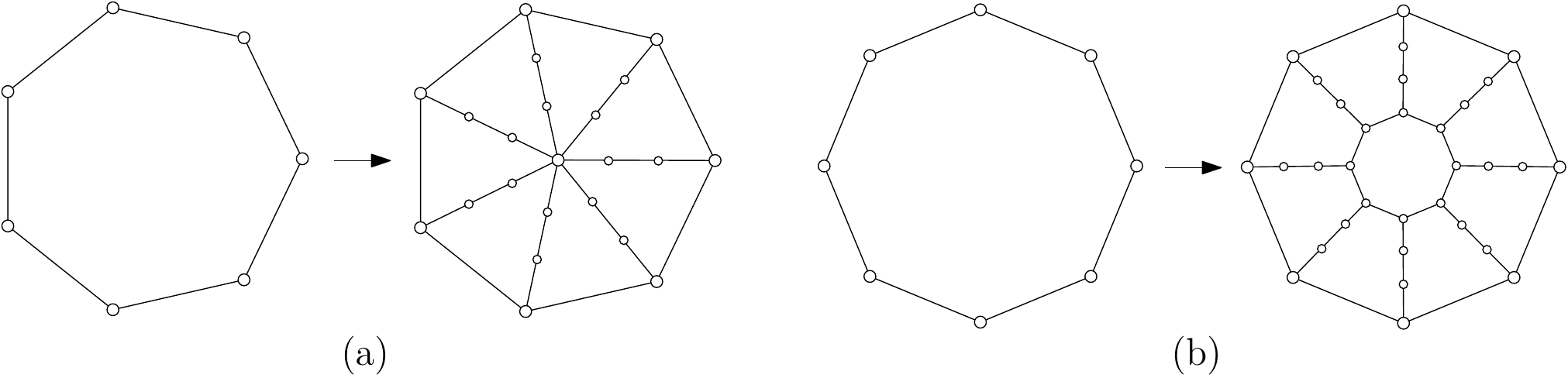}
\end{center}
\caption{Operations within each face for the mapping
from $M$ to $M'$ used in the proof of Lemma~\ref{lem:augment}. 
(a) shows the case of odd $d$;  (b) shows the case of even~$d$.}
\label{fig:augment}
\end{figure}

Note that the connection-paths have weight $0$
  at the incident vertex of $M$, hence for any non-contractible cycle
  $C$ of $M$, the $\gamma$-score of $C$ is the same for $X$ as for
  $X'$. Hence, if $X'$ is balanced, then so is $X$ and the converse
  also holds by Lemma~\ref{lem:gamma0all}.
  Note also that any
  star-edge $e$ of $\aM$ corresponds to a connection-path of
  $M'$. Accordingly any non-contractible cycle $C$ of $\aM$ naturally
  induces a non-contractible cycle $C'$ in $M'$. In addition, since
  the half-edges at the star-vertex extremity in connection-paths have
  weight $1$,   for any non-contractible cycle $C$ of $M^\star$, 
  we have
 $\gamma^{X^\star}(C)=\gamma^{X'}(C')$. So again 
if $X'$ is balanced, then so is $X^\star$ and the converse
  also holds by Lemma~\ref{lem:gamma0all}.  So  $X$ is balanced if and
  only if $X^\star$ is balanced.

  Moreover, if the $\gamma$-score of two non-contractible
  non-homotopic cycles $C_1,C_2$ of $\aM$ is $0$, then the
  $\gamma$-score of the two corresponding cycles in $X'$ is $0$, so by
  Lemma~\ref{lem:gamma0all}, $X'$ is balanced and so is $X^\star$.

  For $d$ even, the augmentation from $M$ to $M'$ is done differently;
  we insert a $d$-gon $D_f$ inside every face $f$, we set a one-to-one
  correspondence between the corners in clockwise order around $f$ and
  the vertices in clockwise order around $D_f$, and we connect any
  matched pair by a path of length $d/2-1$, called a \emph{connection
    path}, see Figure~\ref{fig:augment}(b).  Similarly as before,
  every $\frac{d}{d-2}$-orientation $X$ of $M$ induces a
  $\frac{d}{d-2}$-orientation $X'$ of $M'$: we give weight $d/2-1$ to
  every half-edge of $D_f$, and for each connection-path
  $e_1,\ldots,e_{d/2-1}$ (which is traversed starting from the
  star-vertex extremity), we give weight $2i$ (resp. $d-2i-2$) to the
  first (resp. second) traversed half-edge of $e_i$.

  Similarly as in
  the odd case, the half-edges of connection-paths incident to
  vertices of $M$ have weight $0$, hence for any non-contractible cycle
  $C$ of $M$, the $\gamma$-score of $C$ is the same for $X$ as for
  $X'$.
Hence $X$ is
  balanced if and only if $X'$ is balanced. For $C$ a non-contractible cycle of $\aM$
  together with a traversal direction, let $C'$ be the induced
  cycle of $M'$, with the convention that when $C$ passes by a
  star-vertex $v_f$, then $C'$ takes the left side of the
  corresponding $d$-gon.  Let $f$ be a face of $M$ such that $C$
  passes by the corresponding star-vertex $v_f$, and let $n_{L}(f)$
  (resp. $n_R(f)$) be the number of star-edges on the left
  (resp. right) of $C$ at $v_f$.  Then the contribution to
  $\gamma^{X'}_{L}(C')$ within $f$ is $2n_{L}(f)$, while the
  contribution to $\gamma^{X'}_R(C')$ within $f$ is $d-2$ (due to the
  two half-edges of $D_f$ incident to $C'$ on its right side). Hence
  the contribution to $\gamma^{X'}(C')$ within $f$ is
  $d-2-2n_{L}(f)=n_R(f)-n_{L}(f)$. Since $\gamma^{X^\star}(C)$ and
  $\gamma^{X'}(C')$ have the same contribution within $f$, we conclude
  that $\gamma^{X^\star}(C)=\gamma^{X'}(C')$. From here, the lemma is
  proved in the same way as in the odd case.
\end{proof}

\begin{lemma}
\label{lem:phispebal}  
The mapping $\Phi_+$ specializes into a bijection between face-rooted
toroidal $d$-angulations endowed with a balanced $\frac{d}{d-2}$-orientation in
$\cO_d^{1}$,  and the
family $\mathcal U_d^{Bal}$ of $\NN$-bimobiles.
\end{lemma}
\begin{proof}
As already mentioned, the bijection $\Phi_+$ specializes into a bijection between face-rooted
toroidal $d$-angulations endowed with a $\frac{d}{d-2}$-orientation in
$\cO_d^{1}$,  and the $\NN$-bimobile family
$\mathcal U_d$. We  show here the
``balanced version''  of this bijection.

\begin{figure}
\begin{center}
\includegraphics[width=6cm]{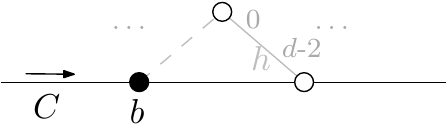}
\end{center}
\caption{Each black vertex $b$ on $C$ corresponds ($1$-to-$1$) to a $C$-adjacent half-edge $h$ in $\cHci_L(C)$ (of weight $d-2$).}
\label{fig:Cadj}
\end{figure}

Let $M$ be a face-rooted
toroidal $d$-angulation endowed with a $\frac{d}{d-2}$-orientation $X$
in
$\cO_d^{1}$. Let $T$ be 
the corresponding $\NN$-bimobile in $\mathcal U_d$ given by $\Phi_+$. Let $C$ be a (non-contractible) cycle
of $T$ together with a traversal direction. Note that $C$
 is also a non-contractible cycle of $\aM$. 
Consider the extension $X^\star$ of $X$ to $\aM$.  Clearly, for any black vertex $u$ on $C$, the 
contribution of $u$ to the left (resp. right) $\gamma$-score of $C$ is the same 
in $\aM$ as in $T$. 
We let $\cHci_L(C)$ be the sef of half-edges of $M$ that are on the left of $C$ and incident to a white vertex on $C$. 
A half-edge $h$ in $\cHci_L(C)$, with $v$ its incident vertex,
 is called \emph{$C$-adjacent} if the next half-edge in $\aM$ in ccw order around $v$ is on $C$; it is called
 \emph{$C$-internal} otherwise. 
Then the local rules of Figure~\ref{fig:local_rule_weighted_biori} ensure that $\gamma_L^T(C)$ gives the 
total contribution to $\gamma_{L}^{X^\star}(C)$ by $C$-internal half-edges in $\cHci_L(C)$.  
We let $A_L(C)$ be the total contribution to $\gamma_{L}^{X^\star}(C)$ by $C$-adjacent half-edges in $\cHci_L(C)$.
As shown in Figure~\ref{fig:Cadj}, each black vertex on $C$ yields a contribution $d-2$ to $A_L(C)$, hence $A_L(C)=(d-2)\nb(C)$.  
We conclude that $\gamma_{L}^{X^\star}(C)=\gamma_{L}^{T}(C)+(d-2)n_{\bullet}(C)$. Similarly we have 
$\gamma_{R}^{X^\star}(C)=\gamma_{R}^{T}(C)+(d-2)n_{\bullet}(C)$, so that $\gamma^{X^\star}(C)=\gamma^{T}(C)$ for any
non-contractible cycle $C$ of $T$. Hence, if $X$ is balanced, then, by
Lemma~\ref{lem:augment}, so is $X^\star$ and so is $T$.
Conversely, if $T$ is balanced, then it has  two
non-contractible non-homotopic cycles with $\gamma$-score equal to
zero. Hence, by what precedes, $X^\star$ has also $\gamma$-score equal
to zero on these two cycles.  
Then Lemma~\ref{lem:augment} ensures that $X^\star$ is balanced, and so is $X$.  
\end{proof}

\subsubsection{Properties of rightmost walks}
\label{sec:right}

Consider a face-rooted $d$-toroidal map $M$.

We have
the following crucial lemma regarding rightmost walks in $\frac{d}{d-2}$-orientations of
$M$:

\begin{lemma}
\label{lem:facialwalktilde}
In a balanced $\frac{d}{d-2}$-orientation of $M$,
any rightmost walk of $M$ eventually loops on the contour of a $d$-angle $W$ with
the (contractible) interior of $W$  on its right side.
\end{lemma}

\begin{proof} Let $W$ be the looping part of a
   rightmost path. Note that $W$ is a non-repetitive closed walk,  
and it cannot cross itself, otherwise it is not a rightmost walk. However $W$
  may have repeated vertices but in that case $W$ intersects itself
  tangentially on the left side.

 Let $(e_1,\ldots,e_p)$ be the cyclic list of edges in $W$.
 Suppose by contradiction that there is an oriented subwalk $W'=e_i,\ldots,e_{(i+k')\ \mathrm{mod}\ p}$ of
  $W$ (possibly $W'=W$) that forms a closed walk (i.e., the head of the last edge is the same as 
the tail of the first edge of $W'$) enclosing on its left side a region $R$ homeomorphic to an open disk.  
Let $v$ be the starting and ending vertex of $W'$.  
Let $H$ be the planar map obtained from $M$ by keeping $R\cup W'$, where $W'$ (which may visit vertices repeated times, 
but only `from the outside') is turned into a cycle of length $k'$, the outer cycle of $H$. 
Let $n',m',f'$ be the numbers of vertices, edges and faces of $G$.  By
  Euler's formula, $n'-m'+f'=2$. All the inner faces of $H$ have degree $d$ and
  the outer face has degree $k'$, so $2m'=d(f'-1)+k'$. Since $W'$ is a
  subwalk of a rightmost walk, 
all the half-edges that are not in $H$ and
  incident to a vertex $v'\neq v$ on $W'$ have weight
  zero.  The first half-edge of $W'$ has non-zero weight.  Thus, as we
  are considering a $\frac{d}{d-2}$-orientation, we have
  $(d-2)m'\geq d(n'-1)+1$.  Combining these three (in)equalities gives
  $k'\leq -1$, a contradiction.

We have the following crucial property:

\begin{claimn}
\label{cl:leftsidedisk}
  The right side of $W$ encloses a region  homeomorphic to an open disk
\end{claimn}

\begin{proofclaim}
  We consider two cases depending on the fact that $W$ is a cycle
  (i.e., with no repetition of vertices) or not.

\begin{itemize}
\item \emph{$W$ is a cycle} 

  Suppose by contradiction that $W$ is a non-contractible cycle $C$. Let $k$ be its length. 
  Since $W$ is a rightmost walk, all the half-edges incident to the right
  side of $C$ have weight $0$.  Since we are considering a
  $\frac{d}{d-2}$-orientation of $M$, the sum of the weights of all
  edges of $W$ is $(d-2)k$ and the sum of the
  weights of all the half-edges incident to vertices of $W$ is
  $dk$. So finally the sum of the weights of all the half-edges
  incident to the left side of $C$ is $2k$ and we have $\gamma(C)=-2k<0$.
  So the orientation is not balanced, a contradiction. 

  Thus $W$ is a contractible cycle. By previous arguments, the
  contractible cycle $W$ does not enclose a region homeomorphic to an
  open disk on its left side. So $W$ encloses a region homeomorphic to
  an open disk on its right side, as claimed.

\item \emph{$W$ is not a cycle} 

  Since $W$ cannot cross itself nor intersect itself tangentially on
  the right side, it has to intersect tangentially on the left side.
  Such an intersection can be on a single vertex or a path, as
  depicted on Figure~\ref{fig:righttangent}(i). The edges of $W$
  incident to this intersection are noted as on figure (i)--(iv),
  where $W$ is going periodically through $a,b,c,d$ in this order.  By
  previous arguments, the (green) subwalk of $W$ from $a$ to $b$ does
  not enclose regions homeomorphic to  open disks on its left
  side. So we are not in the case depicted on
  Figure~\ref{fig:righttangent}(ii). Moreover if this (green) subwalk
  encloses a region homeomorphic to an open disk on its right side,
  then this region contains the (red) subwalk of $W$ from $c$ to $d$,
  see Figure~\ref{fig:righttangent}(iii). Since $W$ cannot cross
  itself, this (red) subwalk necessarily encloses  regions
  homeomorphic to  open disks on its left side, a contradiction. So
  the (green) subwalk of $W$ starting from $a$ has to form a
  non-contractible curve before reaching $b$. Similarly for the (red)
  subwalk starting from $c$ and reaching $d$. Since $W$ is a rightmost
  walk
  and cannot cross itself, we are, without loss of generality, in the situation of
  Figure~\ref{fig:righttangent}(iv) (with possibly more tangent
  intersections on the left side). In any case, $W$
  encloses a region homeomorphic to an open disk on its right side.

 \begin{figure}[!h]
 \center
 \begin{tabular}{cc}
 \includegraphics[scale=0.3]{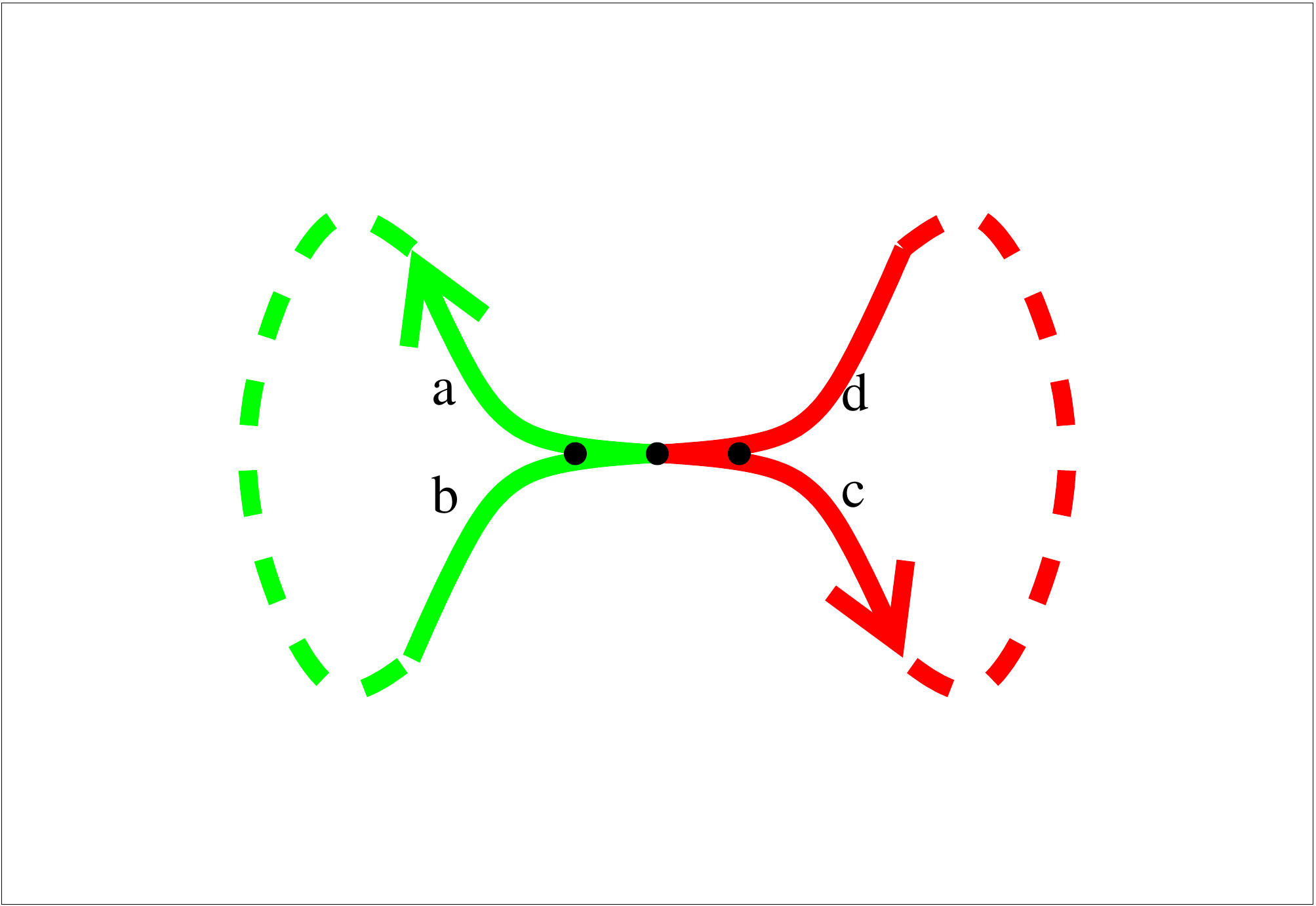} \ \ & \ \ 
 \includegraphics[scale=0.3]{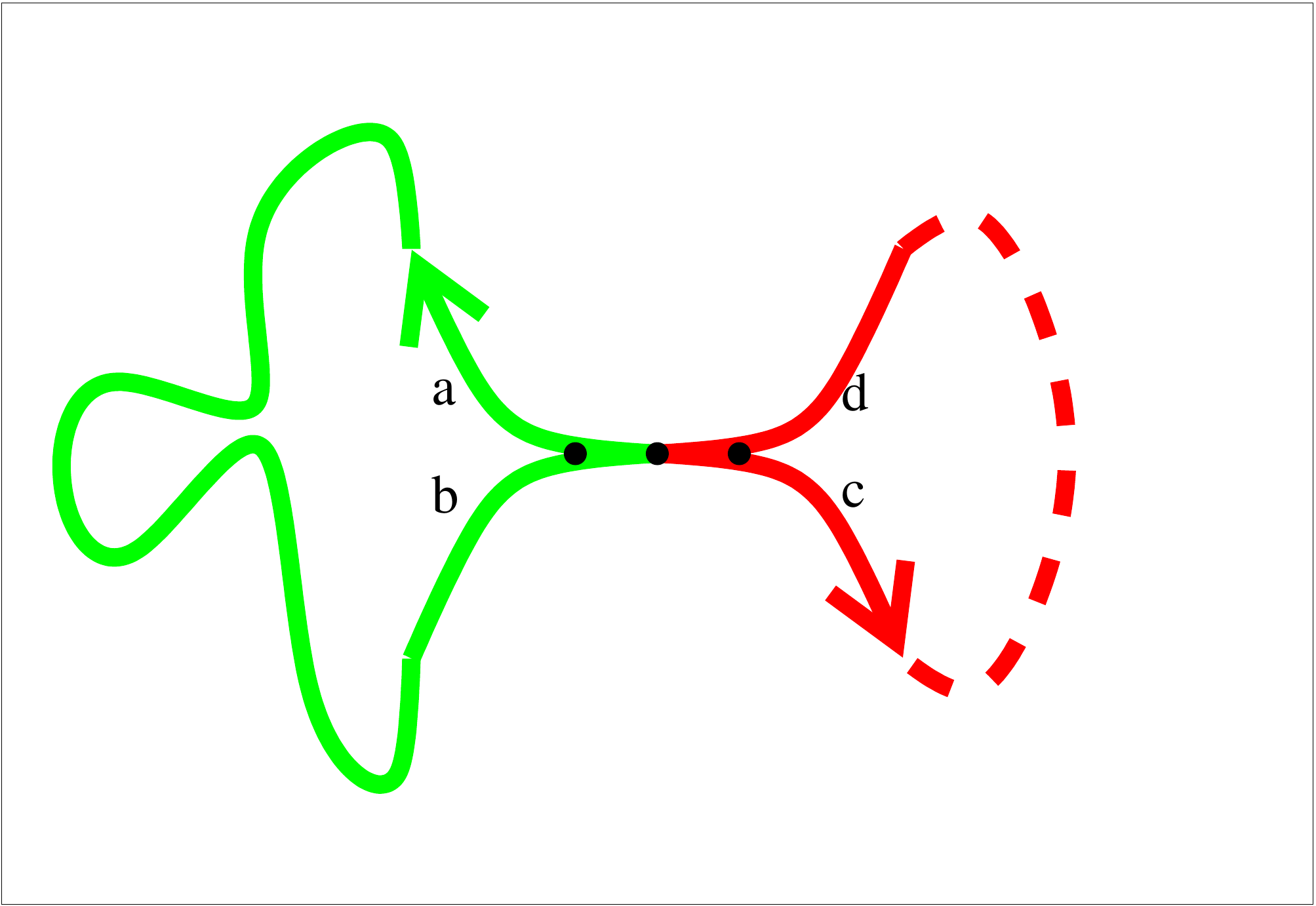}\\
(i) \ \ & \ \ (ii)\\
& \\
 \includegraphics[scale=0.3]{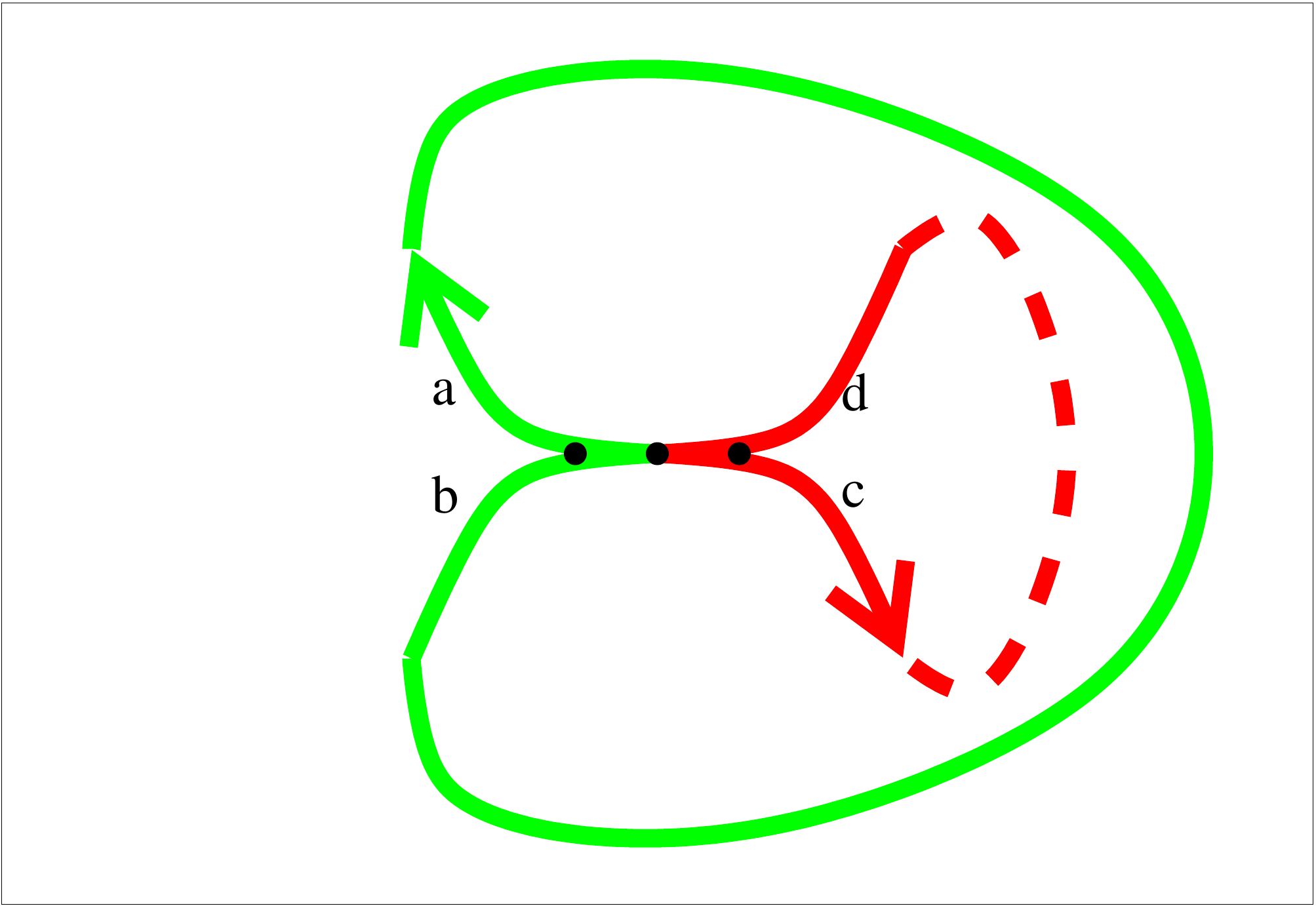}\ \ & \ \ 
 \includegraphics[scale=0.3]{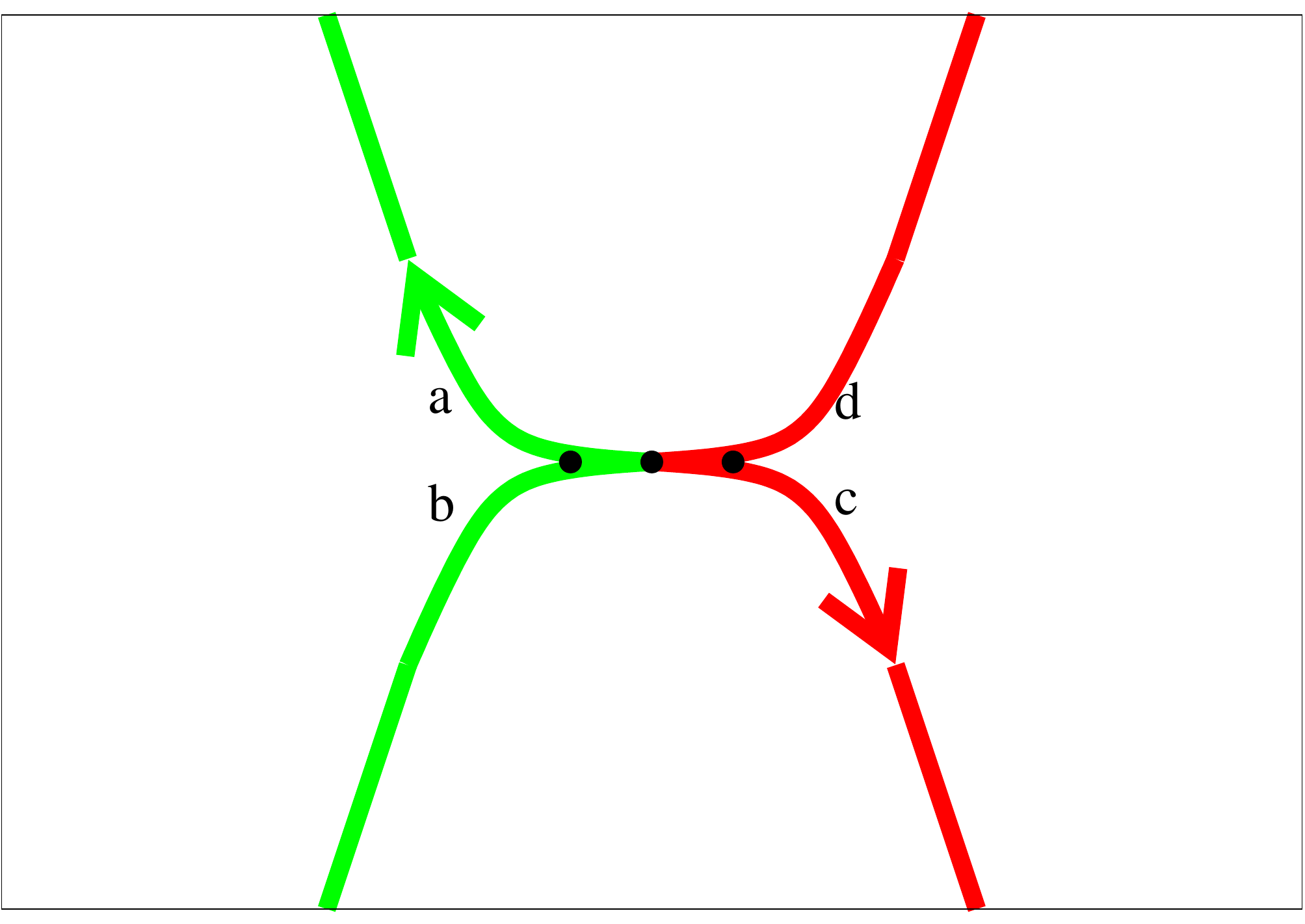}\\
(iii) \ \ & \ \ (iv)\\
 \end{tabular}
\caption{Case analysis for the proof of the claim in Lemma~\ref{lem:facialwalktilde}.}
 \label{fig:righttangent}
 \end{figure}

\end{itemize}

\end{proofclaim}

The claim ensures that $W$ encloses a region $R$ homeomorphic to an open
disk on its right side. Since $W$ is a rightmost walk, there is no outgoing half-edge in $R$ whose incident vertex is on $W$. Hence, by Claim~\ref{claim:epsilon}, we conclude that $W$ has length $d$. 
\end{proof}

Recall that $\cF_d$ is the family of face-rooted $d$-toroidal maps such that the root-face contour is a maximal $d$-angle. 

\begin{lemma}\label{lem:canonical_dori}
A face-rooted toroidal $d$-angulation $M$
has a balanced  $\frac{d}{d-2}$-orientation in $\cO_d^{1}$ if and
only if  $M\in\cF_d$.
In that case, $M$ has a unique balanced  $\frac{d}{d-2}$-orientation 
 in $\cO_d^{1}$, which is the minimal one.
\end{lemma}

\begin{proof}
  $(\Longrightarrow)$ Suppose that $M$ has a balanced
  $\frac{d}{d-2}$-orientation $O$ in $\cO_d^{1}$. Then, by
  Lemma~\ref{th:existencefractional}, $M$ has essential girth $d$.
  
  Suppose by contradiction that the contour $C_0$ of the root-face
  $f_0$ is not a maximal $d$-angle.  Consider a maximal $d$-angle
  $C_{\max}$ whose interior strictly contains the interior of $C_0$. Consider an outgoing
  half-edge $h$ of $C_{\max}$ and the rightmost walk $W$ started from
  $h$.  By Claim~\ref{claim:epsilon}, all the half-edges that are in the
  interior of $C_{\max}$ and incident to it have weight zero, i.e. they
  are ingoing at their incident vertex. So it is not possible that $W$
  enters in the interior of $C_{\max}$. So $W$ does not loop on $C_0$, a
  contradiction.  So $M\in\cF_d$.

$(\Longleftarrow)$ Suppose that $M\in\cF_d$.  By
Proposition~\ref{th:existencebalanced}, $M$ admits a balanced
$\frac{d}{d-2}$-orientation.
Then by Corollary~\ref{theo:general_gamma}, $M$ has a (unique) balanced
$\frac{d}{d-2}$-orientation $D_{\min}$
that is minimal.

Let us prove that $D_{\min}\in\cO_d^{1}$.  Consider an outgoing
half-edge $h$ of $D_{\min}$ and the rightmost walk $W$ starting from
$h$.  By Lemma~\ref{lem:facialwalktilde}, $W$ ends on a $d$-angle $W'$
with its interior $R$ on the right side.  Consider the
$(d-2)$-expansion $M'$ of $M$ and the orientation $D'_{\min}$ of $M'$
corresponding to $D_{\min}$ (see Section~\ref{sec:preliminaries} for
the definition of $\beta$-expansion). Let $S$ be the set of faces
corresponding to the region $R$ in $M'$.  The set $S$ is such that
every edge on the boundary of $S$ has a face in $S$ on its right.
Since $D'_{\min}$ is minimal, $S$ contains the face of
$M'$ corresponding to the root face $f_0$. Since $M \in \cF_d$, 
 the contour of $f_0$ is a maximal $d$-angle.  So $W'$ is
indeed the contour of $f_0$ with $f_0$ on its right.  So
$D_{\min}\in\cO_d^{1}$.

Moreover, suppose by contradiction, that $M$ has a balanced
$\frac{d}{d-2}$-orientation $D$ in $\cO_d^{1}$ that is different from
$D_{\min}$. By unicity of the balanced $\frac{d}{d-2}$-orientation
that is minimal (Corollary~\ref{theo:general_gamma}), we have that $D$
is non-minimal, a contradiction to Lemma~\ref{lem:necmin}.
\end{proof}

\subsection{Proof of Theorems~\ref{theo:bij_maps_2b}
  and~\ref{theo:bij_maps_d}}
\label{sec:proofs_theorems}

  \subsubsection{Proof of Theorem~\ref{theo:bij_maps_2b} for $b\geq    2$} 
\label{sec:proof_b_g_2}
We start by giving some terminology and results for $b\geq 1$, before
continuing with $b\geq 2$ in the rest of the section.

Let $b\geq 1$. Let $\cE_{2b}$ be the family of face-rooted 
toroidal maps with
root-face  of degree exactly $2b$ and with
 all  face-degrees even and at least
$2b$.

Recall from  Section~\ref{sec:weightbi},
that a $\ZZ$-biorientation has the weights at outgoing half-edges
that are in $\ZZ_{>0}$ while the weights at ingoing half-edges are in $\ZZ_{\leq 0}$. In
an $\NN$-biorientation all the ingoing half-edges have weight~$0$.

For $M\in \cE_{2b}$, we define a \emph{$\bb$-$\zZ$-orientation} of $M$ as a
$\ZZ$-biorientation of $M$ with weights in $\{-1,\ldots,b\}$, such that
each vertex has weight $b$, each edge has weight $b-1$, and each face
$f$ has weight $-\tfrac1{2}\mathrm{deg}(f)+b$.  Recall that the weight
of a face $f$ is the sum of the weights of the ingoing half-edges that
have $f$ on their left (traversing the half-edge toward its incident
vertex).

The bijection $\Phi_+$ specializes into a bijection between maps in $\cE_{2b}$
endowed with a $\bb$-$\zZ$-orientation in $\cO_{2b}^1$
and the family $\hat{\mathcal V}_b$ of
  toroidal $\frac{b}{b-1}$-$\ZZ$-mobiles, as defined in Section~\ref{sec:bijdangul}.  
Showing Theorem~\ref{theo:bij_maps_2b} for $b\geq 2$ thus amounts to proving
the following statement:

\begin{proposition}\label{prop:unique_b_at_least_2}
Let $b\geq 2$ and  $M$ be a map in $\cE_{2b}$. Then 
$M$ admits a $\bb$-$\zZ$-orientation 
in $\cO_{2b}^1$ whose associated mobile by $\Phi_+$  is in $\hat{\mathcal V}_b^{Bal}$ if and only if $M$ is in
$\chL_{2b} $.

In that case $M$ admits a unique such orientation. 
\end{proposition}

The rest of this section is devoted to proving
Proposition~\ref{prop:unique_b_at_least_2}.
Similarly as in the planar
case~\cite{BF12b} we  work with closely related orientations
called $b$-regular orientations.

Let $b\geq 2$ and $M\in\cE_{2b}$. Let $\aM$ be
the star-completion of $M$, as defined in
Section~\ref{sec:proofbijdangspe}.  A \emph{$b$-regular orientation}
of $\aM$ is defined as an $\NN$-biorientation of $\aM$ such that every
$M$-edge has weight $b-1$, every star-edge has weight $1$ (hence is a
simply oriented edge), every $M$-vertex has weight $b$, and every
star-vertex $u$ has weight (i.e., outdegree)
$\tfrac1{2}\mathrm{deg}(u)+b$ (hence indegree
$\tfrac1{2}\mathrm{deg}(u)-b$).

A $b$-regular orientation of $\aM$
is called \emph{transferable} if for each star-edge $e$ directed out of 
its incident $M$-vertex $v$, the $M$-edge $\eps$ just after $e$ in clockwise order
around $v$ is of weight $b-1$ at $v$ (and thus weight $0$ at the other
half-edge).  

For a transferable $b$-regular orientation $X$ of $\aM$, the 
\emph{induced} $\frac{b}{b-1}$-$\ZZ$-orientation $Y=\sigma(X)$ of $M$ is obtained
by applying the weight-transfer 
rules of Figure~\ref{fig:rule_sigma} to each
star-edge going toward its black extremity, and then
deleting the star-edges and black vertices.

\begin{figure}[!h]
\begin{center}
\includegraphics[width=8cm]{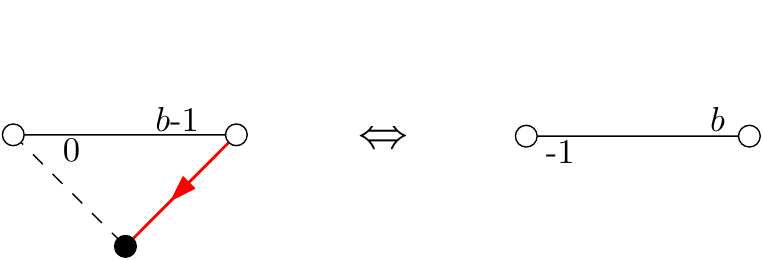}
\end{center}
\caption{The local rules of the 1-to-1 correspondence $\sigma$ 
to turn a transferable $b$-regular orientation 
of $\aM$ into a $\frac{b}{b-1}$-$\ZZ$-orientations of $M$ (the half-edge directions on the $M$-edges are not indicated, these are determined by the status of the weights, either in 
$\ZZ_{>0}$ or in $\ZZ_{\leq 0}$); after applying these rules
the star-vertices and star-edges of $\aM$ are to be deleted.} 
\label{fig:rule_sigma}
\end{figure}

\begin{lemma}\label{lem:bij_transferable}
  The mapping $\sigma$ is a bijection from the transferable
  $b$-regular orientations of $\aM$ to the
  $\frac{b}{b-1}$-$\ZZ$-orientations of $M$. In addition a
  transferable $b$-regular orientation $X$ is in $\cO_{2b}^{1}$ if and
  only if $\sigma(X)$ is in $\cO_{2b}^{1}$.
\end{lemma}

\begin{proof}
The bijectivity of the mapping is straightforward.   
And the second statement 
 follows from the observation that if $X$
is a transferable $b$-regular orientation, then 
the rightmost walk $P_e$ starting at any edge $e\in X$ 
will only pass by $M$-edges after reaching an $M$-vertex $w$ for the 
first time (indeed,  
when entering an $M$-vertex $v$, 
the rightmost outgoing edge to leave $v$ 
can not be a star-edge since the orientation is transferable).
Once it has reached an $M$-edge $e'$, it will follow a rightmost walk 
$P_{e'}$ that consists
only of $M$-edges, the rightmost walk $P_{e'}$
being exactly the same in $X$ as in $\sigma(X)$, 
hence $P_{e'}$ eventually loops around the root-face contour in $X$
if and only if the same holds in $\sigma(X)$.   
\end{proof}

Let $X$ be a $b$-regular orientation of $\aM$.  For $C$ a
non-contractible cycle of $\aM$ traversed in a given direction, let
$w_{L}(C)$ (resp. $w_R(C)$) be the total weight of half-edges incident
to an $M$-vertex of $C$ from the left side (resp. right side) of $C$,
and let $o_{L}(C)$ (resp. $o_R(C)$) be the total number of outgoing
star-edges incident to a star-vertex on $C$ on the left side
(resp. right side) of $C$, and let $\iota_{L}(C)$ (resp. $\iota_R(C)$)
be the total number of ingoing star-edges incident to a star-vertex on
$C$ on the left side (resp. right side) of $C$.  Let
$\hgamma_{L}(C)=2w_{L}(C)+o_{L}(C)-\iota_L(C)$,
$\hgamma_{R}(C)=2w_{R}(C)+o_{R}(C)-\iota_R(C)$. We define the
\emph{$\hgamma$-score} of $C$ as
$\hgamma(C)=\hgamma_{R}(C)-\hgamma_{L}(C)$.  Then $X$ is called
\emph{$\hgamma$-balanced} if the $\hgamma$-score of any 
 non-contractible cycle $C$ 
of $\aM$ is $0$.


\begin{lemma}
  \label{lem:gammahgamma}
  Consider two $b$-regular orientations $X,X'$ of $\aM$ and $C$ a
non-contractible cycle of $\aM$ traversed in a given direction.
The cycle $C$ has the same $\gamma$-score in $X$ and $X'$
if and only if it has the same $\hgamma$-score in $X$ and $X'$.
\end{lemma}
\begin{proof}
Let $s_L(C)$ (resp. $s_R(C)$) be the total number of  
star-edges incident to a star-vertex of $C$ on the 
left (resp. right) side of $C$. 
Note that we have 
$$
\hgamma_L^X(C)=2\gamma_L^X(C)-s_L(C),\ \ \ \hgamma_R^X(C)=2\gamma_R^X(C)-s_R(C). 
$$
Hence $\hgamma^X(C)=2\gamma^X(C)+(s_L(C)-s_R(C))$,
where we note that the quantity $s_L(C)-s_R(C)$
only depends on $M^\star$ and $C$ (not on the orientation $X$).  
\end{proof}

Let $\hat{\cM}_{2b}$ be the subfamily of maps in $\cE_{2b}$ that are
bipartite and of essential girth $2b$.

\begin{lemma}\label{lem:exists_bregular}
  Let $M$ be a map in $\cE_{2b}$. If $\aM$ admits a $b$-regular
orientation, then $M$ has essential girth $2b$. 
Moreover $\aM$ admits a \emph{$\hgamma$-balanced} $b$-regular orientation if and only
  if $M\in\hat{\cM}_{2b}$ (i.e., is bipartite of essential girth $2b$).
\end{lemma}
\begin{proof} 
  Assume $\aM$ is endowed with a $b$-regular orientation $X$, and let us
  show that $M$ has essential girth $2b$. Since the root-face has
  degree $2b$, the essential girth is at most $2b$, hence we just have
  to show that the essential girth is at least $2b$. 
Consider a contractible closed walk $C$ in $M$. Let $M_C$ be the planar 
map obtained by keeping $C$ and its interior, where $C$ is `unfolded' 
into a simple cycle, taken as the outer face contour. Since all inner face-degrees in
  $M_C$ are even, the outer face degree is also even, so that the length of $C$ is
  an even number, denoted $2k$, and we have to prove that $b\leq k$.
  Let $v,f,e$ be respectively the numbers of vertices,
  edges, and faces that are strictly inside $C$. By Euler's formula applied to $M_C$ we
  have $v-e+f=1$. Let $2S$ be the sum of the degrees of the faces
  inside $C$. Note that $2S=2e+2k$, i.e., $S=e+k$.  Consider the
  extension of the $b$-regular orientation $X$ to the interior of
  $C$. The total weight over all $M$-vertices strictly inside $C$ is
  $bv$ and the total weight (outdegree) over all star-vertices
  strictly inside $C$ is $S+bf$.  On the other hand the total weight
  over all edges that are strictly inside $C$ is $(b-1)e+2S$. Note
  also that the total weight over all edges strictly inside $C$ must
  be at least the total weight over all vertices strictly inside
  $C$. Hence we must have $bv+S+bf\leq (b-1)e+2S$, so that
  $b(v-e+f)\leq S-e$. But we have seen that $v-e+f=1$ and $S=e+k$,
  hence we obtain $b\leq k$.

  Now we show that if $\aM$ can be endowed with a $\hgamma$-balanced
  $b$-regular orientation $X$, then $M$ is bipartite.  Let $C$ be a
  non-contractible cycle of $M$, and let $k$ be the length of $C$.  We
  also denote by $C$ the corresponding cycle in $\aM$ (that is going
  through $M$-vertices only). We have
  $\hgamma(C)=w_R(C)-w_L(C)$. Since the orientation is
  $\hgamma$-balanced, we have $\hgamma(C)=0$ and thus
  $w_R(C)=w_L(C)$. Since every vertex on $C$ has weight $b$ and every
  edge on $C$ has weight $b-1$ we have
  $w_L(C)+w_R(C)+k\,(b-1)=k\,b$. So finally $k=2\,w_R(C)$ is
  even. Since all face-degrees of $M$ are even and all
  non-contractible cycles have even length, we conclude that $M$ is
  bipartite. So  $M\in\hat{\cM}_{2b}$.

  It now remains to show that if $M\in\hat{\cM}_{2b}$  then $\aM$
  admits a $\hgamma$-balanced $b$-regular orientation. Our strategy is
  the toroidal counterpart of the one for planar maps given
  in~\cite[Prop.~47]{BF12b}.  We define the \emph{$2b$-angular lift of $M$} as the
  bipartite toroidal $2b$-angulation $M'$ obtained by the following
  process.  We first fix for each $\ell\geq b$ an arbitrary planar map
  $M_{\ell}$ of girth $2b$, where the outer face has degree $2\ell$ and its contour is a cycle, and all inner faces have 
  degree $2b$.  Then in each non-root face $f$ of
  $M$, with $\ell=\mathrm{deg}(f)/2$, we insert a copy $Q_f$ of
  $M_{\ell}$ strictly inside $f$, we set a one-to-one correspondence
  between the corners in clockwise order around $f$ and the outer
  vertices of $Q_f$ in clockwise order around $Q_f$, and we connect
  any matched pair by a path of length $b-1$, called a
  \emph{connection path}, see Figure~\ref{fig:lift} for an example.
  Since $M$ is bipartite and all the faces inserted inside each face
  of $M$ have even degree, then $M'$ is bipartite as
  well.  And similarly as in the planar case~\cite{BF12b} it is 
  easy to check that $M'$ has essential girth $2b$.  Hence, by
  Proposition~\ref{th:evencase}, $M'$ can be endowed with a balanced
  $\frac{b}{b-1}$-orientation $X'$.  For $P$ a connection-path within
  a face $f$ of $M$, let $h$ the extremal half-edge of $P$ touching a
  vertex of $f$ and let $h'$ be the extremal half-edge of $P$ touching
  a vertex of $Q_f$. Then it is easy to see that the respective
  weights of $\{h,h'\}$ are either $\{0,1\}$ or $\{1,0\}$. The
  connection-path $P$ is called \emph{outgoing} (resp. \emph{ingoing})
  in the first (resp. second) case. By a simple counting argument
  using Euler's formula, one can check that for each non-root face $f$
  of $M$ of degree $2k$, the number of connection-paths inside $f$
  that are outgoing (resp. ingoing) is $k+b$ (resp. $k-b$).  Hence, if
  for each non-root face $f$ of $M$ we contract $Q_f$ into a black
  vertex $u_f$ and turn every outgoing (resp. ingoing) connection-path
  within $f$ into an edge of weight $1$ directed out of $u_f$
  (resp. toward $u_f$), we obtain a $b$-regular orientation $X$ of
  $\aM$.  It now remains to show that $X$ is $\hgamma$-balanced.

\begin{figure}[!t]
\begin{center}
\includegraphics[width=12cm]{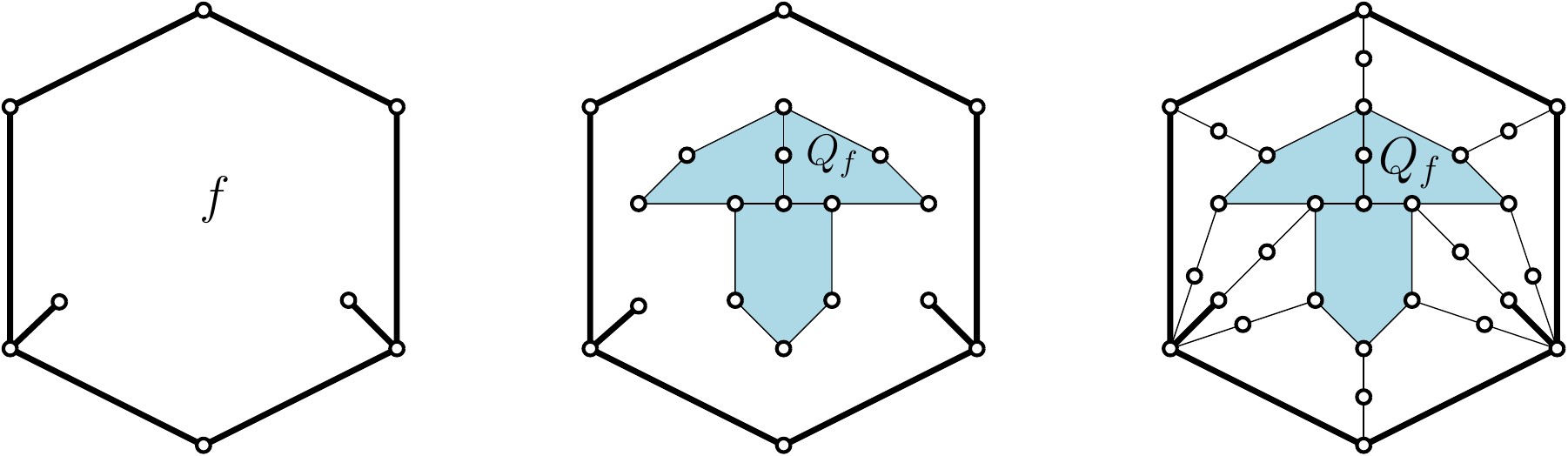}
\end{center}
\caption{Insertion operations to obtain a $2b$-angular lift (case $b=3$ here): in each 
face $f$ of degree $2p$ ($p=5$ here) a map $Q_f$ having outer degree $2p$, 
inner face degrees $2b$, and girth $2b$, is inserted inside $f$, and each outer vertex of $Q_f$
is connected to each corner around $f$ by a path of length $b-1$ called a \emph{connection-path}.}  
\label{fig:lift}
\end{figure} 

Let $e_1,e_2$ be a pair of star-edges incident to a same 
star-vertex $u$, and let $f$ be the face of $M$ corresponding to $u$,
and $v_1,v_2$ the respective white extremities of $e_1,e_2$. 
We let $P(e_1,e_2)$ denote an  arbitrarily selected path of $M'$,
 with $v_1,v_2$ as extremities and staying strictly in (the 
area corresponding to) $f$ in between. 
For each non-contractible
cycle $C$ of $\aM$, we call the \emph{canonical lift of $C$ to $M'$}, the cycle
$C'$ of $M'$ obtained as follows: $C'$ has
the same $M$-edges as in $C$, and for each star-vertex
$u$ on $C$, with $e_1,e_2$ the edges before and after 
$u$ along $C$, we replace the pair $e_1,e_2$ by the path
$P(e_1,e_2)$. Since $X'$ is balanced we have $\gamma^{X'}(C')=0$.
We want to deduce from it that $\hgamma^{X}(C)=0$. 

Let $S$ be the set of star-vertices on $C$. 
For every $u\in S$, let $f$ be the corresponding
face of $M$, let $v_1$ (resp. $e_1$) be the vertex (resp. edge)
before $u$ along $C$, and let $v_2$ (resp. $e_2$) 
be the vertex (resp. edge)
after $u$ along $C$. 
Let $o_L(u)$ (resp. $\iota_L(u)$) 
 be the number of outgoing (resp. ingoing) edges incident to $u$
 on the left side of $C$, and let $o_{R}(u)$ (resp. $\iota_{R}(u)$) 
 be the number of outgoing (resp. ingoing) edges incident to $u$
 on the right side of $C$. And let $w_L(u)$ (resp. $w_R(u)$) 
be the total weight in $X'$ of half-edges incident to vertices
of the path 
$P(e_1,e_2)\backslash\{v_1,v_2\}$ on its left (resp. right) side. 
Note that we have:

$$
\hgamma_L^{X}(C)-2\gamma_L^{X'}(C')
=\sum_{u\in S}o_L(u)-\iota_L(u)-2w_L(u)
$$
$$
\hgamma_{R}^{X}(C)-2\gamma_{R}^{X'}(C')
=\sum_{u\in S}o_{R}(u)-\iota_R(u)-2w_{R}(u).
$$

For $u\in S$ and $f$ the corresponding face of $M$, 
the cycle  $C$ splits the contour of $f$ into a left portion denoted
$P_{L}(u)$ and a right portion denoted $P_R(u)$.  Let $C_{L}(u)$
(resp. $C_R(u)$) be the closed walk formed by the concatenation of
$P_{L}(u)$ and $P(e_1,e_2)$ (resp. of $P_{R}(u)$ and $P(e_1,e_2)$).
Let $\ell_{L}(u)$ be the length of $P_{L}(u)$, let $\ell(u)$ be the
length of $P(e_1,e_2)$, and let $\ell_R(u)$ be the length of $P_R(u)$;
note that $\ell_{L}(u)=o_{L}(u)+\iota_L(u)+1$ and
$\ell_R(u)=o_R(u)+\iota_R(u)+1$.  Note also that
$\iota_{L}(u)+w_{L}(u)$ is the total weight of half-edges inside
$C_{L}(u)$ and incident to a vertex on $C_{L}(u)$. By a simple 
counting argument based on the Euler relation, 
  this number is equal to
$\frac1{2}(\ell_{L}(u)+\ell(u))-b$. This gives the equation
$$
o_{L}(u)-\iota_{L}(u)-2w_{L}(u)=2b-1-\ell(u).
$$
Similarly we obtain
$$
o_{R}(u)-\iota_{R}(u)-2w_{R}(u)=2b-1-\ell(u).
$$
Summing over $u\in S$ we find 
$$
\hgamma_{L}^{X}(C)-2\gamma_{L}^{X'}(C')=\hgamma_{R}^{X}(C)-2\gamma_{R}^{X'}(C'),
$$
hence $\hgamma^X(C)=2\gamma^{X'}(C')=0$. Hence $X$ is $\hgamma$-balanced.  
\end{proof}

\begin{lemma}\label{lem:criterion_balanced_b_reg}
  Consider $M\in\cE_{2b}$ such that $\aM$ admits a $b$-regular
  orientation $X$.
 If  the $\hgamma$-score of two
  non-contractible non-homotopic cycles of
  $\aM$ is $0$, then $X$ is $\hgamma$-balanced.
\end{lemma}
\begin{proof}
  Let $C_1,C_2$ be two non-homotopic
  non-contractible cycles of $\aM$, each given with a traversal
  direction, such that
  $\hgamma(C_1)=\hgamma(C_2)=0$.

By Lemma~\ref{lem:exists_bregular}, $M$ has essential girth $2b$. 
We now show that $M$ has to be bipartite. 
For $C\in\{C_1,C_2\}$, let $\nww(C)$ be the number of $M$-edges
on $C$, and 
let $\Vb(C)$ (resp. $\Vw(C)$) be the set of black (resp. white)
 vertices on $C$ and 
$\nb(C)=|\Vb(C)|,\ \nw(C)=|\Vw(C)|$.  
For each $u\in\Vb(C)$, 
let $c_L(u)$ (resp. $c_R(u)$) be the number of corners of $\aM$ 
incident to $u$ on the left (resp. right) of $C$, 
and let $\kappa_L(C)=\sum_{u\in \Vb(C)}c_L(u)$,
and $\kappa_R(C)=\sum_{u\in \Vb(C)}c_R(u)$, and 
$\kappa(C)=\kappa_L(C)+\kappa_R(C)$; note that $\kappa(C)$
is the total degree of faces corresponding to the black 
vertices on $C$, hence $\kappa(C)$ is an even integer.   
The \emph{left length} of $C$ is defined as 
$$
\ell_L(C)=\nww(C)+\kappa_L(C).
$$
It corresponds to the length of the closed walk of edges of $M$
that coincides with $C$ at $M$-edges, and takes the 
left boundary of the corresponding face of $M$ each time $C$
passes by a black vertex. Since all face-degrees of $M$ 
are even and $C_1,C_2$ are non-contractible non-homotopic cycles,
it is enough to show that $\ell_L(C)$ is even 
for $C\in\{C_1,C_2\}$ to prove that $M$ is bipartite.
 Recall that $w_L(C)$ (resp. $w_R(C)$) 
denotes the total weight of half-edges incident to white 
vertices of $C$ on the left (resp. right) side of $C$,
and $\iota_L(C)$ (resp. $\iota_R(C)$) denotes the total 
number of ingoing edges at black vertices on $C$ on the left
(resp. right) side of $C$.  
 It is easy to see that the property 
 $\hgamma(C)=0$ rewrites as $\eta_L(C)=\eta_R(T)$, where 
$$
\eta_L(C)=2w_L(C)-2\iota_L(C)+\kappa_L(C),\ \ \ \ \eta_R(C)=2w_R(C)-2\iota_R(C)+\kappa_R(C).
$$
Let $e(C)$ be the number of edges on $C$ (which is also 
the length of $C$). Let $\Sigma(C)$ be the total weight of half-edges incident to vertices
in $\Vw(C)$ minus the total ingoing degree of vertices in $\Vb(C)$.  Then  it is easy to see
that 
$$
\Sigma(C)=w_L(C)-\iota_L(C)+w_R(C)-\iota_R(C)+(b-1)\cdot \nww(C). 
$$
Moreover, since $X$ is $b$-regular we have
$$
\Sigma(C)=b\cdot\nw(C)+b\cdot\nb(C)-\frac1{2}\kappa(C)=b\cdot e(C)-\frac1{2}\kappa(C)=b(\nww(C)+2\nb(C))-\frac1{2}\kappa(C).
$$
The equality between the two expressions of $\Sigma(C)$ yields $\eta_L(C)+\eta_R(C)=2\nww(C)+4b\nb(C)$,
which gives $\eta_L(C)=\nww(C)+2b\nb(C)$. Since $\eta_L(C)=2w_L(C)-2\iota_L(C)+\kappa_L(C)$,
we conclude that 
\[\ell_L(C)=\nww(C)+\kappa_L(C)=(\eta_L(C)-2b\nb(C))+(\eta_L(C)-2w_L(C)+2\iota_L(C)),\]
so that $\ell_L(C)$ is even. 
This concludes the proof that $M$ is bipartite.

We now prove that $X$ is $\hgamma$-balanced.
By Lemma~\ref{lem:exists_bregular}, $\aM$
has 
 a $\hgamma$-balanced $b$-regular orientation $X'$.
Then, $C_1$ and $C_2$ have the same
$\hgamma$-score (which is zero) 
in $X$ as in $X'$, hence, by Lemma~\ref{lem:gammahgamma}, they have 
the same $\gamma$-score in $X$ as in $X'$.  
By Corollary~\ref{theo:general_gamma}, $X$ and $X'$ are
$\gamma$-equivalent. Thus, again by Lemma~\ref{lem:gammahgamma}, $X$ is $\hgamma$-balanced.
\end{proof}

\begin{lemma}\label{lem:bregCanonical}
  Let $M\in\hat{\cM}_{2b}$.
Then $\aM$ has a unique minimal
$\hgamma$-balanced $b$-regular orientation. This orientation is 
 transferable.
Moreover, it 
is in $\cO_{2b}^1$ if and only if $M\in\chL_{2b}$ (the root-face contour is a maximal $2b$-angle). 
\end{lemma}
\begin{proof}
  By Lemma~\ref{lem:exists_bregular}, $\aM$ admits a
  $\hgamma$-balanced $b$-regular orientation $X$. By Lemma~\ref{lem:gammahgamma}, a $b$-regular orientation
 is $\hgamma$-balanced if and only if it is $\gamma$-equivalent to $X$. Hence, by 
  Corollary~\ref{theo:general_gamma}, $M$ admits a unique $b$-regular
  orientation $X_0$ that is minimal and $\hgamma$-balanced.

  The argument to ensure that $X_0$ is transferable is the same as
  given in the planar case~\cite[Lemma~50]{BF12b}. Suppose by
  contradiction that there is a star-edge $\epsilon=\{b,w\}$ going
  toward its black extremity $b$, and such that the $M$-edge
  $e=\{w,w'\}$ just after $\epsilon$ in \cw order around $w$ has
  weight different from $b-1$. Thus $e$ has strictly positive weight at $w'$. Then
  let $\epsilon'$ be the star-edge just after $\epsilon$ in \ccw order
  around $b$. Note that $\epsilon'$ has to be directed toward $b$,
  otherwise $(\epsilon',\epsilon,e)$ would form a face $S$ distinct
  from the root-face, such that every edge on the boundary of $S$ has
  a face in $S$ on its right, contradicting the minimality of $X_0$.
  Let $e'=(w',w'')$ be the M-edge just after $\epsilon'$ in \cw order
  around $w'$. Since the edges $e$ and $\epsilon'$ contribute by at
  least $2$ to the weight of $w'$, the edge $e'$ can not have weight
  $b-1$ at $w'$, hence it has positive weight at $w''$. Continuing
  iteratively in \ccw order around $b$ we obtain that $b$ has only
  ingoing edges, a contradiction. So $X_0$ is transferable.

  Let us now characterize when $X_0$ is in $\cO_{2b}^1$.

  Suppose that the root-face contour is not a maximal $2b$-angle, let
  $C$ be a maximal $2b$-angle whose interior contains the root-face. By
  a counting argument similar to the proof of Claim~\ref{claim:epsilon}, 
  all half-edges incident to a vertex on $C$ and in the interior of
  $C$ have weight $0$, hence a rightmost walk starting from an edge on
  $C$ can never loop on the root-face contour. Hence $X_0$ is not in
  $\cO_{2b}^1$.

  Conversely assume that $M$ is in $\chL_{2b}$. Let $e$ be an outgoing
  half-edge of $X_0$ and let $P_e$ be the rightmost path starting at
  $e$.  Since $X_0$ is transferable, it is easy to see that once $P_e$
  has reached an $M$-vertex (which occurs after traversing at most two
  edges), it will only take $M$-edges. Hence the cycle $C$ formed when
  $P_e$ eventually loops is a right cycle of $M$-edges.  By the same
  line of arguments as in Section~\ref{sec:right}, this cycle has to
  be of length $2b$, with a contractible region on its right.  This
  region has to contain the root-face since $X_0$ is minimal.  Since
  the root-face contour is a maximal $2b$-angle, we conclude that $C$
  is actually the root-face contour. Hence $X_0$ is in $\cO_{2b}^1$.
\end{proof}

\begin{lemma}
  \label{lem:hbaltbal}
Let $M$ be a map in $\cE_{2b}$. Let $Y$ be a $\bb$-$\zZ$-orientation
of $M$ 
in $\cO_{2b}^1$, let $X=\sigma^{-1}(Y)$ 
be the associated 
$b$-regular orientation of $\aM$, and let $T$
be the associated mobile in $\hat{\mathcal
    V}_b$.   
Then $X$ is $\hgamma$-balanced if and only if $T$ is balanced.  
\end{lemma} 
\begin{proof}
Recall that the rules to obtain the mobile associated to $Y$ are the ones of
  Figure~\ref{fig:local_rule_weighted_biori}.

  Let $C$ be a non-contractible cycle of $T$ given with a traversal
  direction; note that $C$ is also a non-contractible cycle of $\aM$
  (it is convenient here to see $T$ and $\aM$ as superimposed). Let
  $\nb(C)$ be the number of black vertices on $C$, and let $\nbw(C)$
  (resp. $\nwb(C)$) be the number of black-white edges $e$ on $C$
  where the black extremity is traversed before (resp. after) the
  white extremity when traversing $e$ (along the traversal direction
  of $C$).  Note that all the black-white edges on $C$ have weights
  $(0,b-1)$ (the weights can not be $(-1,b)$ since the white extremity
  is a leaf in that case). Note also that $\nbw(C)=\nb(C)=\nwb(C)$,
  since every black vertex is preceded and followed by white vertices
  along $C$. Let $w_{L}^T(C)$ (resp. $w_R^T(C)$) be the total weight
  of half-edges in $T$ that are incident to a vertex (white or black)
  of $C$ on the left side (resp. right side) of $C$. Let $s_{L}(C)$
  (resp. $s_r(C)$) be the total number of half-edges, including the
  buds, that are incident to a black vertex of $C$ on the left
  (resp. right) side of $C$ (note that this quantity is the same for
  $T$ as for $X$). Let $w_L^X(C)$ (resp. $w_R^X(C)$) be the total weight in
  $X$ of half-edges at $M$-vertices of $C$, on the left (resp. right)
  side of $C$. Let $\iota_L^X(C)$ (resp. $\iota_R^X(C)$) be the total
  number of ingoing edges at black vertices on the left (resp. right)
  side of $C$.  We have $\gamma_L^T(C)=2w_{L}^T(C)+s_L(C)$,
  $\gamma_R^T(C)=2w_{R}^T(C)+s_R(C)$, and
  $\gamma^T(C)=\gamma_R^T(C)-\gamma_L^T(C)$. Moreover, we have
  $\hgamma_L^X(C)=2w_L^X(C)+s_L(C)-2\iota_L^X(C)$,
  $\hgamma_R^X(C)=2w_R^X(C)+s_R(C)-2\iota_R^X(C)$, and
  $\hgamma^X(C)=\hgamma_R^X(C)-\hgamma_L^X(C)$. 

The quantity $w_L^{T}(C)$ decomposes as $w_L^{\circ,T}(C)+w_L^{\bullet,T}(C)$, where the first (resp. second)
term gathers the contribution from the half-edges at white (resp. black) vertices. 
Clearly $w_L^{\bullet,T}(C)=-\iota_L^X(C)$. 
We let $\cHci_L(C)$ be the sef of half-edges of $\aM$ that are on the left of $C$ and incident to a white vertex on $C$. 
A half-edge $h$ in $\cHci_L(C)$, with $v$ its incident vertex,
 is called \emph{$C$-adjacent} if the next half-edge in $\aM$ in ccw order around $v$ is on $C$; it is called
 \emph{$C$-internal} otherwise. 
Then the combined effect of the transfer rule of Figure~\ref{fig:rule_sigma} and of the local rules of Figure~\ref{fig:local_rule_weighted_biori} ensure that $w_L^{\circ,T}(C)$ gives the 
total contribution to $w_L^X(C)$ by $C$-internal half-edges in $\cHci_L(C)$. Let $A_L(C)$ be the total contribution to $w_L^X(C)$ by $C$-adjacent half-edges in $\cHci_L(C)$. 
Then, very similarly
as in the proof of Lemma~\ref{lem:phispebal} (see Figure~\ref{fig:Cadj}), each black vertex on $C$ yields a contribution $b-1$ to $A_L(C)$, 
so that $A_L(C)=(b-1)\nb(C)$.  
We conclude that $ w_L^X(C)-\iota_L^X(C)=w_L^T(C)+(b-1)\nbw(C)$. Very similarly we have 
$w_R^X(C)=w_R^T(C)+\iota_R(C)+(b-1)\nwb(C)$. Hence $\gamma^T(C)=\hgamma^X(C)$. 
 
This implies that if $X$ is $\hgamma$-balanced then $T$ is
 balanced. 
 Now, suppose that $T$ is balanced. Then $\gamma^T(C)=0$ for any
 non-contractible cycle $C$ of $T$. Let $\{C_1,C_2\}$ be two such
 distinct cycles. They are not homotopic since $T$ is unicellular. By what
precedes we
 have $\hgamma^X(C_1)=0$ and $\hgamma^X(C_2)=0$. Hence $X$ is $\hgamma$-balanced by 
 Lemma~\ref{lem:criterion_balanced_b_reg}.
\end{proof}

We are now able to prove Proposition~\ref{prop:unique_b_at_least_2}.

\noindent\emph{Proof of Proposition~\ref{prop:unique_b_at_least_2}.}
Suppose that $M\in\cE_{2b}$ admits a $\bb$-$\zZ$-orientation
$Y\in\cO_{2b}^1$ whose associated mobile by $\Phi_+$ is in
$\hat{\mathcal V}_b^{Bal}$, and let $X=\sigma^{-1}(Y)$.  Then $X$ is
$\hgamma$-balanced (according to Lemma~\ref{lem:hbaltbal}), is
transferable and in $\cO_{2b}^1$ (according to
Lemma~\ref{lem:bij_transferable}), and is minimal (according to
Lemma~\ref{lem:necmin}).  Lemma~\ref{lem:exists_bregular} implies that 
$M\in\hat{\cM}_{2b}$.
Hence, according to
Lemma~\ref{lem:bregCanonical}, $M$ is in $\chL_{2b}$ and moreover $Y$ is unique (it has to
be the image by $\sigma$ of the unique minimal $\hgamma$-balanced $b$-regular
 orientation of
$M^\star$).

Conversely let us prove the existence part, for $M\in\chL_{2b}$.  By
Lemma~\ref{lem:bregCanonical}, let $X$ be the minimal
$\hgamma$-balanced $b$-regular orientation of $M^\star$ that is
moreover transferable and in $\cO_{2b}^1$. Let $Y=\sigma(X)$ so that
$Y\in\cO_{2b}^1$ by Lemma~\ref{lem:bij_transferable}.  And
Lemma~\ref{lem:hbaltbal} ensures that the mobile associated to $Y$ by
$\Phi_+$ is in $\hat{\mathcal V}_b^{Bal}$.$\hfill\square$

\subsubsection{Proof of Theorem~\ref{theo:bij_maps_d} for $d\geq 2$}
\label{sec:proof_d_g_2}

We start by giving some terminology and results for $d\geq 1$, before
continuing with $d\geq 2$ in the rest of the section.

Let $d\geq 1$.  Let $\cH_d$ be the family of face-rooted toroidal maps
with root-face degree $d$ and with all faces of degree at least $d$.
For $M\in \cH_d$, we define a $\frac{d}{d-2}$-$\ZZ$-orientation of $M$
as a $\ZZ$-biorientation with weights in $\{-2,\ldots,d\}$ such that
all vertices have weight $d$, all edges have weight $d-2$, and every
face $f$ has weight $-\mathrm{deg}(f)+d$.

The bijection $\Phi_+$ specializes into a bijection between maps in
$\cH_d$ endowed
with a $\frac{d}{d-2}$-$\zZ$-orientation in $\cO_d^1$ and
 the family $\mathcal V_d$ of toroidal
 $\frac{d}{d-2}$-$\zZ$-mobiles.
Showing Theorem~\ref{theo:bij_maps_2b} for $d\geq 2$ thus amounts to proving
the following statement:

\begin{proposition}\label{prop:unique_d_at_least_2}
Let $d\geq 2$ and let $M$ be a map in $\cH_{d}$. Then 
$M$ admits a $\frac{d}{d-2}$-$\zZ$-orientation 
in $\cO_{d}^1$ whose associated mobile by $\Phi_+$  is in ${\mathcal V}_d^{Bal}$ if and only if $M$ is in
$\cL_{d} $.
In that case $M$ admits a unique such orientation. 
\end{proposition}

Let $d\geq 1$ and let $M\in\cH_{d}$. We denote by $M_2$ the
(bipartite) map obtained from $M$ by inserting a new vertex on each
edge. Note that
$M\in\cL_{d}$ if and only if $M_2\in\chL_{2d}$. Applying (as done in
the planar case in~\cite[Lem.~55]{BF12b}) the rules of Figure~\ref{fig:rule_M2}
to each edge of a $\frac{d}{d-2}$-$\ZZ$-orientation $Z$ of $M$, we
obtain a $\frac{d}{d-1}$-$\ZZ$-orientation $Y=\iota(Z)$ of $M_2$.  The mapping $\iota$ is clearly
bijective. Moreover, $Z$ is in $\cO_{d}^{1}$ if and only if $\iota(Z)$ is in
$\cO_{2d}^{1}$.

\begin{figure}[!h]
\begin{center}
\includegraphics[width=12cm]{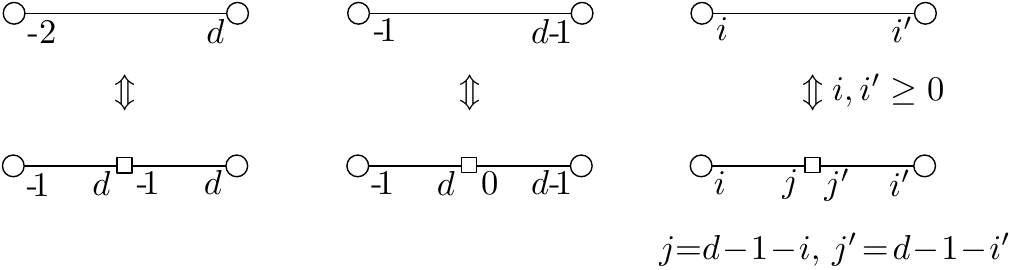}
\end{center}
\caption{The local rule in the 1-to-1 correspondence $\iota$ between
  the $\frac{d}{d-2}$-$\ZZ$-orientations of $M$ and the
  $\frac{d}{d-1}$-$\ZZ$-orientation of $M_2$ (the half-edge direction is not
  indicated, it is determined by the status of
  the weight it carries, either in $\ZZ_{>0}$ or in $\ZZ_{\leq 0}$).}
\label{fig:rule_M2}
\end{figure}

From now on we assume that $d\geq 2$. We first prove the analogue of
Lemma~\ref{lem:hbaltbal}:

\begin{lemma}
  \label{lem:hbaltbald}
Let $d\geq 2$ and $M$ be a map in $\cH_{d}$. Let $Z$ be a $\frac{d}{d-2}$-$\zZ$-orientation
of $M$ 
in $\cO_{d}^1$, let $X=\sigma^{-1}(\iota(Z))$ 
be the associated $d$-regular orientation of $\aM_2$, and let $T$
be the associated mobile (by $\Phi_+$) in ${\mathcal
    V}_d$.   
Then $X$ is $\hgamma$-balanced if and only if $T$ is balanced.  
\end{lemma}

\begin{proof}
Let $C$ be a non-contractible cycle of $T$. 
 Let $\nb(C)$ be the number of black vertices on $C$.
Let $\nbw(C)$ (resp. $\nwb(C)$) be the number of black-white edges $e$
on $C$ where the black extremity is traversed before (resp. after) the
white extremity when traversing $e$ (along the traversal direction of
$C$), and let $\nbb(C)$ be the number of black-black edges along $C$
(note that $\nbb(C)=0$ for $d>2$). As in the last section it is easy
to see that $\nbw(C)=\nwb(C)$.  Note that black-white edges on $C$ can
have weights $(0,d-2)$ or $(-1,d-1)$ (but not $(-2,d)$ since the white
extremity would be a leaf).

Let $w_{L}^T(C)$ (resp. $w_R^T(C)$) be the total weight of half-edges
 in $T$ incident to a vertex, white or black, of $C$ on the left side
(resp. right side) of $C$.  Let $s_{L}^T(C)$ (resp. $s_R^T(C)$) be the
total number of half-edges, including the buds, that are incident to a
black vertex of $C$ on the left (resp. right) side of $C$. 
Note that $C$ identifies to a cycle of $\aM_2$, which we also call $C$ by a slight abuse of notation
(the only difference to keep in mind is that, for each black-black or white-white edge $e$ on $C$ seen as 
a cycle of $T$, in $\aM_2$ there is a white square vertex in the middle of $e$).  
 Let
$w_L^X(C)$ (resp. $w_R^X(C)$) be the total weight (in $X$) of half-edges at
white vertices (round or square) of $C$ on the left (resp. right) side
of $C$.  Let $\iota_L^X(C)$ (resp. $\iota_R^X(C)$) be the total number
of ingoing edges (in $X$) at black vertices on the left (resp. right)
side of $C$. Let $s_L^X(C)$ (resp. $s_R^X(C)$) be the total number of
edges incident to a black vertex on the left (resp. right) side of
$C$.  We have $\gamma_L^T(C)=w_{L}^T(C)+s_L^T(C)$,
$\gamma_R^T(C)=w_{R}^T(C)+s_R^T(C)$, and
$\gamma^T(C)=\gamma_R^T(C)-\gamma_L^T(C)$.  And we have 
$\hgamma_L^X(C)=2w_L^X(C)+s_L^X(C)-2\iota_L^X(C)$,
$\hgamma_R^X(C)=2w_R^X(C)+s_R^X(C)-2\iota_R^X(C)$, and
$\hgamma^X(C)=\hgamma_R^X(C)-\hgamma_L^X(C)$.

The quantity $w_L^T(C)$ decomposes as $w_{L}^{\circ,T}(C)+w_L^{\bullet,T}(C)$ where the first (resp. second) term gathers 
the contributions from the white (resp. black) vertices.  
We let $\cH_L(C)$ be the set of half-edges of $\aM_2$ that are on the left of $C$ and incident to a vertex on $C$ (including
white square vertices on $C$, i.e., seeing $C$ as a cycle in $\aM_2$). 
 The set $\cH_L(C)$ partitions as $\cH_L(C)=\cHci_L(C)\cup\cHbs_L(C)\cup\cHbu_L(C)$ whether the incident vertex is white round, 
white square, or black. A half-edge $h$ of $\cHci_L(C)$ (resp. $\cHbu_L(C)$) is called \emph{$C$-adjacent} if the next half-edge of $\aM_2$ after $h$ in ccw order (resp. cw order) around the vertex incident to $h$ is on $C$; it is called \emph{$C$-internal} otherwise.  By the combined effect of the transfer rule of Figure~\ref{fig:rule_sigma}, the rules of Figure~\ref{fig:rule_M2}, and the local rules in Figure~\ref{fig:local_rule_weighted_biori}, the quantity 
$w_{L}^{\circ,T}(C)$ represents the total contribution to $w_L^X(C)$ of the $C$-internal half-edges in $\cHci_L(C)$,
while $w_L^{\bullet,T}(C)$ represents the total contribution to $-\iota_L^X(C)$ 
of the $C$-internal half-edges in $\cHbu_L(C)$. 

\begin{figure}
\begin{center}
\includegraphics[width=12cm]{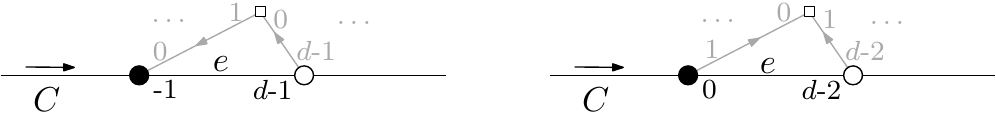}
\end{center}
\caption{The situation at an edge $e$ counted by $\nbw(C)$ (the weights and orientations of the two $C$-adjacent
edges, shown in gray, are determined by the combined effect of the transfer rule of Figure~\ref{fig:rule_sigma}, the rules of Figure~\ref{fig:rule_M2}, and the local rules in Figure~\ref{fig:local_rule_weighted_biori}).   
If $e$ has weights $(-1,d-1)$ it has contribution $d-1$ 
to $w_L^X(C)$ and contribution $1$ to $\iota_L^X(C)$.
If $e$ has weights $(0,d-2)$ it has contribution $d-2$ 
to $w_L^X(C)$ and contribution $0$ to $\iota_L^X(C)$.
Hence it always has a contribution $d-2$ to $w_L^X(C)-\iota_L(X)$.}
\label{fig:contribution_black_white}
\end{figure}

We let $A_L(C)$ be the total contribution to $w_L^X(C)-\iota_L^X(C)$ of $C$-adjacent half-edges
 from $\cHci(C)\cup\cHbu(C)$.  
An important observation (see Figure~\ref{fig:contribution_black_white}) is that  an edge of $T$ counted
by $\nbw(C)$ always gives a contribution $d-2$ to $A_L(C)$ 
(whether it has weights $(-1,d-1)$ or $(0,d-2)$), hence $A_L(C)=(d-2)\nbw(C)$. 
Finally, the total contribution to $w_L^X(C)-\iota_L^X(C)$ by half-edges in $\cHbs(C)$ is $\nbb(C)$. Indeed the only 
contribution is a contribution by one to $w_L^X(C)$ for each black-black edge $e$ on $C$
(there is a white square vertex in the middle of $e$, with an outgoing edge on each side, recalling that black-black edges in $T$ occur only for $d=2$). 
 We thus have
$$
w_L^X(C)-\iota_L^X(C)=w_L^T(C)+(d-2)\nbw(C)+\nbb(C), 
$$
and similarly we have 
$$
w_R^X(C)-\iota_R^X(C)=w_R^T(C)+(d-2)\nwb(C)+\nbb(C). 
$$
Moreover we have
$$
s_L^X(C)=2s_L^T(C)+\nb(C)-2\nbb(C),\ \ \ \ s_R^X(C)=2s_R^T(C)+\nb(C)-2\nbb(C).
$$
With these equalities, and using the fact that $\nbw(C)=\nwb(C)$,
we easily deduce $2\gamma^T(C)=\hgamma^X(C)$, and in particular 
$\gamma^T(C)=0$ if and only if $\hgamma^X(C)=0$. 

From there, very similarly as in the end of the proof of
Lemma~\ref{lem:hbaltbal}, we conclude that $X$ is balanced if and
only if $T$ is balanced, which concludes the proof.
\end{proof}

We are now able to prove Proposition~\ref{prop:unique_d_at_least_2}:

\noindent\emph{Proof of Proposition~\ref{prop:unique_d_at_least_2}.}
Suppose that $M\in\cH_{d}$ admits a $\frac{d}{d-2}$-$\zZ$-orientation  $Z\in\cO_{2b}^1$ whose
associated mobile by $\Phi_+$ is in $\mathcal{V}_d^{Bal}$, and let $X=\sigma^{-1}(\iota(Z))$.  
Then $X$ is $\hgamma$-balanced (according to Lemma~\ref{lem:hbaltbald}), 
is transferable and in $\cO_{2d}^1$
(by Lemma~\ref{lem:bij_transferable} and since $\iota$ preserve the property of being in $\cO_{2d}^1$), 
and  minimal (according to Lemma~\ref{lem:necmin}).
 Lemma~\ref{lem:exists_bregular} implies that 
$M_2\in\hat{\cM}_{2d}$.
Hence, according to Lemma~\ref{lem:bregCanonical}, $M_2$ is in $\chL_{2d}$, so that 
$M$ is in $\cL_d$, and moreover
$Z$ is unique (it has to be the image by $\iota^{-1}\circ \sigma$ of
the unique minimal  $\hgamma$-balanced $d$-regular orientation of $M_2^\star$). 

Conversely let us prove the existence part, for $M\in\cL_{d}$. By Lemma~\ref{lem:bregCanonical},
let $X$ be the minimal $\hgamma$-balanced  $d$-regular  orientation 
of $M_2^\star$, that is
moreover transferable and in $\cO_{2d}^1$. Let
$Z=\iota^{-1}(\sigma(X))$ so that $Z\in\cO_{2d}^1$ by
Lemma~\ref{lem:bij_transferable} and since $\iota$
preserve  the property of being in $\cO_{2d}^1$. 
 And Lemma~\ref{lem:hbaltbald} ensures that the  mobile associated
to $Z$ by $\Phi_+$ is in $\mathcal V_d^{Bal}$. $\hfill\square$

\subsubsection{Proof of Theorem~\ref{theo:bij_maps_2b} for $b=1$}
Before proving Theorem~\ref{theo:bij_maps_2b} for $b=1$ let us make a simple observation. 
We have proved Theorem~\ref{theo:bij_maps_2b} for $b\geq 2$ and Theorem~\ref{theo:bij_maps_d} for $d\geq 2$. 
For $b\geq 1$ a $\ZZ$-bimobile in $\cV_{2b}^{Bal}$ is called \emph{even} 
if all its half-edge weights are even. The mapping consisting in doubling the half-edge
weights gives a bijection between  
$\hat{\cV}_b^{Bal}$ and even $\ZZ$-bimobiles in $\cV_{2b}^{Bal}$. Moreover
the toroidal map 
(obtained by performing $\Psi_+$) associated to a bimobile in $\hat{\cV}_b^{Bal}$
is the same as the toroidal map associated to the weight-doubled bimobile. 
 Hence, if we call $\phi_d$, for $d\geq 2$, the bijection in Theorem~\ref{theo:bij_maps_d}
and $\hat{\phi}_b$, for $b\geq 2$, the bijection in
Theorem~\ref{theo:bij_maps_2b} then we have already obtained:

\begin{center}
\begin{minipage}{13.2cm}
`For $b\geq 2$ and $M\in\cL_{2b}$, we have that $\phi_{2b}(M)$ is even
if and only if $M$ is bipartite, 
 and in that case $\phi_{2b}(M)$ is equal to $\hat{\phi}_b(M)$ upon doubling the half-edge
weights'.  
\end{minipage}
\end{center}

Note that if we can establish (as stated next) 
the similar bipartiteness condition for $b=1$ then we 
will have Theorem~\ref{theo:bij_maps_2b} for $b=1$
 (as the bipartite specialization of Theorem~\ref{theo:bij_maps_d} for $d=2$). 

\begin{lemma}\label{lem:bip_2}
Let $M\in\cL_2$ and let $T=\phi_2(M)$. Then $T$ is even if and only if $M$ is bipartite.
\end{lemma}

\begin{proof}
Assume $T$ is even, and let $T'$ be obtained from $T$ after dividing
by $2$ the half-edge weights.  Note that $T'\in\hat\cV_1^{Bal}$ and in 
particular the degrees of all black vertices of $T'$ are even, 
so that all face-degrees of $M$ are even. 
Since the weight of a white vertex of $T'$ is $1$, in  $T'$ 
all white vertices are leaves. 
  Consider two distinct cycles $C_1,C_2$ of $T'$ and
  $C\in \{C_1,C_2\}$. Since white vertices are leaves, the cycle $C$ is
  made only of black vertices and black-black edges, with zero weights
on both half-edges.  Let $k$ be the length of $C$.
  Let $w_L(C)$ (resp. $w_R(C)$) be the total weight of half-edges of
  $T$ incident to (black) vertices of $C$ on the left (resp. right)
  side of $C$. Let $s_L(C)$ (resp.  $s_R(C)$) be the total number of
  half-edges (including buds) incident to black vertices of $C$ on the
  left (resp. right) side of $C$.  
Since $T'$ is balanced we have 
  $2\,w_L(C)+s_L(C)=2\,w_R(C)+s_R(C)$. Let $\kappa(C)$ be the total
  degree of faces corresponding to vertices of $C$, so
  $\kappa(C)=s_L(C)+s_R(C)+2k$.  Since all the half-edges on $C$ have weight $0$, the
  total weight of vertices of $C$ is
  $w_L(C)+w_R(C)=\sum_{u\in C} (-\frac{1}{2}
  \mathrm{deg}(u)+1)=-\frac{1}{2}\kappa(C)+k$. By combining the three
  equalities, we obtain that $s_L(C)=-2 w_L(C)$. So $s_L(C)$ is even.
  So $s_L(C_i)$ is even for $i\in\{1,2\}$.
For $i\in\{1,2\}$, let $W_i$ be the walk of $M$ that is ``just on the
left'' of $C_i$ (seeing $M$ and $T$ as superimposed), 
i.e. obtained by following the left boundary of the
corresponding face of $M$ each time $C_i$ passes by a black vertex.
By the local rules of $\Phi_+$ shown in Figure~\ref{fig:local_rule_weighted_biori}, 
the length of $W_i$ is precisely equal to
$s_L(C_i)$ and thus is even.  All the faces of $M$ are even, and the
two walks $W_i$ are non-homotopic to a contractible cycle
and non-homotopic to each other. Thus $M$ is bipartite.

Conversely, assume that $M$ is bipartite. Note that there are 3 types 
of edges in $T\in\cV_2^{Bal}$:
those of weights $(-2,2)$ that connect a black vertex to a white leaf, those
of weights $(-1,1)$ that connect a black vertex to a white vertex of degree $2$ (incident
to two such edges), and those of weights $(0,0)$ that connect two black vertices. 
We call \emph{odd} the edges of weights $(-1,1)$. To prove that $T$ is even 
we thus have to show that $T$ has no odd edges.  
Let $\Gamma$ be the subgraph of $T$ induced by the odd edges. 
Since $M$ is bipartite, all its 
faces have even degree and thus all black vertices of $T$ have even weight
(since for $d=2$ the weight of a black vertex of $T$ is $2$ minus the degree of the 
associated face). Moreover every white vertex is incident to either no odd edge or 
to two odd edges. Hence $\Gamma$ is an Eulerian subgraph of $T$.   There are two
types of toroidal unicellular maps since two cycles of a toroidal
unicellular map may intersect either on a single vertex (square case)
or on a path (hexagonal case), as depicted on
Figure~\ref{fig:hexasquare}.  If $T$ is hexagonal, then $\Gamma$ is
exactly one of the cycles of $T$.
If $T$ is square, then $\Gamma$ can be either one of the cycles of $T$
or the union of the two cycles of $T$. One easily checks that in all cases, there exists a cycle
$C$ of $T$ that has exactly one incident edge in $\Gamma$ on each side. We endow $C$
with a traversal direction. 

\begin{figure}[!h]
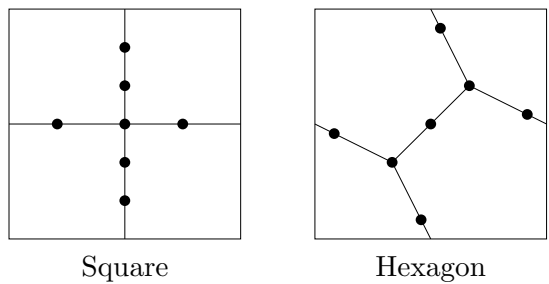

\center
\begin{tabular}{cc}
\includegraphics[scale=0.4]{cases-2} \ \ & \ \
\includegraphics[scale=0.4]{cases-3}\\
Square \ \ & \ \ Hexagon
\end{tabular}

\caption{The two types of toroidal unicellular maps.} 
\label{fig:hexasquare}
\end{figure}
Recall from Section~\ref{sec:bij_extended}, that $\gamma_L(C)=w_L(C)+s_L(C)$,
$\gamma_R(C)=w_R(C)+s_R(C)$. Moreover since
$T \in {\mathcal V}_2^{Bal}$, we have $\gamma_L(C)=\gamma_R(C)$.
Note that white vertices of $C$ have all their weight on $C$.  Let $\nb(C)$ be
the number of black vertices on $C$.  Let $\kappa(C)$ be the total
degree of faces corresponding to the black vertices on $C$.  So black
vertices of $C$ have total weight $-\kappa(C)+2\nb(C)$.  Let $\nbw(C)$
(resp. $\nwb(C)$)
be the number of black-white (resp. white-black) edges on $C$ while following the traversal
direction of $C$. Clearly $\nbw(C)=\nwb(C)$.
The total weight of half-edges on $C$ incident to a black vertex
 is precisely $-\nbw(C)-\nwb(C)=-2\nbw(C)$. So
$w_L(C)+w_R(C)=-\kappa(C)+2\nb(C)+2\nbw(C)$.
Note that we have $\kappa(C)=s_L(C)+s_R(C)+2\nb(C)$.
Combining the equalities gives $w_L(C)=-s_L(C)+\nbw(C)$.

Let $W$ be the walk of $M$ that is ``just on the left'' of $C$ (seeing
$M$ and $T$ as superimposed), 
i.e. obtained by following the left boundary of the
corresponding face of $M$ each time $C$ passes by a black vertex.
Since
$M$ is bipartite,  the length of $W$ is even, and according to the local rules of $\Phi_+$
shown in Figure~\ref{fig:local_rule_weighted_biori}, 
it is equal  to $s_L(C)+\nbw(C)$. So $w_L(C)=(s_L(C)+\nbw(C)) -
2s_L(C)$ is even. So $C$ is incident to an even number of edges of
$\Gamma$ on its left side, a contradiction.
   \end{proof}

\subsubsection{Proof of Theorem~\ref{theo:bij_maps_d} for $d=1$}

Recall from Section~\ref{sec:proof_d_g_2} that $\cH_1$ denotes
the family of face-rooted toroidal maps with root-face degree
  $1$ (i.e. a loop).  Moreover, for $M\in \cH_1$, a
  $\frac{1}{-1}$-$\ZZ$-orientation of $M$ is a $\ZZ$-biorientation
  with weights in $\{-2,-1,0,1\}$ such that all vertices have weight
  $1$, all edges have weight $-1$, and every face $f$ has weight
  $-\mathrm{deg}(f)+1$. Note that there are just two types of edges in
  such an orientation, with weights $(-1,0)$ or $(-2,1)$ (see the
  first row of Figure~\ref{fig:mapping_tau}).

The bijection $\Phi_+$ specializes into a bijection between maps in
$\cH_1$ endowed
with a $\frac{1}{-1}$-$\zZ$-orientation in $\cO_1^1$ and
 the family $\mathcal V_1$ of toroidal
 $\frac{1}{-1}$-$\zZ$-mobiles.
Showing Theorem~\ref{theo:bij_maps_2b} for $d=1$ thus amounts to proving
the following statement:

\begin{proposition}\label{prop:unique_d1}
Let  $M$ be a map in $\cH_{1}$. Then 
$M$ admits a $\frac{1}{-1}$-$\zZ$-orientation 
in $\cO_{1}^1$ whose associated mobile by $\Phi_+$  is in ${\mathcal V}_1^{Bal}$ if and only if $M$ is in
$\cL_{1} $.
In that case $M$ admits a unique such orientation. 
\end{proposition}

\begin{figure}[!h]
\begin{center}
\includegraphics[width=8cm]{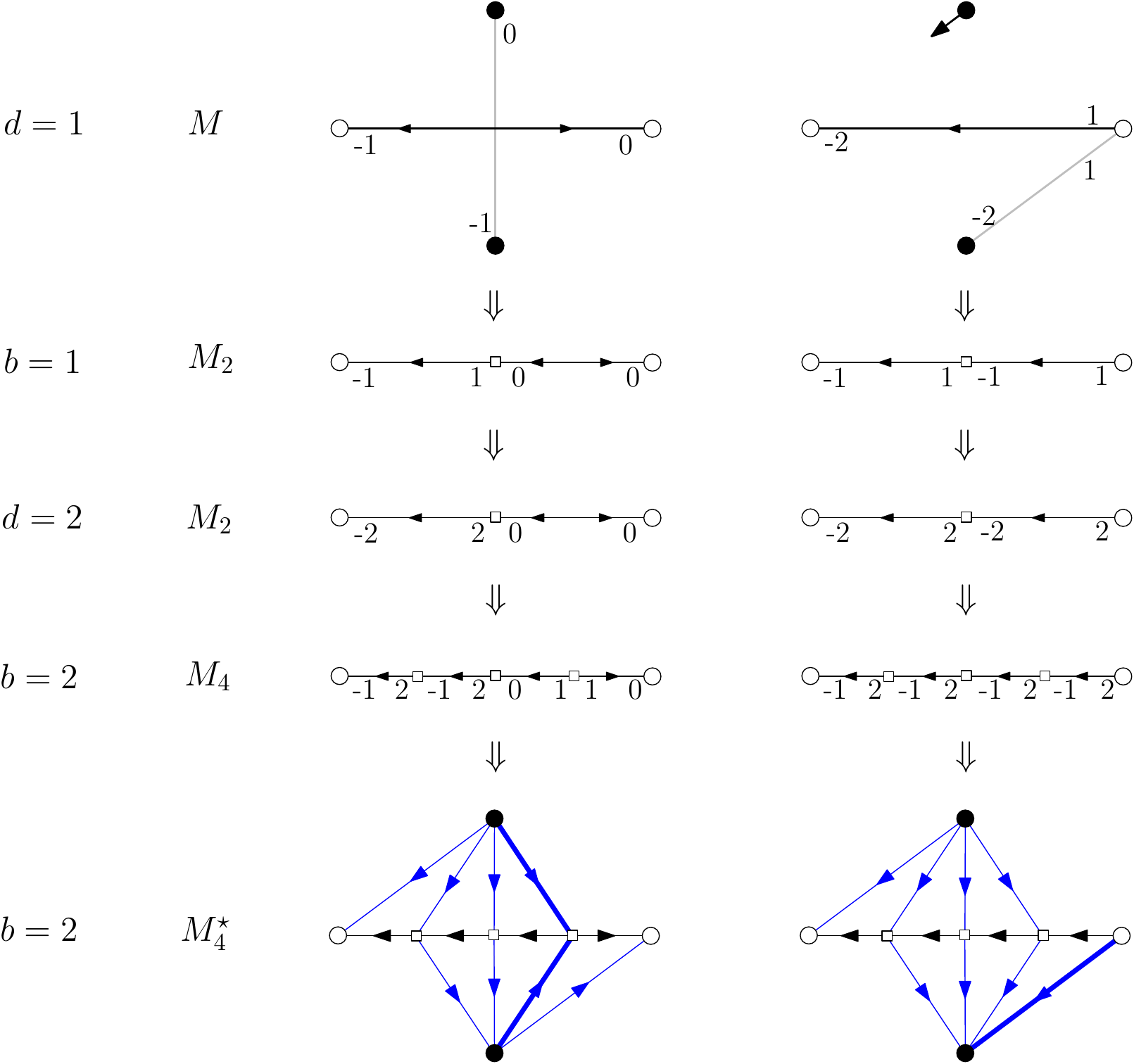}
\end{center}
\caption{The mapping $\tau$ from $\frac{1}{-1}$-orientations in 
$\cO_1^1$ 
to (certain) transferable $2$-regular orientations in $\cO_{4}^1$. 
In the top row, we show the 
corresponding mobile-edge; in the bottom-row
we show (in bolder form) on which star-edges we lift the mobile-edge.}
\label{fig:mapping_tau}
\end{figure}

For a map $M\in\cH_1$, let $M_2$ (resp. $M_4$) be the map obtained
from $M$ by subdividing every edge into a path of length $2$
(resp. $4$).  If $M$ is endowed with a $\frac{1}{-1}$-$\ZZ$-orientation $Z$
let $\tau(Z)$ be the (transferable) 2-regular orientation of
$M_4^\star$ obtained from $M$ using the rules of
Figure~\ref{fig:mapping_tau}, i.e., applying the rule of
Figure~\ref{fig:rule_M2} to obtain a $\frac{1}{0}$-$\ZZ$-orientation of
$M_2$, then doubling the weights to get to an even
$\frac{2}{0}$-$\ZZ$-orientation of $M_2$, then applying the rule of
Figure~\ref{fig:rule_M2} to get to a $\frac{2}{1}$-$\ZZ$-orientation of
$M_4$, and finally applying the mapping $\sigma^{-1}$ to get to a
transferable $2$-regular orientation of $M_4^\star$. Note that $Z$ is
in $\cO_1^1$ if and only if $\tau(Z)$ is in $\cO_4^1$.
Note that $\tau$ is injective but not a bijection since when doubling
the weights to obtain a $\frac{2}{0}$-orientation of $M_2$ we have
only even weights.

We first prove the analogue of
Lemma~\ref{lem:hbaltbald}:

\begin{lemma}
  \label{lem:hbaltbald1}
Let $M$ be a map in $\cH_{1}$, let $Z$ be a $\frac{1}{-1}$-$\zZ$-orientation
of $M$ 
in $\cO_{1}^1$, let $X=\tau(Z)$ 
be the associated 
$2$-regular orientation of $\aM_4$, and let $T$
be the associated mobile (by $\Phi_+$) in ${\mathcal
    V}_1$.   
Then $X$ is $\hgamma$-balanced if and only if $T$ is balanced.  
\end{lemma}

\begin{proof}
  Let $C$ be a non-contractible cycle of
$T$ given with a traversal direction.  We call \emph{canonical lift}
of $C$ the (non-contractible) cycle $C'$ of $M_4^\star$ obtained by
keeping the bolder edges as shown in the bottom-row of
Figure~\ref{fig:mapping_tau}.

Let $e$ be a black-black edge
on $C$, where the half-edge of weight $-1$ is traversed
before the half-edge of weight $0$. 
 Looking at the left part of Figure~\ref{fig:mapping_tau} it is 
clear that $e$ has contribution $5$ 
 to $s_L^{X}(C')$, contribution $1$ 
 to $s_R^{X}(C')$, contribution $1$ 
 to $w_L^{X}(C')$, contribution $1$ 
 to $w_R^{X}(C')$, contribution $2$ 
 to $\iota_L^{X}(C')$, and contribution $0$ 
 to $\iota_R^{X}(C')$. Hence $e$ has contribution $3$
to $\gamma_L^X(C')=2(w_L^{X}(C')-\iota_L^{X}(C'))+s_L^X(C')$,
and contribution $3$
to $\gamma_R^X(C')=2(w_R^{X}(C')-\iota_R^{X}(C'))+s_R^X(C')$,
hence has zero contribution to $\gamma^X(C')$.   
Similarly a black-black edge where the half-edge of weight $-1$
is traversed after the half-edge of weight $0$ has zero contribution
to $\gamma^X(C')$.  

Now let $e$ be a black-white edge on $C$ where
the black extremity is traversed before the white extremity.  
Then it is easy to see (again looking at Figure~\ref{fig:mapping_tau}) 
that $e$ has contribution $3$ 
 to $s_L^{X}(C')$, contribution $0$ 
 to $s_R^{X}(C')$, contribution $1$ 
 to $w_L^{X}(C')$, contribution $0$ 
 to $w_R^{X}(C')$, contribution $3$ 
 to $\iota_L^{X}(C')$, and contribution $0$ 
 to $\iota_R^{X}(C')$. Hence it has contribution $-1$ to $\gamma_L^X(C')$ and contribution $0$ to $\gamma_R^X(C')$, hence contribution 
$-1$ to $\gamma^X(C')$. 
Symmetrically a black-white edge whose black extremity is traversed
after the white extremity has contribution $1$ to $\gamma^X(C')$.
Now the numbers of black-white edges of both types on $C$
 are clearly equal, so that the total contribution of black-white
edges on $C$ to $\gamma^X(C')$ is zero. 

On the other hand, let $h$ be 
a half-edge of $T$ not on $C$ but incident to a vertex on $C$,
and let $\delta$ be the weight of $h$ (by convention $\delta=0$
if $h$ is a bud).   
Then it is easy to see (still looking at Figure~\ref{fig:mapping_tau})
that if $h$ is on the left (resp. right) side of $C$ and incident to a
black vertex, then it has contribution $4$ 
to $s_L^X(C')$ (resp. to $s_R^X(C')$) 
and contribution $2\delta$ to $-\iota_L^X(C')$. 
And if $h$ is on the left (resp. right) side of $C$ and incident to a
white vertex, then it has 
contribution $2\delta$ to $w_L^X(C')$ (resp. to $w_R^X(C')$). 
From what precedes we conclude that $\gamma^X(C')=4\gamma^T(C)$,
and in particular $\gamma^X(C')=0$ if and only if $\gamma^T(C)=0$. 

From there, very similarly to the end of the proof of
Lemma~\ref{lem:hbaltbal}, we conclude that $X$ is balanced if and
only if $T$ is balanced, which concludes the proof.
\end{proof}

We are now able to prove Proposition~\ref{prop:unique_d1}:

\noindent\emph{Proof of Proposition~\ref{prop:unique_d1}.}
Suppose that $M$ admits a $\frac{1}{-1}$-$\zZ$-orientation $Z$ in
$\cO_{1}^1$ whose associated mobile by $\Phi_+$ is in
${\mathcal V}_1^{Bal}$.  By Lemma~\ref{lem:hbaltbald1}, we have
$X=\tau(Z)$ is a $\hgamma$-balanced $2$-regular orientation of
$\aM_4$.  Since $X$ is in $\cO_{4}^{1}$, by Lemma~\ref{lem:necmin}, we
have that $X$ is minimal. By  Lemma~\ref{lem:exists_bregular}, we have 
$M_4\in\hat{\cM}_{4}$.
Then, by Lemma~\ref{lem:bregCanonical}, we
have $M_4\in\chL_{4}$, hence  $M\in\cL_{1}$. In addition we have uniqueness of the
orientation $Z$, since $Z$ has to be preimage under the 
injective mapping $\tau$ of the unique minimal $\hgamma$-balanced $2$-regular orientation of $M_4^\star$.

Conversely we prove the existence part, for $M\in\cL_{1}$. Then $M_4\in\chL_{4}$ and
by Lemma~\ref{lem:bregCanonical}, $\aM_4$ admits a transferable
$\hgamma$-balanced $2$-regular orientation $X$ in $\cO_{4}^{1}$.
Consider $Y_4$ the  $\frac{2}{1}$-$\ZZ$-orientation of
$M_4$ in $\cO_{4}^{1}$ such that $Y=\sigma(X)$.

Consider $Z'$ the $\frac{2}{0}$-$\ZZ$-orientation of $M_2$ in
$\cO_{2}^{1}$ such that $Z'=\iota^{-1}(Y)=\iota^{-1}(\sigma(X))$.  By
Lemma~\ref{lem:hbaltbald}, the mobile $T'$ associated to $Z'$ is in
${\mathcal V}_2^{Bal}$.  Since $M\in\cL_{1}$, we have
$M_2\in\chL_{2}$, hence all the
weights of $T'$ are even according to Lemma~\ref{lem:bip_2}. 
So all the weights of $Z'$ are even.  Let
$Y'$ be the $\frac{1}{0}$-$\ZZ$-orientation of $M_2$ in $\cO_{2}^{1}$
obtained by dividing all the weights of $Z'$ by two.
Consider $Z$ the
$\frac{1}{-1}$-$\ZZ$-orientation of $M$ in $\cO_{1}^{1}$ such that
$Z=\iota^{-1}(Y')$.  Note that $X=\tau(Z)$. Let $T\in{\mathcal V}_1$ be the mobile
associated to $Z$. Lemma~\ref{lem:hbaltbald1} then ensures that $T$ is in ${\mathcal V}_1^{Bal}$.
\hfill$\Box$\vspace{1em}
 
\vspace{.4cm}

\noindent {\bf Acknowledgments.} The authors thank
Olivier Bernardi and the members of the ANR project ``Gato"
for very interesting and helpful discussions; they also thank the two anonymous referees for their insightful reports. 

\bibliographystyle{plain}
\bibliography{bibli}

\end{document}